\documentclass[a4paper,10pt]{amsart}
\usepackage[utf8]{inputenc}
 \usepackage{setspace}
\usepackage[cmtip,arrow]{xy}
\usepackage{lscape}
\usepackage{latexsym}
\usepackage[activeacute,english]{babel}
\usepackage{graphicx}
\usepackage{amsmath}
\usepackage{amsfonts}

\usepackage{pst-node}

\usepackage{amsthm}
\usepackage{amssymb}
\usepackage{amscd}
\usepackage{mathrsfs}
\usepackage{tikz}
\usepackage{pb-diagram,pb-xy}
\xyoption{all}
\usepackage{graphics}
\usepackage{stmaryrd}
\usepackage{yfonts,bbm}

\newtheorem{theorem}{Theorem}[section]
\newtheorem{lemma}[theorem]{Lemma}
\newtheorem{proposition}[theorem]{Proposition}
\newtheorem{corollary}[theorem]{Corollary}
\newtheorem{definition}[theorem]{Definition}

\newtheorem*{note}{Note}

\newtheorem{Remark}[theorem]{Remark}

\newcommand{\C}{\mathcal{C}}

\newcommand{\Mod}{\mathrm{Mod}}

\begin{document}
\title{Triangular Matrix Categories I: Dualizing Varieties and generalized one-point extension}

\author{Alicia Le\'on-Galeana, Mart\'in Ortiz-Morales, Valente  Santiago Vargas}
\thanks{The authors thank  the Project PAPIIT-Universidad Nacional Aut\'onoma de M\'exico IA105317. }
\date{March 3, 2019}
\dedicatory{}
\maketitle

\begin{abstract}
Following Mitchell's philosophy, in this paper we define the analogous of the triangular matrix algebra to the context of rings with several objects. Given two additive categories $\mathcal{U}$ and $\mathcal{T}$ and $M\in \mathsf{Mod}(\mathcal{U}\otimes \mathcal{T}^{op})$ we construct the triangular matrix category $\mathbf{\Lambda}:=\left[\begin{smallmatrix}
\mathcal{T} & 0 \\ 
M & \mathcal{U}
\end{smallmatrix}\right]$. First, we prove that there is an equivalence 
$\Big( \mathsf{Mod}(\mathcal{T}), \mathbb{G}\mathsf{Mod}(\mathcal{U})\Big) \simeq \mathrm{Mod}(\mathbf{\Lambda})$. One of our main results is that if $\mathcal{U}$ and $\mathcal{T}$ are dualizing $K$-varieties and $M\in \mathsf{Mod}(\mathcal{U}\otimes \mathcal{T}^{op})$ satisfies certain conditions then $\mathbf{\Lambda}:=\left[\begin{smallmatrix}
\mathcal{T} & 0 \\ 
M & \mathcal{U}
\end{smallmatrix}\right]$ is a dualizing variety (see theorem \ref{matricesdualizan}). In particular, $\mathrm{mod}(\mathbf{\Lambda})$ has Auslander-Reiten sequences. Finally, we apply the theory developed in this paper to quivers and give a  generalization of the so called one-point extension algebra.
\end{abstract}

\section{introduction}

 



The idea that additive categories are rings with several objects was developed convincingly by Barry Mitchell (see \cite{Mitchelring})  who showed that a substantial amount of noncommutative ring theory is still true in this generality. Here we would like to emphazise that sometimes  clarity in concepts, statements, and proofs are  gained by dealing with additive categories, and that familiar theorems for rings come out of the natural development of category theory. For instance, the notions of radical of an additive category,  perfect  and semisimple rings, global dimensions etc, have been amply  studied in the context of rings with several objects (see \cite{Rune}, \cite{HaradaI},  \cite{HaradaII},  \cite{HaradaIII},  \cite{HaradaIV}, \cite{Hart}, \cite{Holm}, \cite{Hu},  \cite{Kelly},  \cite{Oberst}, \cite{Schroer}, \cite{Simson}, \cite{Bo}, \cite{Stovi}, \cite{Roos}).\\
As an example of the power of this point of view is the approach that M. Auslander and I. Reiten gave to the study of representation theory (see for example \cite{AusM1}, \cite{AusM2}, \cite{AusVarietyI}, \cite{AusVarietyII},\cite{AusVarietyIII}, \cite{AusVarietyIV}, \cite{AusVarietyV}, \cite{AusM3}, \cite{AusM4}, \cite{AusM5}, \cite{AusM6}, \cite{AusFunapproac}), which gave birth to the concept of almost split sequence. There were two different approaches to the existence of almost split sequences. One was inspired by \cite{AusCoheren} and focused on showing that simple functors are finitely presented. An essential ingredient in this proof as to establish a duality between finitely presented contravariant and finitely presented covariant functors. This led to the notion of dualizing $R$-varieties, introduced and investigated in \cite{AusVarietyI}. Therefore the existence of almost split sequences is proved in the context of dualizing $R$-varieties. 
Dualizing $R$-varieties have appeared  in the context of the locally bounded $k$-categories over a field $k$,  categories of graded modules over artin algebras and also in connection with covering theory. M. Auslander and I. Reiten continued a systematic study of  $R$-dualizing varieties in \cite{AusVarietyII}, \cite{AusVarietyIII}, \cite{AusVarietyIV}.  
One of the advantage of notion of a dualizing $R$-variety defined in  \cite{AusVarietyI} is that it provides a common setting for the category $\mathrm{proj}(A)$ of finitely generated projective $A$-modules, $\mathrm{mod}(A)$ and $\mathrm{mod}(\mathrm{mod}(A))$, which all play an important role in the study of an artin algebra $A$. \\
On the other hand, rings of the form $\left[\begin{smallmatrix}
T & 0 \\ 
M & U
\end{smallmatrix}\right]$ where $T$ and $U$ are rings and $M$ is a $T$-$U$-bimodule have appeared often in the study of the representation theory of artin rings and algebras (see for example \cite{AusPlatRei},\cite{DlabRingel}, \cite{Gordon}, \cite{Green1}, \cite{Green2}). Such a rings, called {\it{triangular matrix rings}}, appear naturally in the study of homomorphic images of hereditary  artin algebras and in the study of the decomposition of algebras and direct sum of two rings. Triangular matrix rings and their homological properties have been widely studied,
(see for example \cite{Chase}, \cite{Fields}, \cite{Haghany}, \cite{Michelena},\cite{Palmer}, \cite{tameskow}, \cite{Hong}). 
The so-called one-point extension is a special case of the triangular matrix algebra and this types of algebras have been studied in several contexts. 
For example, in \cite{Happel}, D. Happel showed how to compute the Coxeter polynomial for $\Gamma$ from the Coxeter polynomial of $\Lambda$ and homological invariantes of $M$ for a given one-point extension algebra $\Gamma=\Lambda[M]$.  In \cite{Bin}, Zhu 
consider the triangular matrix algebra  $\Lambda:=\left[\begin{smallmatrix}
T & 0 \\ 
M & U
\end{smallmatrix}\right]$ where $T$ and $U$ are quasi-hereditary algebras and he proved that under suitable conditions on $M$, $\Lambda$ is quasi-hereditary algebra.\\
It is well known that the category $\mathrm{Mod}\Big(\left[\begin{smallmatrix}
T & 0 \\ 
M & U
\end{smallmatrix}\right]\Big)$ is equivalent to the category whose objects are triples $(A,B,f)$ where $f:M\otimes_{T}A\longrightarrow B$ is a morphism of $U$-modules and the morphisms are pairs $(\alpha,\beta)$ with $\alpha:A\longrightarrow A'$ and $\beta:B\longrightarrow B'$ such that the following diagram commutes
$$\xymatrix{M\otimes_{T}A\ar[r]^{1_{M}\otimes\alpha}\ar[d]^{f} & 
M\otimes_{T}A'\ar[d]^{f'}\\
B\ar[r]^{\beta} & B'.}$$
In particular,  given a ring $T$ we can consider $\Lambda:=\left[\begin{smallmatrix}
T & 0 \\ 
T & T
\end{smallmatrix}\right]$. Then the category of $\Lambda$-modules is equivalent to the category  whose objects are triples $(A,B,f)$ where $f:A\longrightarrow B$ is a morphism of $T$-modules.\\
In this context, in \cite{RingScmid} Ringel and Schmidmeier studied the category of monomorphisms and epimorphisms and they proved that if $\Gamma$ is an artin algebra, the category of all the embeddings ($A\subseteq B$) where $B$ is a finitely generated  $\Gamma$-module and $A$ is a submodule of $B$, is a Krull-Schmidt category which has Auslander-Reiten sequences. Also in this direction, R.M. Villa and M. Ortiz studied the Auslander-Reiten sequences in the category  of maps and also studied some contravariantly finite subcategories (see \cite{MVOM}).\\
Following Mitchell's philosophy, in this paper we define the analogous of the triangular matrix algebra to the context of rings with several objects. Given two additive categories $\mathcal{U}$ and $\mathcal{T}$ and $M\in \mathsf{Mod}(\mathcal{U}\otimes \mathcal{T}^{op})$ we construct the triangular matrix category $\mathbf{\Lambda}:=\left[\begin{smallmatrix}
\mathcal{T} & 0 \\ 
M & \mathcal{U}
\end{smallmatrix}\right]$  and several properties of $\mathrm{Mod}(\mathbf{\Lambda})$ are studied and we proved that under certain conditions  $\mathbf{\Lambda}:=\left[\begin{smallmatrix}
\mathcal{T} & 0 \\ 
M & \mathcal{U}
\end{smallmatrix}\right]$ is a dualizing variety, this result is the analogous of the following one: $\Lambda=\left[\begin{smallmatrix}
T & 0 \\ 
M & U
\end{smallmatrix}\right]$ is an artin algebra if and only if there is a commutative ring $R$ such that $T$ and $U$ are artin $R$-algebras and $M$ is finitely generated over $R$ which acts centrally on $M$ (see \cite[Theorem 2.1]{AusBook} in page 72). We give some applications to path categories given by infinite quivers,
other applications (as the construction of recollements and the study of functorially finite subcategories of $\mathrm{mod}(\mathbf{\Lambda})$) of the theory developed in this paper will be in a forthcoming paper.\\
We now give a brief description of the contents on this paper.\\
In section 2, we recall basic results of $\mathrm{Mod}(\mathcal{C})$ that will be used throughout this paper.\\
In section 3,  we construct the category of matrices $\mathbf{\Lambda}:=\left[\begin{smallmatrix}
\mathcal{T} & 0 \\ 
M & \mathcal{U}
\end{smallmatrix}\right]$ and we proved that there is an equivalence between the comma category $\big( \mathsf{Mod}(\mathcal{T}), \mathbb{G}\mathsf{Mod}(\mathcal{U})\big)$ and $\mathrm{Mod}(\mathbf{\Lambda})$ and we compute $\mathrm{rad}_{\mathbf{\Lambda}}\left (\left[ \begin{smallmatrix}
T & 0 \\
M & U
\end{smallmatrix} \right] ,  \left[ \begin{smallmatrix}
T' & 0 \\
M & U'
\end{smallmatrix} \right]  \right)$ in terms of the radical of $\mathcal{U}$ and $\mathcal{T}$  (see \ref{radtriangular}).\\
In section 4, we consider $K$-categories and define a functor $\mathbb{D}_{\mathbf{\Lambda}}:\mathrm{Mod}(\mathbf{\Lambda})\longrightarrow \mathrm{Mod}(\mathbf{\Lambda}^{op})$ and we describe it how it acts, when we identify $\big( \mathsf{Mod}(\mathcal{T}), \mathbb{G}\mathsf{Mod}(\mathcal{U})\big)$ with $\mathrm{Mod}(\mathbf{\Lambda})$  (see theorem 	\ref{cuadrocasicon}).\\
In section 5, we show that there exists an adjoint pair $(\mathbb{F},\mathbb{G})$ and we describe the finitely generated projectives in $\big( \mathsf{Mod}(\mathcal{T}), \mathbb{G}\mathsf{Mod}(\mathcal{U})\big)$ (see \ref{projecinmaps}). We also prove that
there exists an isomorphism between the comma categories
$\big( \mathbb{F}(\mathsf{Mod}(\mathcal{T})),\mathsf{Mod}(\mathcal{U})\big) \simeq\big( \mathsf{Mod}(\mathcal{T}), \mathbb{G}(\mathsf{Mod}(\mathcal{U}))\big),$ (see \ref{equicoma2}) and in section 6 we restrict that isomorphism to the category of finitely presented modules.\\
In section 6, we prove that if  $\mathcal{T}$ and $\mathcal{U}$ are categories with splitting idempotents, then $\mathbf{\Lambda}$ is with splitting idempotents (see \ref{spliidempomat}). We consider the case in which $\mathcal{U}$ and $\mathcal{T}$ are Hom-finite Krull-Schmidt $K$-categories and we show that under this conditions $\mathbf{\Lambda}$ is Hom-finite and Krull-Schmidt (see \ref{homfinikrulsc}). Finally, we prove that if $\mathcal{U}$ and $\mathcal{T}$ are dualizing $K$-varieties and $M\in \mathsf{Mod}(\mathcal{U}\otimes \mathcal{T}^{op})$ satisfies  that $M_{T}\in \mathsf{mod}(\mathcal{U})$ and  $M_{U}\in \mathsf{mod}(\mathcal{T}^{op})$ for all $T\in \mathcal{T}$ and $U\in \mathcal{U}^{op}$ then $\mathbf{\Lambda}:=\left[\begin{smallmatrix}
\mathcal{T} & 0 \\ 
M & \mathcal{U}
\end{smallmatrix}\right]$ is a dualizing variety (see \ref{matricesdualizan}). In particular, $\mathrm{mod}(\mathbf{\Lambda})$ has Auslander-Reiten sequences (see \ref{modtieneseq}).\\
In section 7, we make some applications to splitting torsion pairs which are in relation with tilting theory and path categories which are studied in \cite{RingelTame}. In particular, we prove that given a splitting torsion pair  $(\mathcal{U}, \mathcal{T})$  in a Krull-Schmidt category $\mathcal{C}$ we have an equivalence
$\mathcal{C} \cong \left[ \begin{matrix}
\mathcal{T} & 0 \\
\widehat{\mathbbm{Hom}} & \mathcal{U}
\end{matrix} \right]$ and this result is kind of generalization of the one-point extension algebra.


\section{Preliminaries}
We recall that a category $\C$ together with an abelian group structure on each of the sets of morphisms $\C(C_{1},C_{2})$ is called  \textbf{preadditive category}  provided all the composition maps
$\C(C,C')\times \C(C',C'')\longrightarrow \C(C,C'')$
in $ \C $ are bilinear maps of abelian groups. A covariant functor $ F:\C_{1}\longrightarrow \C_{2} $ between  preadditive categories $ \C_{1} $ and $ \C_{2} $ is said to be \textbf{additive} if for each pair of objects $ C $ and $ C' $ in $ \C_{1}$, the map $ F:\C_{1}(C,C')\longrightarrow \C_{2}(F(C),F(C')) $ is a morphism of abelian groups. Let $\mathcal C$ and $\mathcal D$  be preadditive categories and $\mathbf{Ab}$ the category of abelian groups. A functor $F: \mathcal C\times \mathcal D\rightarrow \mathbf{Ab}$ is called \textbf{biadditive} if
$F:\mathcal C(C,C')\times \mathcal D(D,D')\rightarrow \mathbf{Ab} (F(C,D),F(C',D'))$ is bi additive, that is,  $F(f+f',g)=F(f,g)+F(f',g)$ and $F(f,g+g')=F(f,g)+F(f,g')$.\\
If $ \C $ is a  preadditive category we always considerer its opposite  category $ \C^{op}$ as a preadditive category by letting $\C^{op} (C',C)= \C(C,C') $. We follow the usual convention of identifying each contravariant  functor $F$ from a category $ \C $ to $ \mathcal{D} $ with the covariant functor $F$ from  $ \C^{op} $ to $ \mathcal{D}$.

\subsection{The category $\mathrm{Mod}(\mathcal{C})$}
Throughout this section $\mathcal{C}$ will be an arbitrary skeletally small preadditive category, and $\mathrm{Mod}(\mathcal{C})$ will denote the \textit{category of covariant functors} from $\mathcal{C}$ to  the category of abelian groups $ \mathbf{Ab}$, called the category of $\mathcal{C}$-modules. This category has as objects  the functors from $\mathcal C$ to $\mathbf{Ab}$, and  and a morphism $ f:M_{1}\longrightarrow M_{2} $ of $ \C $-modules is a natural transformation,  that is, the set of morphisms $\mathrm{Hom}_\mathcal C(M_1,M_2)$ from $M_1$ to $M_2$  is given by $\mathrm{Nat} (M_{1}, M_{2} )$.  We sometimes we will write for short, $\mathcal{C}(-,?)$
instead of $\mathrm{Hom}_{\mathcal{C}}(-,?)$ and when it is clear from the context we will use just $(-,?).$\\
We now recall some of properties of the category $ \Mod(\C) $, for more details consult  \cite{AusM1}. The category $\Mod(\C) $ is an abelian with the following properties:
\begin{enumerate}
\item A sequence
\[
\begin{diagram}
\node{M_{1}}\arrow{e,t}{f}
 \node{M_{2}}\arrow{e,t}{g}
  \node{M_{3}}
\end{diagram}
\]
is exact in $ \Mod(\C) $ if and only if
\[
\begin{diagram}
\node{M_{1}(C)}\arrow{e,t}{f_{C}}
 \node{M_{2}(C)}\arrow{e,t}{g_{C}}
  \node{M_{3}(C)}
\end{diagram}
\]
is an exact sequence of abelian groups for each $ C$ in $\C $.

\item Let $ \lbrace M_{i}\rbrace_{i\in I} $ be a family of $ \C $-modules indexed by the set $ I $.
The $ \C $-module $ \underset{\i\in I}\amalg M_{i}$ defined by $  (\underset{i\in I}\amalg M_{i})\ (C)=\underset{i\in I}\amalg \ M_{i}(C)$ for all $ C $ in $ \C $, is a direct sum for the family $ \lbrace M_{i}\rbrace_{i\in I} $ in $ \Mod(\C) $, where $\underset{i\in I} \amalg M_{i}(C)  $ is the direct sum in $ \mathbf{Ab} $ of the family of abelian groups $ \lbrace M_{i}(C)\rbrace_{i\in I} $.  The $ \C $-module $ \underset{\i\in I}\prod M_{i}$ defined by $(\underset{i\in I}\prod M_{i})\ (C)=\underset{i\in I}\prod M_{i}(C)   $ for all $C$ in $\C$, is a product for the family $ \lbrace M_{i}\rbrace_{i\in I} $ in $\Mod(\C)$, where $ \underset{i\in I}\prod M_{i}(C)  $ is the product in $\mathbf{Ab}$.

\item  For each $C$ in $\C $, the $\C$-module $(C,-)$ given by $(C,-)(X)=\C(C,X)$ for each $X$ in $\C$, has the property that for each $\C$-module $M$, the map $\left( (C,-),M\right)\longrightarrow M(C)$ given by $f\mapsto f_{C}(1_{C})$ for each $\C$-morphism $f:(C,-)\longrightarrow M$ is an isomorphism of abelian groups. We will often consider this isomorphism an identification.
Hence
\begin{enumerate}
\item The functor $ P:\C\longrightarrow \Mod(\C) $ given by $ P(C)=(C,-) $ is fully faithful.
\item For each family $\lbrace  C_{i}\rbrace _{i\in I}$ of objects in $ \C $, the $ \C $-module $ \underset{i\in I}\amalg P(C_{i}) $ is a projective $ \C $-module.
\item Given a $ \C $-module $ M $, there is a family $ \lbrace C_{i}\rbrace_{i\in I} $ of objects in $ \C $ such that there is an epimorphism $ \underset{i\in I}\amalg P(C_{i})\longrightarrow M\longrightarrow 0 $.
\end{enumerate}
\end{enumerate}

\subsection{Dualizing varietes and Krull-Schmidt Categories}

The subcategory of $\mathrm{Mod}(\mathcal{C})$ consisting of
all finitely generated projective objects, $\mathfrak{p}(\mathcal{C})$, is a skeletally small additive category in which idempotents split, the functor $P:\mathcal{C}\rightarrow \mathfrak{p}(\mathcal{C})$, $P(C)=\mathcal{C}(C,-)$, is fully faithful and induces by restriction $\mathrm{res}:\mathrm{Mod}(\mathfrak{p}(\mathcal{C}))\rightarrow \mathrm{Mod}(\mathcal{C})$, an equivalence of categories. For this reason, we may assume that our categories are skeletally small, additive categories, such that idempotents split. Such categories were called \textbf{annuli varieties} in \cite{AusVarietyI}, for
short, \textbf{varieties}.\\
To fix the notation, we recall known results on functors and categories that we use through the paper, referring for the proofs to the papers by Auslander and Reiten \cite{AusQueen}, \cite{AusM1}, \cite{AusVarietyI}.

\begin{definition}
Let  $\mathcal{C}$ be a variety. We say $\mathcal{C}$ has \textbf{pseudokernels}; if given a map $f:C_1\rightarrow C_0$, there exists a map $g:C_2 \rightarrow C_1$ such that the sequence of morphisms  $\mathcal{C}(-, C_2 )\xrightarrow{(-,g)}\mathcal{C}( -,C_1 )\xrightarrow{(-,f)}\mathcal{C}(-, C_0 )$ is exact in $\mathrm{Mod}(\mathcal{C}^{op})$.
\end{definition}
Now, we recall some results from \cite{AusVarietyI}.


\begin{definition}
Let $R$ be a commutative artin ring. An $R$-variety $\mathcal{C}$, is a variety such that $\mathcal{C}(C_{1},C_{2})$ is an $R$-module, and the composition is $R$-bilinear. An $R$-variety $\mathcal{C}$ is $\mathbf{Hom}$-\textbf{finite}, if for each pair of objects $C_{1},C_{2}$ in $\mathcal{C},$ the $R$-module $\mathcal{C}(C_{1},C_{2})$ is finitely generated. We denote by $(\mathcal{C},\mathrm{mod}(R))$, the full subcategory of $(\mathcal{C},\mathrm{
\mathrm{Mod}}(R))$ consisting of the $\mathcal{C}$-modules such that; for
every $C$ in $\mathcal{C}$ the $R$-module $M(C)$ is finitely generated. 
\end{definition}

Suppose $\mathcal{C}$ is a Hom-finite $R$-variety. If $M:\mathcal{C}\longrightarrow \mathbf{Ab}$ is a $\mathcal{C}$-module, then for each $C\in \mathcal{C}$ the abelian group $M(C)$ has a structure of $\mathrm{End}_{\mathcal{C}}(C)^{op}$-module and hence as an $R$-module since $\mathrm{End}_{\mathcal{C}}(C)$ is an $R$-algebra. Further if $f:M\longrightarrow M'$ is a morphism of $\mathcal{C}$-modules it is easy to show that $f_{C}:M(C)\longrightarrow M'(C)$ is a morphism of $R$-modules for each $C\in \mathcal{C}$. Then, $\mathrm{\mathrm{Mod}}(\mathcal{C})$ is an $R$-variety, which we identify with the category of
covariant functors $(\mathcal{C},\mathrm{Mod}(R))$. Moreover, the
category $(\mathcal{C},\mathrm{mod}(R))$ is abelian and the inclusion $(\mathcal{C},\mathrm{mod}(R))\rightarrow (\mathcal{C},\mathrm{\mathrm{Mod}}(R))$ is exact.

\begin{definition}
Let $\mathcal{C}$ be a Hom-finite $R$-variety. We denote by $\mathrm{mod}(\mathcal{C})$ the full subcategory of $\mathrm{Mod}(\mathcal{C})$ whose objects are the  $\textbf{finitely presented functors}$.
That is, $M\in \mathrm{mod}(\mathcal{C})$ if and only if, there exists an exact sequence in $\mathsf{Mod}(\mathcal{C})$
$$\xymatrix{P_{1}\ar[r] & P_{0}\ar[r] & M\ar[r] & 0,}$$
where $P_{1}$ and $P_{0}$ are finitely generated projective $\mathcal{C}$-modules. 
\end{definition}
It is easy to see that if $\mathcal{C}$ has finite coproducts, then
a functor $M$ is finitely presented if there exists an exact
sequence
\begin{equation*}
\mathcal{C}( -,C_1 )\rightarrow \mathcal{C}(-, C_0 )\rightarrow M\rightarrow 0
\end{equation*}
It was proved in \cite{AusVarietyI},  that $\mathrm{mod}(C)$ is abelian if and only if $\mathcal{C}$ has pseudokernels.\\

Consider the functors $D:(\mathcal{C}^{op},\mathrm{mod}
(R))\rightarrow (\mathcal{C},\mathrm{mod}(R))$, and $D:(\mathcal{C},\mathrm{mod}(R))\rightarrow (\mathcal{C}^{op},\mathrm{mod}(R))$, which are defined as
follows: for any object $C$ in $\mathcal{C}$, $D(M)(C)=\mathrm{Hom}
_{R}(M(C),I(R/r)) $, with $r$ the Jacobson radical of $R$, and $I(R/r)$ is
the injective envelope of $R/r$. The functor $D$ defines a duality between $(
\mathcal{C},\mathrm{mod}(R))$ and $(\mathcal{C}^{op},\mathrm{mod}(R))$. We know that  since $\mathcal{C}$ is Hom-finite, $\mathrm{mod}(\mathcal{C})$ is a subcategory of $(\mathcal{C},\mathrm{mod}(R))$. Then we have the following definition due to Auslander and Reiten (see \cite{AusVarietyI}.).

\begin{definition}\label{dualizinvar}
An $\mathrm{Hom}$-finite $R$-variety $\mathcal{C}$ is \textbf{dualizing}, if
the functor
\begin{equation}\label{duality}
D:(\mathcal{C}^{op},\mathrm{mod}(R))\rightarrow (\mathcal{C},\mathrm{mod}(R))
\end{equation}
induces a duality between the categories $\mathrm{mod}(\mathcal{C})$ and $%
\mathrm{mod}(\mathcal{C}^{op}).$
\end{definition}

It is clear from the definition that for dualizing categories $\mathcal{C}$
the category $\mathrm{mod}(\mathcal{C})$ has enough injectives. To finish, we recall the following definition:

\begin{definition}
An additive category $\mathcal{C}$ is \textbf{Krull-Schmidt}, if every
object in $\mathcal{C}$ decomposes in a finite sum of objects whose
endomorphism ring is local.
\end{definition}

Asumme that $R$ is a commutative ring and $R$ is a dualizing $R$-variety.
Since the endomorphism ring of each object in $\mathcal C$ is an artin algebra, it follows that $\mathcal C$ is a 
Krull-Schmidt category \cite[p.337]{AusVarietyI} moreover,  we have that for a dualizing  variety the
finitely presented functors have projective covers \cite[Cor. 4.13]{AusM1}, \cite[Cor. 4.4]{Krause}.

\subsection{Tensor Product of Categories}
If $\mathcal{C}$ and $\mathcal{D}$ are additive categories, B. Mitchell defined in \cite{Mitchelring} the  $\textbf{tensor product}$  $\mathcal{C}\otimes\mathcal{D} $  of two additive categories, whose objects are those of $\mathcal{C}\times \mathcal{D}$ and the abelian group of morphism from $(C,D)$ to $(C',D')$ is the ordinary tensor product of abelian groups $\mathcal{C}(C,C')\otimes_{\mathbb{Z}}\mathcal{D}(D,D')$. Since that the tensor product of abelian groups is associative and commutative and the composition in $\mathcal{C}$ and $\mathcal{D}$ is $\mathbb{Z}$-bilinear then the bilinear composition in $\mathcal{C}\otimes \mathcal{D}$ is given as follows:
\begin{equation*}
  (f_{2}\otimes g_{2})\circ (f_{1}\otimes g_{1}):=(f_{2}\circ f_{1})\otimes(g_{2}\circ g_{1})
\end{equation*}
for all $f_{1}\otimes g_1\in \mathcal{C}(C,C')\otimes \mathcal{D}(D,D')$ and  $f_{2}\otimes g_2\in\mathcal{C}(C',C'')\otimes \mathcal{D}(D',D'')$.

\begin{Remark}\label{productensorR}
If $\mathcal{C}$ and $\mathcal{D}$ are $R$-categories. The tensor product $\mathcal{C}\otimes_{R}\mathcal{D} $  of two $R$-categories, is the $R$-category whose objects are those of $\mathcal{C}\times \mathcal{D}$ and the abelian group of morphism from $(C,D)$ to $(C',D')$ is the ordinary tensor product of $R$-modules $\mathcal{C}(C,C')\otimes_{R}\mathcal{D}(D,D')$ and the composition is defined as above.
\end{Remark}

Assume that $\mathcal{C}$ and $\mathcal{D}$ are additive categories. Then, there exists a canonical functor
$T:\mathcal C\times \mathcal D\rightarrow \mathcal C\otimes \mathcal D $, given by $T((C,D))=(C,D)$, and $T:\mathcal C(C,C')\times \mathcal D(D,D')\rightarrow \mathcal C(C,C')\otimes \mathcal D(D,D')$, $(f,g)\mapsto f\otimes g$.

\begin{proposition}\label{tensorr}
Let $F:\mathcal C\times \mathcal D\rightarrow \mathbf{Ab}$ a biadditive bifunctor. Then there exists a functor $\widehat F:\mathcal C\otimes \mathcal D \rightarrow \mathbf{Ab}$
such that  $F=\widehat{F} T$.
\end{proposition}
\begin{proof}
The proof is an easy consequence of  the universal property which characterizes the tensor product of abelian groups.
\end{proof}
Let  $\widehat F:\mathcal C\otimes \mathcal D \rightarrow \mathbf{Ab}$. For all $X\in \mathcal C$ we have a functor 
$\widehat F_X:\mathcal D \rightarrow \mathbf{Ab}$  associated to $\widehat F$, given by $\widehat F_X(Y)= \widehat F(X,Y)$ and 
$\widehat F_X(g)=\widehat F(1_X\otimes g)$ for all $Y\in \mathcal D$ and $g\in \mathcal D(Y,Y')$. Let $\mathcal C$ be a additive category. Then, for each  functor $F(-,-):\mathcal C^{op}\times\mathcal C\rightarrow \mathbf{Ab}$, we have  a functor
$\mathbbm F(-,-):  \mathcal C\times \mathcal C^{op}\rightarrow \mathbf{Ab}$  defined by  $\mathbbm F(X,Y):=F(Y,X)$ and $\mathbbm F(f,g):=F(g,f)$
for all $(X,Y)\in  \mathcal C\times \mathcal C^{op}$ and $(f,g)\in \mathcal C(X,X')\times \mathcal C^{op}(Y,Y')$.\\
Let $Y\in \mathcal C^{op}$, then it is clear that  $\mathbbm F(-,Y)=F(Y,-)$  as $\mathcal C$-modules. By using Proposition \ref{tensorr}, we summarize  the above  observations in the next proposition.

\begin{proposition}\label{funhomext}
\begin{itemize}
\item[(i)] Let $\mathcal{C}$ be a preadditive category and $F:\mathcal C^{op}\times \mathcal C\rightarrow \mathbf{Ab}$ a biadditive functor. Then we have a  functor
$$\widehat{\mathbbm F}:\mathcal C\otimes \mathcal C^{op}\rightarrow \mathbf{Ab}$$
for which  $\widehat{\mathbbm F}(X,Y)=F(Y,X)$,  for all $(X,Y)\in C \otimes \mathcal C ^{op}$, and the $\mathcal C$-modules $\mathbbm F(-,Y)$, $F(Y,-)$ and  
$\widehat{\mathbbm F}_Y$ are isomorphic, for all $Y\in\mathcal C^{op}$.

\item[(ii)] Let $\mathcal C$ be a preadditive category and consider the  bifunctors
$$\mathrm{Hom}_{\mathcal C}(-,-):\mathcal C^{op}\times \mathcal C\rightarrow \mathbf{Ab},\quad \mathrm{Ext}_{\mathcal C}^n(-,-):\mathcal C^{op}\times \mathcal C\rightarrow \mathbf{Ab},\,\, n >1.$$
Then, there exist   functors
$$\widehat{\mathbbm{{Hom}}}:\mathcal C\otimes \mathcal C^{op}\rightarrow \mathbf{Ab},\quad \widehat{{\mathbbm{Ext}}^n}:\mathcal C\otimes \mathcal C^{op}\rightarrow \mathbf{Ab},\,\, n >1.$$
for which  $\widehat{\mathbbm{Hom}}(X,Y)=\mathrm{Hom}_{\mathcal C}(Y,X)$,   $\widehat{\mathbbm{Ext}^n}(X,Y)=\mathrm{Ext}_{\mathcal C}^n(Y,X)$ for all $(X,Y)\in C \otimes \mathcal C ^{op}$, and we have isomorphisms of $\mathcal C$-modules 
$$\widehat{\mathbbm{Hom}}_Y\cong \mathrm{Hom}_{\mathcal C}(Y,-):\mathcal C\rightarrow\mathbf{Ab}\quad\text{and}\quad \widehat{\mathbbm{Ext}^n}_Y\cong \mathrm{Ext}_{\mathcal C}^n(Y,-):\mathcal C\rightarrow\mathbf{Ab},$$ 
 for all $Y\in\mathcal C^{op}$.
\end{itemize}
\end{proposition}

\subsection{Quotient and comma category and radical of a category}

A  $\textbf{two}$ $\textbf{sided}$ $\textbf{ideal}$  $I(-,?)$ is an additive subfunctor of the two variable functor $\mathcal{C}(-,?):\mathcal{C}^{op}\otimes\mathcal{C}\rightarrow\mathbf{Ab}$ such that: (a) if $f\in I(X,Y)$ and $g\in\mathcal C(Y,Z)$, then  $gf\in I(X,Z)$; and (b)
if $f\in I(X,Y)$ and $h\in\mathcal C(U,X)$, then  $fh\in I(U,Z)$. If $I$ is a two-sided ideal, then we can form the $\textbf{quotient category}$  $\mathcal{C}/I$ whose objects are those of $\mathcal{C}$, and where $(\mathcal{C}/I)(X,Y):=\mathcal{C}(X,Y)/I(X,Y)$. Finally the composition is induced by that of $\mathcal{C}$ (see \cite{Mitchelring}). There is a canonical projection functor $\pi:\mathcal{C}\rightarrow \mathcal{C}/I$ such that:
\begin{enumerate}

\item $\pi(X)=X$, for all $X\in \mathcal{C}$.

\item For all $f\in \mathcal{C}(X,Y)$, $\pi(f)=f+I(X,Y):=\bar{f}.$
\end{enumerate}

Based on the Jacobson radical of a ring, we introduce the radical of an additive category. This concept goes back to work of Kelly (see \cite{Kelly}).

\begin{definition}
The (Jacobson) $\textbf{radical}$ of an additive category $\mathcal{C}$ is the two-sided ideal $\mathrm{rad}_{\mathcal C}$ in $\mathcal{C}$ defined by the formula
$$\mathrm{rad}_{\mathcal {C}}(X,Y)=\{h\in\mathcal{C}(X,Y)\mid 1_X-gh\text{ is invertible for any } g\in\mathcal{C}(Y,X)\} $$
for all objects $X$ and $Y$ of $\mathcal{C}$.
\end{definition}

If $\mathcal{A}$ and $\mathcal{B}$ are abelian categories and $F:\mathcal{A}\rightarrow\mathcal{B}$ is an additive functor. The $\textbf{comma category}$ $(\mathcal{B},F \mathcal{A})$ is the category whose objects are triples $(B,f,A)$ where $f:B\rightarrow FA$; and whose morphisms between the objects $(B,f,A)$ and $(B',f',A')$ are pair $(\beta,\alpha)$ of morphisms in $\mathcal{B}\times\mathcal{A}$ such that the diagram

\[
\begin{diagram}
\node{B}\arrow{e,t}{\beta}\arrow{s,l}{f}
 \node{B'}\arrow{s,r}{f'}\\
\node{FA}\arrow{e,t}{F\alpha}
 \node{FA'}
\end{diagram}
\]
is commutative in $\mathcal{B}$ (see \cite{Robert}).

\section{Triangular Matrix Categories}\label{section2}
In all what follows $\mathcal{U}$ and $\mathcal{T}$ will be additive categories.

\begin{proposition}\label{dosfuntores}
Let $ M\in \mathsf{Mod}(\mathcal{U}\otimes \mathcal{T}^{op})$ be. Then, there exists two covariant functors
$$E:\mathcal{T}\longrightarrow \mathsf{Mod}(\mathcal{U})^{op}$$
$$E':\mathcal{U}\longrightarrow  \mathsf{Mod}(\mathcal{T}^{op}).$$
\end{proposition}
\begin{proof}
${}$
\begin{enumerate}
\item [(a)]For $T\in \mathcal{T}$, we define a covariant functor $E(T):=M_{T}:\mathcal{U} \longrightarrow \mathbf{Ab}$  as follows
\begin{enumerate}
\item [(i)] $ M_{T}(U):=M(U,T) $, for all $ U\in \mathcal{U} $.
\item [(ii)] $ M_{T}(u):=M(u\otimes 1_{T}) $, for all $ u\in  \mathsf{Hom}_{\mathcal{U}}(U,U').$
\end{enumerate}
Now, given a morphism $t:T\longrightarrow T'$ in $\mathcal{T}$ we set $E(t):=\bar{t}:M_{T'}\longrightarrow M_{T}$  where $ \bar{t}=\lbrace [\bar{t}]_{U}:M_{T'}(U)\longrightarrow M_{T}(U)\rbrace_{ U\in \mathcal{U}}$ with $ [\bar{t}]_{U}=M(1_{U}\otimes t^{op}):M(U,T')\longrightarrow M(U,T)$.
It is easy to show that $E$ is a contravariant functor $E:\mathcal{T}\longrightarrow \mathsf{Mod}(\mathcal{U})$.

\item [(b)]
Similarly for $ U\in \mathcal{U} $ we define a contravariant the functor $ E'(U):=M_{U}:\mathcal{T}\longrightarrow \mathbf{Ab}$ (or a covariant functor $M_{U}:\mathcal{T}^{op} \longrightarrow \mathbf{Ab}$) as follows:

\begin{enumerate}
\item [(i)] $ M_{U}(T):=M(U,T) $, for all $ T\in \mathcal{T} $.
\item [(ii)] $ M_{U}(t):=M(1_{U}\otimes t^{op}) $, for all $ t\in  \mathsf{Hom}_{\mathcal{T}}(T,T').$ 
\end{enumerate}
Now, given $u\in  \mathsf{Hom}_{\mathcal{U}}(U,U')$ we set $E'(u):=\bar{u}:M_{U}\longrightarrow M_{U'} $
where $\bar{u}=\lbrace [\bar{u}]_{T}:M_{U}(T)\longrightarrow M_{U'}(T)\rbrace_{ T\in \mathcal{T}^{op}} $ (we are seeing $M_{U}:\mathcal{T}^{op} \longrightarrow \mathbf{Ab}$ as a covariant functor ) with $ [\bar{u}]_{T} =M(u\otimes 1_{T}):M(U,T)\longrightarrow M(U',T).$
\end{enumerate}
\end{proof}

\begin{definition}
We define a covariant functor $\mathbb{G}:\mathsf{Mod}(\mathcal{U})\longrightarrow \mathsf{Mod}(\mathcal{T})$ as follows.
Let $Y:\mathsf{Mod}(\mathcal{U})\longrightarrow \mathsf{Mod}\Big(\mathsf{Mod}(\mathcal{U})^{op}\Big)$ be the Yoneda functor
$Y(B):=\mathrm{Hom}_{\mathsf{Mod}(\mathcal{U})}(-,B),$\\
and consider the functor $I:\mathsf{Mod}\Big(\mathsf{Mod}(\mathcal{U})^{op}\Big)\longrightarrow \mathsf{Mod}(\mathcal{T})$,
induced by $E:T\longrightarrow \mathsf{Mod}(\mathcal{U})^{op}$. We set
$$\mathbb{G}:=I\circ Y: \mathsf{Mod}(\mathcal{U})\longrightarrow \mathsf{Mod}(\mathcal{T}).$$

\end{definition}

\begin{note}
In detail, we have that the following holds.
\begin{enumerate}
\item  For $B\in \mathsf{Mod}(\mathcal{U}) $ , $\mathbb{G}(B)(T):= \mathsf{Hom}_{\mathsf{Mod}(\mathcal{U})}(M_{T},B)$ for all $ T\in \mathcal{T} $.\\
Moreover, for all
$t\in \mathsf{Hom}_{\mathcal{T}}(T,T') $ we have that $\mathbb{G}(B)(t):=\mathrm{Hom}_{\mathsf{Mod}(\mathcal{U})}(\bar{t},B) $.
\item If $\eta:B\rightarrow B' $ is a morphism of $\mathcal{U} $-modules we have that
$\mathbb{G}(\eta):\mathbb{G}(B)\longrightarrow \mathbb{G}(B')$  is such that 
$$\mathbb{G}(\eta)= \Big\{  [ \mathbb{G}(\eta)]_{{T}}:=\mathsf{Hom}_{\mathsf{Mod}(\mathcal{U})}(M_{T},\eta):\mathbb{G}(B)(T)\longrightarrow
\mathbb{G}(B')(T) \Big \} _{T\in \mathcal{T}} .$$
\end{enumerate}
\end{note}



Hence we have the comma category $ \Big ( \mathsf{Mod}(\mathcal{T}),\mathbb{G}\mathsf{Mod}(\mathcal{U})\Big )$
whose objects are the triples $ (A,f,B)$ with $A\in \mathsf{Mod}(\mathcal{T}), B\in \mathsf{Mod}(\mathcal{U}), $
and $ f:A\longrightarrow \mathbb{G}(B) $ a morphism of $ \mathcal{T} $-modules.
A morphism between two objects $ (A,f,B) $ and $ (A',f',B') $ is a pairs of morphism $(\alpha,\beta) $ where
$\alpha:A\longrightarrow A'$ is a morphism of $\mathcal{T}$-modules and $\beta:B\longrightarrow B'$ is a
 morphism of $\mathcal{U}$-modules such that the diagram
\[
\begin{diagram}
\node{A} \arrow{e,t}{\alpha}\arrow{s,l}{f}\node{ A'}\arrow{s,r}{f'}\\
\node{\mathbb{G}(B)} \arrow{e,b}{\mathbb{G}(\beta)}\node{\mathbb{G}(B')}
\end{diagram}
\]
commutes.\\
Note that, since $ f:A\longrightarrow \mathbb{G}(B) $ is a morphism of $ \mathcal{T} $-modules, for each
$ t\in \mathsf{Hom}_{\mathcal{T}}(T,T') $ the following diagram
\[
\begin{diagram}
\node{A(T)} \arrow{e,t}{f_{T}}\arrow{s,l}{A(t)}\node{ \mathbb{G}(B)(T)}\arrow{s,r}{\mathbb{G}(B)(t)}\\
\node{A(T')} \arrow{e,b}{f_{T'}}\node{\mathbb{G}(B)(T')}
\end{diagram}
\]
commutes in $\mathbf{Ab}$. Then, for all $ x\in A(T) $, $ f_{T}(x)\in \mathrm{Hom}_{\mathsf{Mod}(\mathcal{U})}(M_{T},B) $,
that is $ f_{T}(x) $ is a morphism of $ \mathcal{U} $-modules. We denote it by
$$f_{T}(x)=\Big \{[f_{T}(x)]_{U}:M_{T}(U)\longrightarrow B(U)\Big \}_{U\in \mathcal{U}}.$$

\begin{proposition}\label{Tlemma1}
Let $M\in \mathsf{Mod}(\mathcal{U}\otimes \mathcal{T}^{op})$ be and
$f:A\longrightarrow \mathbb{G}(B)$ a morphism in $\mathsf{Mod}(\mathcal{T})$.
Let $U\in \mathcal{U}$ and $T\in \mathcal{T}$, for $m\in M(U,T)$ and $x\in A(T)$ we set
$$m\cdot x:=[f_{T}(x)]_{U}(m)\in B(U) $$
(this product can be defined for each $f:A\longrightarrow \mathbb{G}(B)$).
\begin{enumerate}
\item [(a)] For $m\in M(U,T')$, $t\in \mathsf{Hom}_{\mathcal{T}}(T,T')$  and $x\in A(T)$ we set
$$m\bullet t :=M(1_{U}\otimes t^{op})(m),\quad \quad   t\ast x:=A(t)(x).$$
Then we have that  $(m\bullet t)\cdot x=m\cdot (t\ast x).$

\item [(b)] For $m\in M(U,T)$, $u\in Hom_{\mathcal{U}}(U,U')$ and $z\in B(U)$
we set
$$u\bullet m=M(u\otimes 1_{T})(m),\quad \quad u\diamond z:=B(u)(z).$$
Then for $x\in A(t)$ we have that $(u\bullet m)\cdot x= u\diamond (m\cdot x).$

\item [(c)] Let $m_{1}\in M(U',T')$, $m_{2}\in M(U,T)$, $t\in \mathrm{Hom}_{\mathcal{T}}(T,T')$ and $u\in \mathrm{Hom}_{\mathcal{U}}(U,U')$.
For  $x\in A(T)$ we have that
$$(m_{1}\bullet t+u\bullet m_{2})\cdot x=(m_{1}\bullet t)\cdot x+(u\bullet m_{2}) \cdot x=m_{1}\cdot (t\ast x)+u\diamond (m_{2}\cdot x).$$
\end{enumerate}
\end{proposition}
\begin{proof}

\item [(a)] Since $ f:A\longrightarrow \mathbb{G}(B) $ is a morphism of $\mathcal{T}$-modules, for each
$ t\in \mathsf{Hom}_{\mathcal{T}}(T,T') $ the following
diagram
\[
\begin{diagram}
\node{A(T)} \arrow{e,t}{f_{T}}\arrow{s,l}{A(t)}\node{ \mathbb{G}(B)(T)}\arrow{s,r}{\mathbb{G}(B)(t)}\\
\node{A(T')} \arrow{e,b}{f_{T'}}\node{\mathbb{G}(B)(T')}
\end{diagram}
\]
commutes. That is, for $ x\in A(T) $ we have that
$\Big(\mathbb{G}(B)(t)\circ f_{T}\Big)(x)=\Big(f_{T'}\circ A(t)\Big)(x)\in \mathrm{Hom}_{\mathsf{Mod}(\mathcal{U})}(M_{T'},B)$.
Then, we have that
\begin{eqnarray*}
f_{T'}\Big( A(t)(x)\Big)&=& \Big(\mathbb{G}(B)(t)\circ f_{T}\Big)(x)\\
&=& \Big(\mathsf{Hom}_{\mathsf{Mod}(\mathcal{U})}(\bar{t},B) \circ f_{T}\Big)(x)\\
&=& f_{T}(x)\circ \bar{t}.
\end{eqnarray*}
Hence, for $U\in \mathcal{U}$ we have that
\[
\left [f_{T'}\Big(A(t)(x) \Big) \right ]_{U}= [f_{T}(x)]_{U}\circ[\bar{t}]_{U}
=[f_{T}(x)]_{U}\circ M(1_{U}\otimes t^{op}).
\]

It follows that 
\begin{eqnarray}\label{natural1}
 [f_{T'}(A(t)(x) )]_{U}(m)=
[f_{T}(x)]_{U}(M(1_{U}\otimes t^{op})(m)), \text{   for all } m\in M(U,T').
\end{eqnarray}
This means that
$(m\bullet t)\cdot x=m\cdot (t\ast x).$

\item [(b)] Since $ f_{T}(x)\in \mathrm{Hom}_{\mathsf{Mod}(\mathcal{U})}(M_{T},B)$
with $f_{T}(x)=\Big \{[f_{T}(x)]_{U}:M(U,T)\longrightarrow B(U) \Big \}_{U\in \mathcal{U}}$ is a morphism of
$\mathcal{U}$-modules, for $u:U\longrightarrow U'$ we have the following commutative diagram
\[
\begin{diagram}
\node{M_{T}(U)} \arrow{e,t}{\left[f_{T}(x)\right]_{U}}\arrow{s,l}{M_{T}(u)}\node{ B(U)}\arrow{s,r}{B(u)}\\
\node{M_{T}(U')} \arrow{e,b}{\left[f_{T}(x)\right]_{U'}}\node{B(U').}
\end{diagram}
\]
Then, for $m\in M_{T}(U)=M(U,T) $ we have that
\begin{eqnarray*}
\Big(B(u)\circ\left[f_{T}(x)\right]_{U}\Big)(m)&=& \Big(\left[f_{T}(x)\right]_{U'}\circ M_{T}(u)\Big)(m)\\
&=& \Big(\left[f_{T}(x)\right]_{U'}\circ M(u\otimes 1_{T})\Big)(m)\\
&=& \left[f_{T}(x)\right]_{U'}(M(u\otimes 1_{T})(m))
\end{eqnarray*}
This means that $u\diamond (m\cdot x)=(u\bullet m)\cdot x$.

\item [(c)] Since $m_{1}\in M(U',T')$ and $m_{2}\in M(U,T)$   we have that $m_{1}\bullet t=M(1_{U'}\otimes t^{op})(m_{1})\in M(U',T)$ and  $u\bullet m_{2}=M(u\otimes 1_{T})(m_{2})\in M(U',T)$.
Now, since $M(U',T)$ is an abelian group we can consider
the element $m':=m_{1}\bullet  t+u\bullet m_{2}\in M(U',T)$.
Then by definition, for $x\in A(T)$ we have that
$$m'\cdot x=[f_{T}(x)]_{U'}(m').$$
Since $[f_{T}(x)]_{U'}:M(U',T)\longrightarrow B(U')$ is a morphism of abelian groups we have that
$[f_{T}(x)]_{U'}(m')=[f_{T}(x)]_{U'}(m_{1}\bullet t)+[f_{T}(x)]_{U'}(u\bullet m_{2})=(m_{1}\bullet t)\cdot x+
(u\bullet m_{2})\cdot x$. Then, the result follows from $(a)$ and $(b)$.
\end{proof}

\begin{definition}\label{defitrinagularmat}
We define the \textbf{triangular matrix category}
$\mathbf{\Lambda}=\left[ \begin{smallmatrix}
\mathcal{T} & 0 \\ M & \mathcal{U}
\end{smallmatrix}\right]$ as follows.
\begin{enumerate}
\item [(a)] The class of objects of this category are matrices $ \left[
\begin{smallmatrix}
T & 0 \\ M & U
\end{smallmatrix}\right]  $ with $ T\in \mathrm{obj} \ \mathcal{T} $ and $ U\in \mathrm{obj} \ \mathcal{U} $.

\item [(b)] Given a pair of objects in
$\left[ \begin{smallmatrix}
T & 0 \\
M & U
\end{smallmatrix} \right] ,  \left[ \begin{smallmatrix}
T' & 0 \\
M & U'
\end{smallmatrix} \right]$ in
$\mathbf{\Lambda}$ we define

$$\mathsf{ Hom}_{\mathbf{\Lambda}}\left (\left[ \begin{smallmatrix}
T & 0 \\
M & U
\end{smallmatrix} \right] ,  \left[ \begin{smallmatrix}
T' & 0 \\
M & U'
\end{smallmatrix} \right]  \right)  := \left[ \begin{smallmatrix}
\mathsf{Hom}_{\mathcal{T}}(T,T') & 0 \\
M(U',T) & \mathsf{Hom}_{\mathcal{U}}(U,U')
\end{smallmatrix} \right]$$.
\end{enumerate}
The composition is given by
\begin{eqnarray*}
\circ&:&\left[  \begin{smallmatrix}
{\mathcal{T}}(T',T'') & 0 \\
M(U'',T') & {\mathcal{U}}(U',U'')
\end{smallmatrix}  \right] \times \left[
\begin{smallmatrix}
{\mathcal{T}}(T,T') & 0 \\
M(U',T) & {\mathcal{U}}(U,U')
\end{smallmatrix} \right]\longrightarrow\left[
\begin{smallmatrix}
{\mathcal{T}}(T,T'') & 0 \\
M(U'',T) & {\mathcal{U}}(U,U'')\end{smallmatrix} \right] \\
&& \left( \left[ \begin{smallmatrix}
t_{2} & 0 \\
m_{2} & u_{2}
\end{smallmatrix} \right], \left[
\begin{smallmatrix}
t_{1} & 0 \\
m_{1} & u_{1}
\end{smallmatrix} \right]\right)\longmapsto\left[
\begin{smallmatrix}
t_{2}\circ t_{1} & 0 \\
m_{2}\bullet t_{1}+u_{2}\bullet m_{1} & u_{2}\circ u_{1}
\end{smallmatrix} \right].
\end{eqnarray*}
\end{definition}
We recall that $ m_{2}\bullet t_{1}:=M(1_{U''}\otimes t_{1}^{op})(m_{2})$ and
$u_{2}\bullet m_{1}=M(u_{2}\otimes 1_{T})(m_{1})$.

\begin{lemma}\label{asociatmat}
The composition defined above is associative and given 
and object $\left[
\begin{smallmatrix}
T & 0 \\
M & U
\end{smallmatrix} \right]\in \mathbf{\Lambda}$, the identity morphism is given by $1_{\left[
\begin{smallmatrix}
T & 0 \\
M & U
\end{smallmatrix} \right]}:=\left[
\begin{smallmatrix}
1_{T} & 0 \\
0 & 1_{U}
\end{smallmatrix} \right].$
\end{lemma}
\begin{proof}
Let
$$\left[
\begin{smallmatrix}
t_{1} & 0 \\
m_{1} & u_{1}
\end{smallmatrix}\right]\in \mathsf{ Hom}_{\mathbf{\Lambda}}\left (\left[ \begin{smallmatrix}
T & 0 \\
M & U
\end{smallmatrix} \right] ,  \left[ \begin{smallmatrix}
T' & 0 \\
M & U'
\end{smallmatrix} \right]  \right)  = \left[ \begin{smallmatrix}
\mathsf{Hom}_{\mathcal{T}}(T,T') & 0 \\
M(U',T) & \mathsf{Hom}_{\mathcal{U}}(U,U')
\end{smallmatrix} \right]$$

$$\left[
\begin{smallmatrix}
t_{2} & 0 \\
m_{2} & u_{2}
\end{smallmatrix}\right]\in \mathsf{ Hom}_{\mathbf{\Lambda}}\left (\left[ \begin{smallmatrix}
T' & 0 \\
M & U'
\end{smallmatrix} \right] ,  \left[ \begin{smallmatrix}
T'' & 0 \\
M & U''
\end{smallmatrix} \right]  \right)  = \left[ \begin{smallmatrix}
\mathsf{Hom}_{\mathcal{T}}(T',T'') & 0 \\
M(U'',T') & \mathsf{Hom}_{\mathcal{U}}(U',U'')
\end{smallmatrix} \right]$$

$$\left[
\begin{smallmatrix}
t_{3} & 0 \\
m_{3} & u_{3}
\end{smallmatrix}\right]\in \mathsf{ Hom}_{\mathbf{\Lambda}}\left (\left[ \begin{smallmatrix}
T ''& 0 \\
M & U''
\end{smallmatrix} \right] ,  \left[ \begin{smallmatrix}
T''' & 0 \\
M & U'''
\end{smallmatrix} \right]  \right)  = \left[ \begin{smallmatrix}
\mathsf{Hom}_{\mathcal{T}}(T'',T''') & 0 \\
M(U''',T'') & \mathsf{Hom}_{\mathcal{U}}(U'',U''')
\end{smallmatrix} \right]$$

On one hand, we have that

$$\left[
\begin{smallmatrix}
t_{3} & 0 \\
m_{3} & u_{3}
\end{smallmatrix} \right]
\left[
\begin{smallmatrix}
t_{2}\circ t_{1} & 0 \\
m_{2}\bullet t_{1}+u_{2}\bullet m_{1} & u_{2}\circ u_{1}
\end{smallmatrix} \right]= \left[
\begin{smallmatrix}
t_{3}(t_{2}\circ t_{1}) & 0 \\
m_{3}\bullet(t_{2}\circ t_{1})+u_{3}\bullet (m_{2}\bullet t_{1}+u_{2}\bullet m_{1}) & u_{3}(u_{2}\circ u_{1})
\end{smallmatrix} \right].$$
On the other hand
$$
\left[
\begin{smallmatrix}
t_{3}\circ t_{2} & 0 \\
m_{3}\bullet t_{2}+u_{3}\bullet m_{2} & u_{3}\circ u_{2}
\end{smallmatrix} \right]
\left[
\begin{smallmatrix}
t_{1} & 0 \\
m_{1} & u_{1}
\end{smallmatrix} \right]
= \left[
\begin{smallmatrix}
(t_{3}\circ t_{2})\circ t_{1} & 0 \\
(m_{3}\bullet t_{2}+u_{3}\bullet m_{2})\bullet t_{1}
+(u_{3}\circ u_{2})\bullet m_{1} & (u_{3}\circ u_{2})\circ u_{1}
\end{smallmatrix} \right].$$
Now, we observe that
\begin{enumerate}
\item [(a)] $m_{3}\bullet (t_{2}\circ t_{1})=(m_{3}\bullet t_{2})\bullet t_{1}$. Indeed,
$m_{3}\bullet (t_{2}\circ t_{1})=M(1_{U'''},(t_{2}t_{1})^{op})(m_{3})=M(1_{U'''},t_{1}^{op})\Big(M(1_{U'''},t_{2}^{op})(m_{3})\Big)
=(m_{3}\bullet t_{2})\bullet t_{1}\in M(U''',T)$.

\item [(b)] $u_{3}\bullet (m_{2}\bullet t_{1}+u_{2}\bullet m_{1})=u_{3}\bullet (m_{2}\bullet t_{1})+
u_{3}\bullet (u_{2}\bullet m_{1}).$ Indeed, this follows from the fact that $M(u_{3}\otimes 1_{T})$ is a morphism
of abelian groups.

\item [(c)] $(u_{3}\circ u_{2})\bullet m_{1}=u_{3}\bullet (u_{2}\bullet m_{1})$. This is similar to (a).

\item [(d)] $(m_{3}\bullet t_{2}+u_{3}\bullet m_{2})\bullet t_{1}=
(m_{3}\bullet t_{2})\bullet t_{1}+(u_{3}\bullet m_{2})\bullet t_{1}$. This is similar to (c)

\item [(e)] $u_{3}\bullet (m_{2}\bullet t_{1})=(u_{3}\bullet m_{2})\bullet t_{1}$.
Indeed, since we have the morphisms
$t_{1}:T\longrightarrow T'$ and $u_{3}:U''\longrightarrow U'''$,
there exists a morphism $M(u_{3}\otimes t_{1}^{op}):M(U'',T')\longrightarrow M(U''',T)$.
The assertion follows from the fact
$$M(u_{3}\otimes t_{1}^{op})=M(u_{3}\otimes 1_{T})\circ M(1_{U''}\otimes t_{1}^{op})=
M(1_{U'''}\otimes t_{1}^{op})\circ M(u_{3}\otimes 1_{T'}).$$
\end{enumerate}
Then the associativity follows from (a) to (d).
It is easy to see that
$1_{\left[
\begin{smallmatrix}
T & 0 \\
M & U
\end{smallmatrix} \right]}:=\left[
\begin{smallmatrix}
1_{T} & 0 \\
0 & 1_{U}
\end{smallmatrix} \right]$ is the identity.
\end{proof}

Now, for
$$\left[
\begin{smallmatrix}
t_{1} & 0 \\
m_{1} & u_{1}
\end{smallmatrix}\right],\left[
\begin{smallmatrix}
r_{1} & 0 \\
n_{1} & v_{1}
\end{smallmatrix}\right] \in \mathsf{ Hom}_{\mathbf{\Lambda}}\left (\left[ \begin{smallmatrix}
T & 0 \\
M & U
\end{smallmatrix} \right] ,  \left[ \begin{smallmatrix}
T' & 0 \\
M & U'
\end{smallmatrix} \right]  \right)  = \left[ \begin{smallmatrix}
\mathsf{Hom}_{\mathcal{T}}(T,T') & 0 \\
M(U',T) & \mathsf{Hom}_{\mathcal{U}}(U,U')
\end{smallmatrix} \right]$$
we define
$$\left[
\begin{smallmatrix}
t_{1} & 0 \\
m_{1} & u_{1}
\end{smallmatrix}\right]+ \left[
\begin{smallmatrix}
r_{1} & 0 \\
n_{1} & v_{1}
\end{smallmatrix}\right]:=\left[
\begin{smallmatrix}
t_{1}+r_{1} & 0 \\
m_{1}+n_{1} & u_{1}+v_{1}
\end{smallmatrix}\right]$$
Then, it is clear that $\mathbf{\Lambda} $ is a preadditive category since
$\mathcal{T} $ and $\mathcal{U}$ are preadditive categories and $M(U',T)$ is an abelian group.\\

\begin{proposition}\label{fincopromat}
If  $\mathcal{U}$ and $\mathcal{T}$ have finite coproducts, then
$\mathbf{\Lambda}$ has finite coproducts.
\end{proposition}
\begin{proof}
It is straightforward.
\end{proof}
Now, we compute the radical in $\mathbf{\Lambda}$.

\begin{proposition}\label{radtriangular}
$\mathrm{rad}_{\mathbf{\Lambda}}\left (\left[ \begin{smallmatrix}
T & 0 \\
M & U
\end{smallmatrix} \right] ,  \left[ \begin{smallmatrix}
T' & 0 \\
M & U'
\end{smallmatrix} \right]  \right)  = \left[ \begin{smallmatrix}
\mathsf{rad}_{\mathcal{T}}(T,T') & 0 \\
M(U',T) & \mathsf{rad}_{\mathcal{U}}(U,U')
\end{smallmatrix} \right]$

\end{proposition}
\begin{proof}

Let
$\left[
\begin{smallmatrix}
t & 0 \\
m & u
\end{smallmatrix}\right]\in \mathsf{rad}_{\mathbf{\Lambda}}\left (\left[ \begin{smallmatrix}
T & 0 \\
M & U
\end{smallmatrix} \right] ,  \left[ \begin{smallmatrix}
T' & 0 \\
M & U'
\end{smallmatrix} \right]  \right)$ and $t':T'\longrightarrow  T$ and $u':U'\longrightarrow U$ morphisms in $\mathcal{T}$  and $\mathcal{U}$ respectively. Consider  $\left[
\begin{smallmatrix}
t' & 0 \\
0 & u'
\end{smallmatrix}\right]\in \mathsf{Hom}_{\mathbf{\Lambda}}\left (\left[ \begin{smallmatrix}
T' & 0 \\
M & U'
\end{smallmatrix} \right] ,  \left[ \begin{smallmatrix}
T & 0 \\
M & U
\end{smallmatrix} \right]  \right),$
then  $\left[
\begin{smallmatrix}
1_{T} & 0 \\
0 & 1_{U}
\end{smallmatrix} \right]-\left[
\begin{smallmatrix}
t' & 0 \\
0 & u'
\end{smallmatrix}\right]\left[
\begin{smallmatrix}
t & 0 \\
m & u
\end{smallmatrix}\right]=\left[
\begin{smallmatrix}
1_{T}-t't & 0 \\
-u'\bullet m & 1_{U}-u'u
\end{smallmatrix}\right]$ is invertible in $\mathbf{\Lambda}$. It follows from this that $1_{T}-t't $ and $1_{U}-u'u$ are invertibles in $\mathcal{T}$  and $\mathcal{U}$ respectively. Then $t\in \mathrm{rad}_{\mathcal{T}}(T,T')$ and $u\in \mathrm{rad}_{\mathcal{U}}(U,U')$.\\
Conversely, let $t\in \mathrm{rad}_{\mathcal{T}}(T,T')$, $u\in \mathrm{rad}_{\mathcal{U}}(U,U')$ and $m\in M(U',T)$. We assert that  $\left[
\begin{smallmatrix}
t & 0 \\
m & u
\end{smallmatrix}\right]\in \mathsf{rad}_{\mathbf{\Lambda}}\left (\left[ \begin{smallmatrix}
T & 0 \\
M & U
\end{smallmatrix} \right] ,  \left[ \begin{smallmatrix}
T' & 0 \\
M & U'
\end{smallmatrix} \right]  \right).$ Indeed, let $\left[
\begin{smallmatrix}
t' & 0 \\
m' & u'
\end{smallmatrix}\right]\in \mathsf{Hom}_{\mathbf{\Lambda}}\left (\left[ \begin{smallmatrix}
T' & 0 \\
M & U'
\end{smallmatrix} \right] ,  \left[ \begin{smallmatrix}
T & 0 \\
M & U
\end{smallmatrix} \right]  \right).$
Since $1_{T}-t't$ and $1_{U}-u'u$ are invertibles, there exists $t''\in \mathrm{Hom}_{\mathcal{T}}(T,T)$ and $u''\in \mathrm{Hom}_{\mathcal{U}}(U,U)$ such that 
$$(1_{T}-t't)t''=t''(1_{T}-t't)=1_{T},\quad (1_{U}-u'u)u''=u''(1_{U}-u'u)=1_{U}.$$ 
Let $m'':=u''\bullet(m'\bullet t+u'\bullet m)\bullet t''\in M(U,T)$. 
Using that $(1_{U}-u'u)u''=1_{U}$ and that $t''(1_{T}-t't)=1_{T}$ we see that
 
\begin{align*}
\Big(\left[
\begin{smallmatrix}
1_{T} & 0 \\
0 & 1_{U}
\end{smallmatrix} \right]-\left[
\begin{smallmatrix}
t' & 0 \\
m' & u'
\end{smallmatrix}\right]\left[
\begin{smallmatrix}
t & 0 \\
m & u
\end{smallmatrix}\right]\Big)\Big(\left[
\begin{smallmatrix}
t'' & 0 \\
m'' & u''
\end{smallmatrix}\right]\Big) & =\Big(\left[
\begin{smallmatrix}
1_{T} & 0 \\
0 & 1_{U}
\end{smallmatrix} \right]-\left[
\begin{smallmatrix}
t' t & 0 \\
m'\bullet t+u'\bullet m & u'u 
\end{smallmatrix}\right]
\Big)\Big(\left[
\begin{smallmatrix}
t'' & 0 \\
m'' & u''
\end{smallmatrix}\right]\Big)\\
& =\Big(\left[
\begin{smallmatrix}
1_{T}-t' t & 0 \\
-(m'\bullet t+u'\bullet m) & 1_{U}- u'u 
\end{smallmatrix}\right]
\Big)\Big(\left[
\begin{smallmatrix}
t'' & 0 \\
m'' & u''
\end{smallmatrix}\right]\Big)\\
& = \left[
\begin{smallmatrix}
(1_{T}-t' t )t''& 0 \\
-(m'\bullet t+u'\bullet m)\bullet t''+(1_{U}-u'u)\bullet m'' & (1_{U}- u'u)u''
\end{smallmatrix}\right]\\
& = \left[
\begin{smallmatrix}
1_{T} & 0 \\
-(m'\bullet t+u'\bullet m)\bullet t''+(1_{U}-u'u)\bullet m'' & 1_{U}
\end{smallmatrix}\right]\\
& = \left[
\begin{smallmatrix}
1_{T} & 0 \\
0  & 1_{U}
\end{smallmatrix}\right],
\end{align*}
where the last equality is because $(1_{U}-u'u)\bullet m'' =(1_{U}-u'u)\bullet \Big(u''\bullet(m'\bullet t+u'\bullet m)\bullet t''\Big)=\Big((1_{U}-u'u)u''\Big)\bullet\Big( (m'\bullet t+u'\bullet m)\bullet t''\Big)=1_{U}\bullet\Big( (m'\bullet t+u'\bullet m)\bullet t''\Big)=(m'\bullet t+u'\bullet m)\bullet t''.$
Similarly, 
\begin{align*}
\Big(\left[
\begin{smallmatrix}
t'' & 0 \\
m'' & u''
\end{smallmatrix}\right]\Big)
\Big(\left[
\begin{smallmatrix}
1_{T} & 0 \\
0 & 1_{U}
\end{smallmatrix} \right]-\left[
\begin{smallmatrix}
t' & 0 \\
m' & u'
\end{smallmatrix}\right]\left[
\begin{smallmatrix}
t & 0 \\
m & u
\end{smallmatrix}\right]\Big) & =\Big(\left[
\begin{smallmatrix}
t'' & 0 \\
m'' & u''
\end{smallmatrix}\right]\Big)\Big(\left[
\begin{smallmatrix}
1_{T} & 0 \\
0 & 1_{U}
\end{smallmatrix} \right]-\left[
\begin{smallmatrix}
t' t & 0 \\
m'\bullet t+u'\bullet m & u'u 
\end{smallmatrix}\right]
\Big)\\
& =\Big(\left[
\begin{smallmatrix}
t'' & 0 \\
m'' & u''
\end{smallmatrix}\right]\Big)\Big(\left[
\begin{smallmatrix}
1_{T}-t' t & 0 \\
-(m'\bullet t+u'\bullet m) & 1_{U}- u'u 
\end{smallmatrix}\right]
\Big)\\
& = \left[
\begin{smallmatrix}
t''(1_{T}-t' t ) & 0 \\
m''\bullet (1_{T}-t't) -u''\bullet (m'\bullet t+u'\bullet m) & u''(1_{U}- u'u)
\end{smallmatrix}\right]\\
& = \left[
\begin{smallmatrix}
1_{T}& 0 \\
m''\bullet (1_{T}-t't) -u''\bullet (m'\bullet t+u'\bullet m)  & 1_{U}
\end{smallmatrix}\right]\\
& = \left[
\begin{smallmatrix}
1_{T} & 0 \\
0  & 1_{U}
\end{smallmatrix}\right],
\end{align*}
where the last equality is because $m''\bullet (1_{T}-t't)=\Big(u''\bullet(m'\bullet t+u'\bullet m)\bullet t''\Big)\bullet (1_{T}-t't)=\Big(u''\bullet (m'\bullet t+u'\bullet m)\Big)\bullet \Big(t''(1_{T}-t't)\Big)=\Big(u''\bullet  (m'\bullet t+u'\bullet m)\Big)\bullet 1_{T}=u''\bullet (m'\bullet t+u'\bullet m).$ Proving that $\left[
\begin{smallmatrix}
t & 0 \\
m & u
\end{smallmatrix}\right]\in \mathsf{rad}_{\mathbf{\Lambda}}\left (\left[ \begin{smallmatrix}
T & 0 \\
M & U
\end{smallmatrix} \right] ,  \left[ \begin{smallmatrix}
T' & 0 \\
M & U'
\end{smallmatrix} \right]  \right).$ 
\end{proof}

Our main pourpose in this part is to show that we have an equivalence of categories 
$$\Big(\mathsf{Mod}(\mathcal{T}),\mathbb{G}\mathsf{Mod}(\mathcal{U})\Big)\simeq \mathsf{Mod}(\mathbf{\Lambda}).$$
Let $(A,f,B) \in \Big(\mathsf{Mod}(\mathcal{T}),\mathbb{G}\mathsf{Mod}(\mathcal{U})\Big)$ be, that is, we have a morphism of $\mathcal{T}$-modules
$f:A\longrightarrow \mathbb{G}(B)$. We can construct a functor
$$A\amalg _f B: \mathbf{\Lambda} \rightarrow\mathbf {Ab}$$
as follows.
\begin{enumerate}
\item [(a)] For
$\left[\begin{smallmatrix}
T& 0 \\
M & U \\
\end{smallmatrix} \right]\in\mathbf{\Lambda}$ we set  $\Big(A\amalg_f B\Big)(\left[\begin{smallmatrix}
T& 0 \\
M & U \\
\end{smallmatrix} \right]) :=A (T)\amalg  B(U)\in \mathbf{Ab}.$

\item [(b)] If $\left[\begin{smallmatrix}
t& 0 \\
m& u \\
\end{smallmatrix} \right]\in \mathsf{Hom}_{\mathbf{\Lambda}}(\left[\begin{smallmatrix}
T& 0 \\
M & U \\
\end{smallmatrix} \right], \left[\begin{smallmatrix}
T'& 0 \\
M & U' \\
\end{smallmatrix} \right])=\left[ \begin{smallmatrix}
\mathsf{Hom}_{\mathcal{T}}(T,T') & 0 \\
M(U',T) & \mathsf{Hom}_{\mathcal{U}}(U,U')
\end{smallmatrix} \right]$ we define the map
$$\Big(A\amalg_f B\Big)(\left[\begin{smallmatrix}
t& 0 \\
m& u \\
\end{smallmatrix} \right]):=\left [\begin{smallmatrix}
A(t)& 0 \\
m & B(u) \\
\end{smallmatrix} \right]: A(T)\amalg B(U)\rightarrow A(T')\amalg B(U')$$
given by $\left [\begin{smallmatrix}
A(t)& 0 \\
m& B(u) \\
\end{smallmatrix} \right]\left [\begin{smallmatrix}
x \\
\\
y \\
\end{smallmatrix} \right]=\left [\begin{smallmatrix}
A(t)(x) \\
m\cdot x+ B(u)(y) \\
\end{smallmatrix} \right]$
for $(x,y)\in A(T)\amalg B(U)$, where $m\cdot x:=[f_{T}(x)]_{U'}(m)\in B(U')$ (see  \ref{Tlemma1}).
\end{enumerate}
\begin{note}
In terms of \ref{Tlemma1}, we have that
$$\left [\begin{smallmatrix}
A(t)& 0 \\
m& B(u) \\
\end{smallmatrix} \right]\left [\begin{smallmatrix}
x \\
\\
y \\
\end{smallmatrix} \right]=\left [\begin{smallmatrix}
t\ast x \\
\\
m\cdot x + u\diamond y \\
\end{smallmatrix} \right]\quad \forall (x,y)\in A(T)\amalg B(U)$$
\end{note}

The following lemma tell us that $A\amalg_{f}B$ is a functor.

\begin{lemma}
Let $(A,f,B) \in \Big(\mathsf{Mod}(\mathcal{T}),\mathbb{G}\mathsf{Mod}(\mathcal{U})\Big)$ be, then $A\amalg_{f}B:\mathbf{\Lambda} \rightarrow\mathbf {Ab}$ is a functor.
\end{lemma}
\begin{proof}
Let $\left[ \begin{smallmatrix}
t_{1} & 0 \\
m_{1} & u_{1}
\end{smallmatrix} \right]: \left[ \begin{smallmatrix}
T & 0 \\
M & U
\end{smallmatrix} \right] \longrightarrow \left[ \begin{smallmatrix}
T' & 0 \\
M & U'
\end{smallmatrix} \right]$, $\left[ \begin{smallmatrix}
t_{2} & 0 \\
m_{2} & u_{2}
\end{smallmatrix} \right]: \left[ \begin{smallmatrix}
T' & 0 \\
M & U'
\end{smallmatrix} \right]\longrightarrow \left[ \begin{smallmatrix}
T'' & 0 \\
M & U''
\end{smallmatrix} \right]$ and  $(x,y)\in A(T)\amalg B(U)$.\\
We have that
\begin{eqnarray*}
\Big((A\amalg_{f} B) \Big(\left[ \begin{smallmatrix}
t_{2} & 0 \\
m_{2} & u_{2}
\end{smallmatrix} \right]  \circ \left[
\begin{smallmatrix}
t_{1} & 0 \\
m_{1} & u_{1}
\end{smallmatrix} \right] \Big)\Big)
\left[ \begin{smallmatrix}
x
\\
\\
y
\end{smallmatrix} \right]
& =&(A\amalg_{f} B) \Big(\left[ \begin{smallmatrix}
t_{2}t_{1} & 0 \\
m_{2}\bullet t_{1}+u_{2}\bullet m_{1} & u_{2}u_{1}
\end{smallmatrix} \right]\Big)\left[ \begin{smallmatrix}
x
\\
\\
y
\end{smallmatrix} \right]\\
&=&\left[
\begin{smallmatrix}
 A(t_{2}\circ t_{1}) & 0 \\
m_{2}\bullet t_{1}+u_{2}\bullet m_{1} & B( u_{2}\circ u_{1})
\end{smallmatrix} \right]
\left[ \begin{smallmatrix}
x
\\
\\
y
\end{smallmatrix} \right]\\
&=&\left[\begin{smallmatrix}
A(t_{2}\circ t_{1}) (x) \\
(m_{2}\bullet t_{1}+u_{2}\bullet m_{1})\cdot x+B( u_{2}\circ u_{1})(y)
\end{smallmatrix} \right]
\end{eqnarray*}
On the other hand,
\begin{align*}
 \Big( \!(A\amalg_{f} B) \!\left( \!\left[
\begin{smallmatrix}
t_{2} & 0 \\
m_{2} & u_{2}
\end{smallmatrix} \right] \! \right) \!\!\Big)\!\!\circ \!\! \Big( \! (A\amalg_{f} B) \!\left( \!  \left[
 \begin{smallmatrix}
t_{1} & 0 \\
m_{1} & u_{1}
\end{smallmatrix} \right] \!\right) \! \!\Big)
\!\!\left[
\begin{smallmatrix}
x \\
\\
y
\end{smallmatrix} \right] &  \!\!= \!\!\Big(\!\left[
\begin{smallmatrix}
A(t_{2}) & 0 \\
m_{2} & B( u_{2})
\end{smallmatrix} \right] \left[
\begin{smallmatrix}
A(t_{1}) & 0 \\
m_{1} & B( u_{1})
\end{smallmatrix} \right] \! \Big)\left[
\begin{smallmatrix}
x \\
\\
y
\end{smallmatrix} \right] \\
& \!\!= \!\! \left[ \begin{smallmatrix}
A(t_{2}) & 0 \\
m_{2} & B( u_{2})
\end{smallmatrix} \right]\left[ \begin{smallmatrix}
A(t_{1})(x) \\
m_{1}\cdot x+  B( u_{1})(y)
\end{smallmatrix} \right]\\
 & \!\!= \!\!  \left[\!\begin{smallmatrix}
A(t_{2})(A(t_{1})(x)) \\
m_{2}\cdot (A(t_{1})(x))+B(u_{2})(m_{1}\cdot x+B(u_{1})(y))
\end{smallmatrix} \!\right]\\
& \!\!= \!\!  \left[
\begin{smallmatrix}
A(t_{2}t_{1})(x) \\
m_{2}\cdot (t_{1}\ast x)+u_{2}\diamond(m_{1}\cdot x)+B(u_{2}u_{1})(y)
\end{smallmatrix} \right]
\end{align*}
By \ref{Tlemma1}(c), we conclude that $(m_{2}\bullet t_{1}+u_{2}\bullet m_{1})\cdot x=
m_{2}\cdot (t_{1}\ast x)+u_{2}\diamond(m_{1}\cdot x)$. Proving that
$A\amalg_{f}B$ preserves compositions. Now, consider $\left[ \begin{smallmatrix}
1_{T} & 0 \\
0 & 1_{U}
\end{smallmatrix} \right]: \left[ \begin{smallmatrix}
T & 0 \\
M & U
\end{smallmatrix} \right] \longrightarrow \left[ \begin{smallmatrix}
T & 0 \\
M & U
\end{smallmatrix} \right]$. We have that
$\Big(A\amalg_f B\Big)(\left[\begin{smallmatrix}
1_{T}& 0 \\
0& 1_{U} \\
\end{smallmatrix} \right]):=\left [\begin{smallmatrix}
A(1_{T})& 0 \\
0 & B(1_{U}) \\
\end{smallmatrix}\right]$ is such that $\left [\begin{smallmatrix}
A(1_{T})& 0 \\
0 & B(1_{U}) \\
\end{smallmatrix}\right]\left[
\begin{smallmatrix}
x \\
\\
y
\end{smallmatrix} \right]=\left [\begin{smallmatrix}
A(1_{T})(x)\\
0\cdot x + B(1_{U})(y) \\
\end{smallmatrix}\right]=\left [\begin{smallmatrix}
x\\
\\
y\\
\end{smallmatrix}\right]$ for $(x,y)\in A(T)\amalg B(U)$,  since $0\cdot x=0$. Then $(A\amalg_{f}B)\big (1_{_{\left[
\begin{smallmatrix}
T & 0 \\
M & U
\end{smallmatrix} \right]}}\big )=1_{_{(A\amalg _{f} B)\left[
\begin{smallmatrix}
T & 0 \\
M & U
\end{smallmatrix} \right]}}=1_{A(T)\amalg B(U)}.$ Proving that $A\amalg_{f}B$ is a functor.
\end{proof}
In this way we can construct a functor
$$\textswab{F}:\Big( \mathsf{Mod}(\mathcal{T}),\mathbb{G}\mathsf{Mod}(\mathcal{U})\Big)\rightarrow \mathrm{Mod}(\mathbf\Lambda)$$
which is defined as follows.
\begin{enumerate}
\item [(a)] For $(A,f,B)\in \Big ( \mathsf{Mod}(\mathcal{T}),\mathbb{G}\mathsf{Mod}(\mathcal{U})\Big ) $ we  define $\textswab{F}((A,f,B)):=A\amalg_f B$.

\item [(b)] If we have
$(\alpha,\beta):(A,f,B)\rightarrow (A',f',B')$ in $\Big(\mathsf{Mod}(\mathcal{T}),\mathbb{G}\mathsf{Mod}(\mathcal{U})\Big)$ then $\textswab{F}(\alpha,\beta)=\alpha\amalg \beta$
is the natural transformation
$$\alpha\amalg \beta\! =\!\left\{\!\!
(\alpha\!\amalg\!\beta)_{{}_{\left[
\begin{smallmatrix}
T& 0 \\
M& U\\
\end{smallmatrix}
\right] }}\!\!:=\!\alpha_{T}\!\amalg \!\beta_{U}\!\!:\!\! (A\amalg_{f} B) (\left [\begin{smallmatrix}
T& 0 \\
M& U\\
\end{smallmatrix} \right])\rightarrow (A'\amalg_{f'}\!\!B')(\left [\begin{smallmatrix}
T& 0 \\
M& U\\
\end{smallmatrix} \right])\!\!\right\}_{\!\!\left[\begin{smallmatrix}
T& 0 \\
M& U\\
\end{smallmatrix} \right ]\in \mathbf{\Lambda}}$$
\end{enumerate}

\begin{lemma}
Let $(\alpha,\beta):(A,f,B)\rightarrow (A',f',B')$  be in $\Big(\mathsf{Mod}(\mathcal{T}),\mathbb{G}\mathsf{Mod}(\mathcal{U})\Big)$, then
$\alpha\amalg\beta$ is a  natural transformation. 
\end{lemma}
\begin{proof}
Let $\left[ \begin{smallmatrix}
t_{1} & 0 \\
m_{1} & u_{1}
\end{smallmatrix} \right]: \left[ \begin{smallmatrix}
T & 0 \\
M & U
\end{smallmatrix} \right] \longrightarrow \left[ \begin{smallmatrix}
T' & 0 \\
M & U'
\end{smallmatrix} \right]$ be a morphism in $\mathbf{\Lambda}=\left[ \begin{smallmatrix}
\mathcal{T} & 0 \\ M & \mathcal{U}
\end{smallmatrix}\right]$. We have to check that the following diagram commutes in $\mathbf{Ab}$

$$\xymatrix{A(T)\amalg B(U)\ar[rr]^{\alpha_{T}\amalg \beta_{U}}\ar[d]_{{\left[ \begin{smallmatrix}
A(t_{1}) & 0 \\
m_{1} & B(u_{1})\\
\end{smallmatrix} \right]}}  & & A'(T)\amalg B'(U)\ar[d]^{{\left[ \begin{smallmatrix}
A'(t_{1}) & 0 \\
m_{1} & B'(u_{1})\\
\end{smallmatrix} \right]}}\\
A(T')\amalg B(U')\ar[rr]^{\alpha_{T'}\amalg \beta_{U'}}
 & & A'(T')\amalg B'(U').}$$
Indeed, for $(x,y)\in A(T)\amalg B(U)$ we have that
\begin{eqnarray*}
\Big((\alpha_{T'}\amalg \beta_{U'})\circ \left[ \begin{smallmatrix}
A(t_{1}) & 0 \\
m_{1} & B(u_{1})\\
\end{smallmatrix} \right]\Big)\left[
\begin{smallmatrix}
x \\
\\
y
\end{smallmatrix} \right] &\!\!\!\!\!\!= &\!\!\!\!\!(\alpha_{T'}\amalg \beta_{U'})
\Big (A(t_{1})(x),m_{1}\cdot x+B(u_{1})(y)\Big )\\
&\!\!\!\!\!\!= & \!\!\!\!\!\! \Big( \alpha_{T'}(A(t_{1}(x))),\beta_{U'}\big(
m_{1}\cdot x+B(u_{1})(y)\big)\Big)
\end{eqnarray*}
We also have that
\begin{align*}
\Big(\!\left[ \begin{smallmatrix}
A'(t_{1}) & 0 \\
m_{1} & B'(u_{1})\\
\end{smallmatrix} \right]\!\circ\! (\alpha_{T}\amalg \beta_{U}) \!\Big)\left[
\begin{smallmatrix}
x \\
\\
y
\end{smallmatrix} \right] &\!\!=\!\!
\left[ \begin{smallmatrix}
A'(t_{1}) & 0 \\
m_{1} & B'(u_{1})\\
\end{smallmatrix} \right]\left[
\begin{smallmatrix}
\alpha_{T}(x)\\
\beta_{U}(y)
\end{smallmatrix} \right]
\\
&\!\!= \!\!\Big( A'(t_{1})\big(\alpha_{T}(x)\big),
m_{1}\cdot (\alpha_{T}(x))+B'(u_{1})\big(\beta_{U}(y)\big)\Big)
\end{align*}
Let us see that $\beta_{U'}(m_{1}\cdot x)=m_{1}\cdot (\alpha_{T}(x))$.\\
Since $m_{1}\in M(U',T)$, $x\in A(T)$ we have that $m_{1}\cdot x=[f_{T}(x)]_{U'}(m_{1})\in B(U')$ (see \ref{Tlemma1}). Then $\beta_{U'}(m_{1}\cdot x)=\beta_{U'}\Big([f_{T}(x)]_{U'}(m_{1})\Big)\in B'(U').$\\
On the other hand, since $m_{1}\in M(U',T)$, $\alpha_{T}(x)\in A'(T)$ we have that $m_{1}\cdot (\alpha_{T}(x))=[f'_{T}(\alpha_{T}(x))]_{U'}(m_{1})$ (see \ref{Tlemma1}).
Consider the following commutative diagram in $\mathsf{Mod}(\mathcal{T})$ (this is because  $(\alpha,\beta):(A,f,B)\rightarrow (A',f',B')$)

\[
\begin{diagram}
\node{A} \arrow{e,t}{\alpha}\arrow{s,l}{f}\node{ A'}\arrow{s,r}{f'}\\
\node{\mathbb{G}(B)} \arrow{e,b}{\mathbb{G}(\beta)}\node{\mathbb{G}(B').}
\end{diagram}
\]
Then,  for  each $T\in \mathcal{T}$, the following diagram commutes in $\mathbf{Ab}$
\[
\begin{diagram}
\node{A(T)} \arrow{e,t}{\alpha_{T}}\arrow{s,l}{f_{T}}\node{A'(T)}\arrow{s,r}{f'_{T}}\\
\node{\mathbb{G}(B)(T)} \arrow{e,b}{\mathbb{G}(\beta)_{T}}\node{\mathbb{G}(B')(T).}
\end{diagram}
\]
Then, we have that $f'_{T}(\alpha_{T}(x)):M_{T}\longrightarrow B'$ coincides with
$(\mathbb{G}(\beta)_{T})(f_{T}(x))=\mathrm{Hom}_{\mathsf{Mod}(\mathcal{U})}(M_{T},\beta)(f_{T}(x))=
\beta\circ f_{T}(x)$. In particular for $U'$ we have that
$$[\beta]_{U'}\circ [f_{T}(x)]_{U'}=[f'_{T}(\alpha_{T}(x))]_{U'}:M(U',T)\longrightarrow B'(U').$$
Hence, for $m_{1}\in M(U',T)$ we have that
$$\Big([\beta]_{U'}\circ [f_{T}(x)]_{U'}\Big)(m_{1})=[f'_{T}(\alpha_{T}(x))]_{U'}(m_{1}).$$
This means precisely that $\beta_{U'}(m_{1}\cdot x)=m_{1}\cdot (\alpha_{T}(x))$.\\
On the other hand, since $\alpha:A\longrightarrow A'$ and $\beta:B\longrightarrow B'$ are
natural transformations, for $T\in \mathcal{T}$ and $U\in \mathcal{U}$, we have that $\alpha_{T'}(A(t_{1})(x))=A'(t_{1})(\alpha_{T}(x))$ and
$(B'(u_{1}))(\beta_{U}(y))=\beta_{U'}(B(u_{1})(y))$. Proving that
$$\Big (\!\alpha_{T'}(A(t_{1}(x))),\beta_{U'}\Big(
m_{1}\cdot x+B(u_{1})(y)\!\Big)\!\Big)\!\!=\!\!\Big ( \!A'(t_{1})(\alpha_{T}(x)),
m_{1}\cdot (\alpha_{T}(x))+B'(u_{1})\big(\beta_{U}(y)\!\big)\!\Big).$$
This proves that $\alpha\amalg \beta:A\amalg_{f}B\longrightarrow A'\amalg_{f'}B'$
is a natural transformation.
\end{proof}

\begin{proposition}
The assignment
$\textswab{F}:\Big( \mathsf{Mod}(\mathcal{T}),\mathbb{G}\mathsf{Mod}(\mathcal{U})\Big)\longrightarrow \mathrm{Mod}(\mathbf\Lambda)$
is a functor.
\end{proposition}
\begin{proof}
It is straightforward.
\end{proof}

\begin{lemma}\label{Tlemma3}
Let $(A,f,B),(A',f',B')\in \Big( \mathsf{Mod}(\mathcal{T}),\mathbb{G}\mathsf{Mod}(\mathcal{U})\Big)$ be.
The map 
$$\textswab{F}:\mathrm{Hom}\Big((A,f,B),(A',f',B')\Big)\rightarrow \mathrm{Hom}_{\mathsf{Mod}(\mathbf{\Lambda})}(A\amalg _f B, A'\amalg _{f'} B')$$
is bijective. That is, $\textswab{F}$ is full and  faithful.
\end{lemma}
\begin{proof}
Firstly, we will see that $\textswab{F}$ is surjective.\\
Let $(A,f,B),(A',f',B')\in \Big( \mathsf{Mod}(\mathcal{T}),\mathbb{G}\mathsf{Mod}(\mathcal{U})\Big)$ and let $S:A\amalg_{f} B\longrightarrow A'\amalg_{f'} B'$
be a morphism in $\mathrm{Mod}(\mathbf{\Lambda})$ whose components are:
$$S=\left\{S_{\left[\begin{smallmatrix}
T& 0 \\
M & U \\
  \end{smallmatrix} \right]}:A(T)\amalg B(U)\rightarrow A'(T)\amalg B'(U)\right\}_{\left[
\begin{smallmatrix}
T&0\\
M&U
\end{smallmatrix}
\right] \in \mathbf{\Lambda}.}$$
For $T\in \mathcal{T}$, consider the object $\left[\begin{smallmatrix}
T& 0 \\
M & 0 \\
\end{smallmatrix} \right]\in \mathbf{\Lambda}$. Then we have the
morphism
$$S_{\left[\begin{smallmatrix}
T& 0 \\
M & 0 \\
\end{smallmatrix} \right]}:A(T)\amalg 0\rightarrow A'(T)\amalg 0.$$
Therefore,  for $(x,0)\in A(T)\amalg 0$ we have that
$S_{\left[\begin{smallmatrix}
T& 0 \\
M & 0 \\
\end{smallmatrix} \right]}(x,0)=(x',0),$ for some $x'\in A'(T)$.\\
Hence, we define $\alpha_{T}:A(T)\longrightarrow A'(T)$ as
$\alpha_{T}(x):=x'$ for $x\in A(T)$.
It is straightforward  to show that $\alpha=\{\alpha_{T}:A(T)\longrightarrow A'(T)\}_{T\in \mathcal{T}}$ is a morphism of $\mathcal{T}$-modules.\\
Now, for $U\in \mathcal{U}$ we consider the morphism
$$S_{\left[\begin{smallmatrix}
0& 0 \\
M & U \\
\end{smallmatrix} \right]}:0\amalg B(U)\rightarrow 0 \amalg B'(U).$$
Then, we define $\beta_{U}:B(U)\longrightarrow B'(U)$ as follows:
for $y\in B(U)$ we have that $\beta_{U}(y):=y'\in B'(U)$ is such that $S_{\left[\begin{smallmatrix}
0& 0 \\
M & U \\
\end{smallmatrix} \right]}(0,y)=(0,y')$.
Similarly, it can be proved that $\beta=\{\beta_{U}:B(U)\longrightarrow B'(U)\}$ is a
morphism of $\mathcal{U}$-modules.\\
Therefore, for $T\in \mathcal{T}$ and $U\in \mathcal{U}$ we define
$\alpha_{T}\amalg \beta_{U}:A(T)\amalg B(U)\longrightarrow A'(T)\amalg B'(U)$
as $(\alpha_{T}\amalg \beta_{U})(x,y):=(\alpha_{T}(x),\beta_{U}(y))$.
In this way, we obtain a family of morphisms in $\mathbf{Ab}$ as follows
$$\alpha\amalg \beta:=\Big \{(\alpha_{T}\amalg \beta_{U}):A(T)\amalg B(U)\rightarrow A'(T)\amalg B'(U)\Big \}_{\left[
\begin{smallmatrix}
T&0\\
M&U
\end{smallmatrix}
\right] \in \mathbf{\Lambda}.}$$
We assert that
$S_{\left[\begin{smallmatrix}
T& 0 \\
M & U\\
\end{smallmatrix} \right]}=\alpha_{T}\amalg \beta_{U}$. Indeed,
for $1_{T}:T\longrightarrow T$ in $\mathcal{T}$ we have the morphism
$$\left[\begin{smallmatrix}
1_{T}& 0 \\
M & 0\\
\end{smallmatrix} \right]:\left[\begin{smallmatrix}
T& 0 \\
M & U\\
\end{smallmatrix} \right]\longrightarrow
\left[\begin{smallmatrix}
T& 0 \\
M & 0\\
\end{smallmatrix} \right] \quad \text{in}\,\, \mathbf{\Lambda}.$$
Then we have the following commutative diagram in $\mathbf{Ab}$

\[
\begin{diagram}
\node{A(T)\amalg B(U)}\arrow{e,t}{S_{{}_{\left[
\begin{smallmatrix}
T&0\\
M&U
\end{smallmatrix}
  \right]}}}\arrow{s,l}{ \left[
\begin{smallmatrix}
A(1_{T})&0\\
0&0
\end{smallmatrix}
  \right]}
  \node{A'(T)\amalg B'(U)}\arrow{s,r}{ \left[
\begin{smallmatrix}
A'(1_{T})&0\\
0&0
\end{smallmatrix}
  \right]}\\
   \node{A(T)\amalg 0}\arrow{e,b}{S_{{}_{\left[
\begin{smallmatrix}
T&0\\
M&0
\end{smallmatrix}
  \right]}}}\node{A'(T)\amalg 0.}
\end{diagram}
\]
\\
\\
For $(x,y)\in A(T)\amalg B(U)$ we consider
$(x',y')=S_{\left[\begin{smallmatrix}
T& 0 \\
M & U \\
\end{smallmatrix} \right]}(x,y)$ then we have that

$$\left[\begin{smallmatrix}
1_{A'(T)}& 0 \\
0 & 0 \\
\end{smallmatrix} \right]S_{\left[\begin{smallmatrix}
T& 0 \\
M & U \\
\end{smallmatrix} \right]}(x,y)=S_{\left[\begin{smallmatrix}
T& 0 \\
M & 0 \\
\end{smallmatrix} \right]}\left[\begin{smallmatrix}
1_{A(T)} & 0 \\
0 & 0 \\
\end{smallmatrix} \right](x,y).$$
It follows that $(x',0)=(\alpha_{T}(x),0)$. Similarly, we have that $(0,y')=(0,\beta_{U}(y))$ and then, we conclude that
$$S_{\left[\begin{smallmatrix}
T& 0 \\
M & U \\
\end{smallmatrix} \right]}(x,y)=(x',y')=(\alpha_{T}(x),\beta_{U}(y))=(\alpha_{T}\amalg \beta_{U})(x,y).$$
Proving that $S_{\left[\begin{smallmatrix}
T& 0 \\
M & U \\
\end{smallmatrix} \right]}=(\alpha_{T}\amalg \beta_{U})$. Therefore, we have that
$S=\alpha\amalg \beta$.
\\
Now, let us check that $(\alpha,\beta)$ is a morphism from $(A,f,B)$ to $(A',f',B')$. We have to show that for $T\in \mathcal{T}$ the following diagram commutes
\[
\begin{diagram}
\node{A(T)} \arrow{e,t}{\alpha_{T}}\arrow{s,l}{f_{T}}\node{A'(T)}\arrow{s,r}{f'_{T}}\\
\node{\mathbb{G}(B)(T)} \arrow{e,b}{\mathbb{G}(\beta)_{T}}\node{\mathbb{G}(B')(T).}
\end{diagram}
\]
Then, for $x\in A(T)$ we have to show  that $f'_{T}(\alpha_{T}(x)):M_{T}\longrightarrow B'$ coincides with
$(\mathbb{G}(\beta)_{T})(f_{T}(x))=\mathrm{Hom}_{\mathsf{Mod}(\mathcal{U})}(M_{T},\beta)(f_{T}(x))=
\beta\circ f_{T}(x):M_{T}\longrightarrow B'$. In particular, we have to show that for $U\in \mathcal{U}$, the following equality holds
$$[\beta]_{U}\circ [f_{T}(x)]_{U}=[f'_{T}(\alpha_{T}(x))]_{U}.$$
Since $S:A\amalg_{f}B\longrightarrow A'\amalg_{f'}B'$ is a morphism in  $\mathsf{Mod}(\mathbf{\Lambda})$,
for every morphism $\left[ \begin{smallmatrix}
1_{T} & 0 \\
m & 1_{U}
\end{smallmatrix} \right]: \left[ \begin{smallmatrix}
T & 0 \\
M & U
\end{smallmatrix} \right] \longrightarrow \left[ \begin{smallmatrix}
T & 0 \\
M & U
\end{smallmatrix} \right]$ in $\mathbf{\Lambda}$ with $m\in M(U,T)$ we have the following commutative diagram in $\mathbf{Ab}$
$$\xymatrix{A(T)\amalg B(U)\ar[rr]^{\alpha_{T}\amalg \beta_{U}}\ar[d]_{{\left[ \begin{smallmatrix}
1_{A(T)} & 0 \\
m & 1_{B(U)}\\
\end{smallmatrix} \right]}}  & & A'(T)\amalg B'(U)\ar[d]^{{\left[ \begin{smallmatrix}
1_{A'(T)} & 0 \\
m & 1_{B'(U)}\\
\end{smallmatrix} \right]}}\\
A(T)\amalg B(U)\ar[rr]^{\alpha_{T}\amalg \beta_{U}}
 & & A'(T)\amalg B'(U).}$$
Hence, for each $(x,y)\in A(T)\amalg B(U)$ we obtain that
\begin{eqnarray*}
\Big((\alpha_{T}\amalg \beta_{U})\circ \left[ \begin{smallmatrix}
1_{A(T)} & 0 \\
m & 1_{B(U)}\\
\end{smallmatrix} \right]\Big)\left[
\begin{smallmatrix}
x \\
\\
y
\end{smallmatrix} \right] &= &(\alpha_{T}\amalg \beta_{U})
\Big(1_{A(T)}(x),m\cdot x+(1_{B(U)})(y)\Big)\\
&= &\Big(\alpha_{T}(x),\beta_{U}(
m\cdot x+ y)\Big )
\end{eqnarray*}
On the other hand, we have that
\begin{eqnarray*}
\Big(\left[ \begin{smallmatrix}
1_{A'(T)} & 0 \\
m & 1_{B'(U)}\\
\end{smallmatrix} \right]\circ (\alpha_{T}\amalg \beta_{U})\Big)\left[
\begin{smallmatrix}
x \\
\\
y
\end{smallmatrix} \right] &=&
\left[ \begin{smallmatrix}
1_{A'(T)} & 0 \\
m & 1_{B'(U)}\\
\end{smallmatrix} \right]\left[
\begin{smallmatrix}
\alpha_{T}(x)\\
\beta_{U}(y)
\end{smallmatrix} \right]
\\
&= &\Big ( \alpha_{T}(x),
m \cdot (\alpha_{T}(x))+\beta_{U}(y)\Big).
\end{eqnarray*}
Then, we have that
$\Big(\alpha_{T}(x),\beta_{U}(
m\cdot x+ y)\Big )=\Big ( \alpha_{T}(x),
m\cdot (\alpha_{T}(x))+\beta_{U}(y)\Big).$
Therefore, since $\beta_{U}$ is a morphism of abelian groups, we have that
$$\beta_{U}(m \cdot x)=m \cdot (\alpha_{T}(x))$$
for $m\in M(U,T)$. This means that
$$\Big(\beta_{U}\circ [f_{T}(x)]_{U}\Big)(m)=[f'_{T}(\alpha_{T}(x))]_{U}(m) \,\,\,\,\,\, \forall m\in M(U,T).$$ That is, we have that
$$\beta_{U}\circ [f_{T}(x)]_{U}=[f'_{T}(\alpha_{T}(x))]_{U}:M_{T}(U)\longrightarrow B'(U)$$
for each $U\in \mathcal{U}$. Therefore, we have that
$$f'_{T}(\alpha_{T}(x))=\beta\circ f_{T}(x)=\mathrm{Hom}_{\mathsf{Mod}(\mathcal{U})}(M_{T},\beta)(f_{T}(x))
=(\mathbb{G}(\beta)_{T})(f_{T}(x))\quad \forall x\in A(T).$$
Therefore $(\alpha,\beta):(A,f,B)\longrightarrow (A',f',B')$ is a morphism in $\Big( \mathsf{Mod}(\mathcal{T}),\mathbb{G}\mathsf{Mod}(\mathcal{U})\Big)$
and $\textswab{F}(\alpha,\beta)=S$, proving that $\textswab{F}$ is surjective.\\
Now, let suppose that $(\alpha',\beta'),(\alpha,\beta):(A,f,B)\longrightarrow (A',f',B')$  are morphisms
such that and $\textswab{F}(\alpha,\beta)=\textswab{F}(\alpha',\beta')$. Then for each $T\in \mathcal{T}$, $U\in \mathcal{U}$
we have that $\alpha_{T}\amalg \beta_{U}=\alpha'_{T}\amalg \beta'_{U}$. This implies, that
$\alpha_{T}=\alpha'_{T}$ and $\beta_{U}=\beta'_{U}$ and therefore $\alpha=\alpha'$ and
$\beta=\beta'$. Proving that $\textswab{F}$ is injective and therefore $\textswab{F}$ is full and faithful.
\end{proof}

We can define a functor
$I_{1}:\mathcal{T} \longrightarrow \mathbf{\Lambda}$ defined as follows
$I_{1}(T):=\left[
\begin{smallmatrix}
T & 0 \\
M & 0 \\
\end{smallmatrix}
\right]$ and for a morphism $t:T\longrightarrow T'$ in $\mathcal{T}$ de define
$$I_{1}(t)=\left[
\begin{smallmatrix}
t & 0 \\
0 & 0 \\
\end{smallmatrix}
\right]:
\left[
\begin{smallmatrix}
T& 0 \\
M & 0 \\
\end{smallmatrix}
\right]\rightarrow \left[
\begin{smallmatrix}
T'& 0 \\
M & 0 \\
\end{smallmatrix}
\right]$$
In the same way, we define a functor $I_{2}:\mathcal{U}\longrightarrow \mathbf{\Lambda}$.
Then we have the induced morphisms
$$\mathbb{I}_{1}:\mathsf{Mod}(\mathbf{\Lambda})\longrightarrow \mathsf{Mod}(\mathcal{T})$$
$$\mathbb{I}_{2}:\mathsf{Mod}(\mathbf{\Lambda})\longrightarrow \mathsf{Mod}(\mathcal{U})$$
Let $C$ be a $\mathbf{\Lambda}$-module, we denote by
$C_{1}=\mathbb{I}_{1}(C)=C\circ I_{1}:\mathcal{T}\longrightarrow \mathbf{Ab}$ and $C_{2}=
\mathbb{I}_{2}(C)=C\circ I_{2}:\mathcal{U}\longrightarrow \mathbf{Ab}.$

\begin{lemma}\label{lemmorcom}
Let $C$ be a $\mathbf{\Lambda}$-module. Then, there exists a
morphism of $\mathcal{T}$-modules
$$f:C_{1}\longrightarrow \mathbb{G}(C_{2}).$$
\end{lemma}
\begin{proof}
Let $C$ be a $\mathbf{\Lambda}$-module and consider
$T\in\mathcal T$ and $U\in\mathcal U$. For all $m\in M(U,T)$
we have a morphism $\overline{m}:=\left[
\begin{smallmatrix}
0& 0 \\
m & 0 \\
\end{smallmatrix}
\right]:
\left[
\begin{smallmatrix}
T& 0 \\
M & 0 \\
\end{smallmatrix}
\right]\rightarrow
\left[
\begin{smallmatrix}
0& 0 \\
M & U \\
\end{smallmatrix}
\right]$.
We note that $\left[
\begin{smallmatrix}
T& 0 \\
M & 0 \\
\end{smallmatrix}
\right]=I_{1}(T)$ and $\left[
\begin{smallmatrix}
0& 0 \\
M & U \\
\end{smallmatrix}
\right]=I_{2}(U)$. 
Applying the $\mathbf{\Lambda}$-module $C$  to $m$ yields a morphism of abelian groups
$$C(\overline{m}):C_{1}(T)=C(I_{1}(T))\longrightarrow C_{2}(U)=C(I_{2}(U)).$$
We assert that $C(\overline{m})$ induces a morphism of abelian groups
$$f_{T}:C_{1}(T)\longrightarrow \mathbb{G}(C_{2})(T)=\mathrm{Hom}_{\mathsf{Mod}(\mathcal{U})}(M_{T},C_{2}).$$
Indeed, for $x\in C_{1}(T) $ we define $f_{T}(x):M_{T}\longrightarrow C_{2}$ such that
$$f_{T}(x)=\left\lbrace [f_{T}(x)]_{U}:M_{T}(U)\longrightarrow C_{2}(U) \right\rbrace _{U\in \mathcal{U}}$$
where $[f_{T}(x)]_{U}(m):=C(\overline{m})(x)$ for all $m\in M_{T}(U)$. It is straightforward to check that
$f_{T}(x):M_{T}\longrightarrow C_{2}$ is a natural transformation. Now, let $ t\in \mathsf{Hom}_{\mathcal{T}}(T,T')$, we assert that the following diagram commutes
\[
\begin{diagram}
\node{C_{1}(T)}\arrow{e,t}{f_{T}}\arrow{s,l}{C_{1}(t)}
 \node{\mathbb{G}(C_{2})(T)}\arrow{s,r}{\mathbb{G}(C_{2})(t)}\\
\node{C_{1}(T')}\arrow{e,b}{f_{T'}}
 \node{\mathbb{G}(C_{2})(T')}
\end{diagram}
\]
Indeed, let $x\in C_{1}(T) $ then
\begin{eqnarray*}
\Big(\mathbb{G}(C_{2})(t)\circ f_{T}\Big)(x) =\Big(\mathsf{Hom}_{\mathsf{Mod}(\mathcal{U})}(\bar{t},C_{2}) \circ f_{T}\Big)(x)\
= f_{T}(x)\circ \bar{t}.
\end{eqnarray*}
Then, for $U\in \mathcal{U}$ and $m'\in M(U,T')$ we have that
\begin{align*}
\Big([f_{T}(x)]_{U}\circ[\bar{t}]_{U}\Big)(m')=[f_{T}(x)]_{U}(M(1_{U}\otimes t^{op})(m')) & =
[f_{T}(x)]_{U}(m'\bullet t)\\
& =C(\overline{m'\bullet t})(x)\\
& =C\Big(\left[
\begin{smallmatrix}
0&0\\
m'\bullet t &0
\end{smallmatrix}
\right]\Big)(x).
\end{align*}
On the other hand, $(f_{T'}\circ C_{1}(t))(x)=f_{T'}(C_{1}(t)(x))
=f_{T'}\Big(C\Big(\left[
\begin{smallmatrix}
t&0\\
0 &0
\end{smallmatrix}
\right]\Big)(x)\Big )$. Hence, for $U\in \mathcal{U}$ and $m'\in M(U,T')$ we have that

\begin{eqnarray*}
\Big[f_{T'}\Big(C\Big(\left[
\begin{smallmatrix}
t&0\\
0 &0
\end{smallmatrix}
\right]\Big)(x)\Big )\Big]_{U}(m') & = & C(\overline{m'})\Big(C\Big(\left[
\begin{smallmatrix}
t&0\\
0 &0
\end{smallmatrix}
\right]\Big)(x)\Big)\\
& = & \Big(C\Big(\left[
\begin{smallmatrix}
0&0\\
m' &0
\end{smallmatrix}
\right]\Big)C\Big(\left[
\begin{smallmatrix}
t&0\\
0 &0
\end{smallmatrix}
\right]\Big)\Big)(x)\\
& = & C\Big(\left[
\begin{smallmatrix}
0&0\\
m'\bullet t &0
\end{smallmatrix}
\right]\Big)(x).
\end{eqnarray*}
Therefore $f:C_{1}\longrightarrow \mathbb{G}(C_{2})$ is a morphism of $\mathcal{T}$-modules.
Hence $\left(C_{1},f,C_{2} \right)\in \Big(\mathsf{Mod}(\mathcal{T}),\mathbb{G}\mathsf{Mod}(\mathcal{U}) \Big)$.
\end{proof}

\begin{note}\label{notaC1C2}
Let $C$ be a $\mathbf{\Lambda}$-module and consider the morphism $f:C_{1}\longrightarrow \mathbb{G}(C_{2})$ constructed in \ref{lemmorcom}.  Since $\left(C_{1},f,C_{2} \right)\in \Big(\mathsf{Mod}(\mathcal{T}),\mathbb{G}\mathsf{Mod}(\mathcal{U}) \Big)$,  we can construct $C_{1}\amalg_{f}C_{2}$.
We recall that 
\begin{enumerate}
\item [(a)] For
$\left[\begin{smallmatrix}
T& 0 \\
M & U \\
\end{smallmatrix} \right]\in\mathbf{\Lambda}$ we have that $\Big(C_{1}\amalg_f C_{2}\Big)(\left[\begin{smallmatrix}
T& 0 \\
M & U \\
\end{smallmatrix} \right]) :=C_{1}(T)\amalg  C_{2}(U)\in \mathbf{Ab}.$

\item [(b)] If $\left[\begin{smallmatrix}
t& 0 \\
m& u \\
\end{smallmatrix} \right]\in \mathsf{Hom}_{\mathbf{\Lambda}}(\left[\begin{smallmatrix}
T& 0 \\
M & U \\
\end{smallmatrix} \right], \left[\begin{smallmatrix}
T'& 0 \\
M & U' \\
\end{smallmatrix} \right])=\left[ \begin{smallmatrix}
\mathsf{Hom}_{\mathcal{T}}(T,T') & 0 \\
M(U',T) & \mathsf{Hom}_{\mathcal{U}}(U,U')
\end{smallmatrix} \right]$ we have that
$$\Big(C_{1}\amalg_f C_{2}\Big)(\left[\begin{smallmatrix}
t& 0 \\
m& u \\
\end{smallmatrix} \right]):=\left [\begin{smallmatrix}
C_{1}(t)& 0 \\
m & C_{2}(u) \\
\end{smallmatrix} \right]: C_{1}(T)\amalg C_{2}(U)\rightarrow C_{1}(T')\amalg C_{2}(U')$$
 is given by $\left [\begin{smallmatrix}
C_{1}(t)& 0 \\
m& C_{2}(u) \\
\end{smallmatrix} \right]\left [\begin{smallmatrix}
x \\
\\
y \\
\end{smallmatrix} \right]=\left [\begin{smallmatrix}
C_{1}(t)(x) \\
m\cdot x+ C_{2}(u)(y) \\
\end{smallmatrix} \right]$
for $(x,y)\in C_{1}(T)\amalg C_{2}(U)$, where $m\cdot x:=[f_{T}(x)]_{U'}(m)\in C_{2}(U')$.
\end{enumerate}
By the construction of the morphism $f:C_{1}\longrightarrow \mathbb{G}(C_{2})$ in \ref{lemmorcom}, we have that 
$$m\cdot x:=[f_{T}(x)]_{U'}(m):=C(\overline{m})(x) \in C_{2}(U'),$$
where $\overline{m}:=\left[
\begin{smallmatrix}
0& 0 \\
m & 0 \\
\end{smallmatrix}
\right]:
\left[
\begin{smallmatrix}
T& 0 \\
M & 0 \\
\end{smallmatrix}
\right]\rightarrow
\left[
\begin{smallmatrix}
0& 0 \\
M & U' \\
\end{smallmatrix}
\right]$.

\end{note}

\begin{lemma}\label{lemma3}
Let $C$ be a $\mathbf{\Lambda}$-module. Then
$C_1\underset{f}\amalg C_2\cong C$.
\end{lemma}

\begin{proof}
(i) Let $T\in\mathcal T$ and $U\in\mathcal U$.  We have sequences of morphism in $\mathbf{\Lambda}$.
$$\left[ {\begin{array}{cc}
T& 0 \\
M& 0 \\
\end{array} } \right]\xrightarrow{\lambda_T:=\left[ {\begin{smallmatrix}
1_T& 0 \\
0& 0 \\
\end{smallmatrix} } \right]} \left[ {\begin{array}{cc}
T& 0 \\
M& U \\
\end{array} } \right]\xrightarrow{\rho_T:=\left[ {\begin{smallmatrix}
1_T& 0 \\
0& 0 \\
\end{smallmatrix} } \right]} \left[ {\begin{array}{cc}
T& 0 \\
M& 0 \\
\end{array} } \right]$$

$$\left[ {\begin{array}{cc}
0& 0 \\
M& U \\
\end{array} } \right]\xrightarrow{\lambda_U:=\left[ {\begin{smallmatrix}
0& 0 \\
0& 1_U\\
\end{smallmatrix} } \right]} \left[ {\begin{array}{cc}
T& 0 \\
M& U \\
\end{array} } \right]\xrightarrow{\rho_U:=\left[ {\begin{smallmatrix}
0& 0 \\
0& 1_U \\
\end{smallmatrix} } \right]} \left[ {\begin{array}{cc}
0& 0 \\
M& U \\
\end{array} } \right]$$

We consider the maps
$$C(\lambda_{T}):=C\Big(\left[
\begin{smallmatrix}
1_{T}&0\\
0 &0
\end{smallmatrix}
\right]\Big): C\Big(\left[
\begin{smallmatrix}
T&0\\
M & 0
\end{smallmatrix}
\right]\Big)=C_{1}(T)\longrightarrow
C\Big(\left[
\begin{smallmatrix}
T &0\\
M & U
\end{smallmatrix}
\right]\Big)$$

$$C(\lambda_{U}):=C\Big(\left[
\begin{smallmatrix}
0 &0\\
0 & 1_{U}
\end{smallmatrix}
\right]\Big): C\Big(\left[
\begin{smallmatrix}
0&0\\
M & U
\end{smallmatrix}
\right]\Big)=C_{2}(U)\longrightarrow
C\Big(\left[
\begin{smallmatrix}
T &0\\
M & U
\end{smallmatrix}
\right]\Big).$$
This produces a map 
$$\phi_{T,U}:C_{1}(T)\amalg C_{2}(U)=\Big(C_{1}\amalg_{f}C_{2}\Big)\Big(\left[
\begin{smallmatrix}
T &0\\
M & U
\end{smallmatrix}
\right]\Big)\longrightarrow
C\Big(\left[
\begin{smallmatrix}
T &0\\
M & U
\end{smallmatrix}
\right]\Big)$$
defined as $\phi_{T,U}(x,y)=C(\lambda_{T})(x)+C(\lambda_{U})(y)$ $\forall (x,y)\in C_{1}(T)\amalg C_{2}(U)$.\\
It is routine to check that
$$\phi=\!\left\{\!\!
\phi_{{}_{\left[
\begin{smallmatrix}
T& 0 \\
M& U\\
\end{smallmatrix}
\right] }}\!\!:=\!\phi_{T,U}\!\!:\!\! \Big(C_{1}\amalg_{f} C_{2}\Big) (\left [\begin{smallmatrix}
T& 0 \\
M& U\\
\end{smallmatrix} \right])\rightarrow  C\Big(\left [\begin{smallmatrix}
T& 0 \\
M& U\\
\end{smallmatrix} \right]\Big)\!\!\right\}_{\!\!\left[\begin{smallmatrix}
T& 0 \\
M& U\\
\end{smallmatrix} \right ]\in \mathbf{\Lambda}}$$
 is a morphism of $\mathbf{\Lambda}$-modules $\phi:C_{1}\amalg_{f} C_{2}\longrightarrow C$.
Now, we consider the maps
$$C(\rho_{T}):=C\Big(\left[
\begin{smallmatrix}
1_{T}&0\\
0 &0
\end{smallmatrix}
\right]\Big): C\Big(\left[
\begin{smallmatrix}
T&0\\
M & U
\end{smallmatrix}
\right]\Big)\longrightarrow
C_{1}(T)=C\Big(\left[
\begin{smallmatrix}
T &0\\
M & 0
\end{smallmatrix}
\right]\Big)$$

$$C(\rho_{U}):=C\Big(\left[
\begin{smallmatrix}
0 &0\\
0 & 1_{U}
\end{smallmatrix}
\right]\Big): C\Big(\left[
\begin{smallmatrix}
T &0\\
M & U
\end{smallmatrix}
\right]\Big)\longrightarrow
C_{2}(U)=C\Big(\left[
\begin{smallmatrix}
0 &0\\
M & U
\end{smallmatrix}
\right]\Big).$$
This produces a map 
$$\psi_{T,U}:C\Big(\left[
\begin{smallmatrix}
T &0\\
M & U
\end{smallmatrix}
\right]\Big)\longrightarrow 
C_{1}(T)\amalg C_{2}(U)=\Big(C_{1}\amalg_{f}C_{2}\Big)\Big(\left[
\begin{smallmatrix}
T &0\\
M & U
\end{smallmatrix}
\right]\Big)$$
such that
$$\psi_{T,U}(z)=\Big(C(\rho_{T})(z),C(\rho_{U})(z)\Big).$$
It is straightforward to check that
$$\psi=\!\left\{\!\!
\psi_{{}_{\left[
\begin{smallmatrix}
T& 0 \\
M& U\\
\end{smallmatrix}
\right] }}\!\!:=\!\psi_{T,U}:
 C\Big(\left [\begin{smallmatrix}
T& 0 \\
M& U\\
\end{smallmatrix} \right]\Big)\rightarrow 
\Big(C_{1}\amalg_{f} C_{2}\Big) (\left [\begin{smallmatrix}
T& 0 \\
M& U\\
\end{smallmatrix} \right]) \!\!\right\}_{\!\!\left[\begin{smallmatrix}
T& 0 \\
M& U\\
\end{smallmatrix} \right ]\in \mathbf{\Lambda}}$$
is a natural transformation
$\psi:C\longrightarrow C_{1}\amalg_{f}C_{2}$.\\
Finally we will see that $\phi$ is invertible with inverse $\psi$.
Indeed, let $T\in \mathcal{T}$ and $U\in \mathcal{U}$. For $(x,y)\in C_{1}(T)\amalg C_{2}(U)$ we have
 
\begin{eqnarray*}
& & \!\!\!\!\!\!\!\!\!\!\!\!\!\!\!\!\!\! \Big(\psi_{T,U}\circ\phi_{T,U}\Big)(x,y)=\\
& = &\!\!\!\!\! \psi_{T,U}\Big(C(\lambda_{T})(x)+C(\lambda_{U})(y)\Big)\\
& = & \!\!\!\!\! \Big(C(\rho_{T})\Big(C(\lambda_{T})(x)+C(\lambda_{U})(y)\Big),\,\,\,\,
C(\rho_{U})\Big(C(\lambda_{T})(x)+C(\lambda_{U})(y)\Big)\Big)\\
& = & \!\!\!\!\! \Big(C(\rho_{T})C(\lambda_{T})(x)+C(\rho_{T})C(\lambda_{U})(y),\,\,\,\,
C(\rho_{U})C(\lambda_{T})(x)+C(\rho_{U})C(\lambda_{U})(y)\Big)\\
& = & \!\!\!\!\! \Big(C(\rho_{T}\lambda_{T})(x)+C(\rho_{T}\lambda_{U})(y),\,\,\,\, C(\rho_{U}\lambda_{T})(x)+C(\rho_{U}\lambda_{U})(y)\Big)\\
& = & \!\!\!\!\! (x,y),
\end{eqnarray*}
since $\rho_{T}\lambda_{T}=1_{\left[\begin{smallmatrix}
T &0\\
M & 0
\end{smallmatrix}
\right]}$, $\rho_{U}\lambda_{U}=1_{\left[\begin{smallmatrix}
0 &0\\
M & U
\end{smallmatrix}
\right]}$, $\rho_{U}\lambda_{T}=0$ and $\rho_{T}\lambda_{U}=0$.
Therefore $\psi\circ \phi=1_{C_{1}\amalg_{f} C_{2}}$.
\\
On the other side,  for $z\in C\Big({\left[\begin{smallmatrix}
T &0\\
M & U
\end{smallmatrix}
\right]}\Big)$ we have that
\begin{eqnarray*}
\Big(\phi_{T,U}\circ\psi_{T,U}\Big)(z)& =& \phi_{T,U}\Big(C(\rho_{T})(z),C(\rho_{U})(z)\Big)\\
& = & C(\lambda_{T})\Big(C(\rho_{T})(z)\Big)+C(\lambda_{U})\Big(C(\rho_{U})(z)\Big)\\
& = & C(\lambda_{T}\rho_{T}+\lambda_{U}\rho_{U})(z)
\end{eqnarray*}
$$\text{But},\quad\,\,\, \lambda_{T}\rho_{T}=\left[\begin{smallmatrix}
1_{T} &0\\
0 & 0
\end{smallmatrix}
\right]: \left[\begin{smallmatrix}
T &0\\
M & U
\end{smallmatrix}
\right]\longrightarrow \left[\begin{smallmatrix}
0 &0\\
M & U
\end{smallmatrix}
\right],\quad \lambda_{U}\rho_{U}=\left[\begin{smallmatrix}
0 &0\\
0 & 1_{U}
\end{smallmatrix}
\right]: \left[\begin{smallmatrix}
T &0\\
M & U
\end{smallmatrix}
\right]\longrightarrow \left[\begin{smallmatrix}
0 &0\\
M & U
\end{smallmatrix}
\right]$$
satisfies that $\lambda_{T}\rho_{T}+\lambda_{U}\rho_{U}= \left[\begin{smallmatrix}
1_{T} &0\\
0 & 1_{U}
\end{smallmatrix}
\right]= 1_{\left[\begin{smallmatrix}
T &0\\
M & U
\end{smallmatrix}
\right]}$. Therefore,  
$C(\lambda_{T}\rho_{T}+\lambda_{U}\rho_{U})(z)=C\Big (1_{\left[\begin{smallmatrix}
T &0\\
M & U
\end{smallmatrix}
\right]}\Big)(z)=z$.
Therefore $\phi\circ \psi=1_{C}$. Then $\psi:C\longrightarrow C_{1}\amalg_{f} C_{2},$ is an isomorphism.
\end{proof}

\begin{theorem}\label{equivalence1}
The functor $\textswab{F}:\Big( \mathsf{Mod}(\mathcal{T}), \mathbb{G}\mathsf{Mod}(\mathcal{U})\Big) \longrightarrow 
\mathrm{Mod}(\mathbf{\Lambda})$ is an equivalence of categories
\end{theorem}
\begin{proof}
It follows from Lemmas \ref{Tlemma3} and \ref{lemma3}.
\end{proof}

\section{Duality functor $\mathbb{D}:\mathrm{mod}(\mathbf{\Lambda})\longrightarrow \mathrm{mod}(\mathbf{\Lambda}^{op})$ }
In this section, we are going to construct a functor $\mathbb{D}:\mathrm{Mod}(\mathbf{\Lambda})\longrightarrow \mathrm{Mod}(\mathbf{\Lambda}^{op})$ and we will describe it when we identify 
$\mathrm{Mod}(\mathbf{\Lambda})$ with $\Big( \mathsf{Mod}(\mathcal{T}), \mathbb{G}\mathsf{Mod}(\mathcal{U})\Big)$. This functor will be useful in the case of dualizing varieties, because we will show that under certain conditions we get a duality
$$\mathbb{D}:\mathrm{mod}(\mathbf{\Lambda})\longrightarrow \mathrm{mod}(\mathbf{\Lambda}^{op}).$$
Let $\mathcal{U}$ and $\mathcal{T}$ additive categories and 
$M\in \mathsf{Mod}(\mathcal{U}\otimes \mathcal{T}^{op})$. We can define $\overline{M}\in \mathsf{Mod}(\mathcal{T}^{op}\otimes \mathcal{U})$ as follows:
\begin{enumerate}
\item [(a)] $\overline{M}(T,U):=M(U,T)$ for all $(T,U)\in \mathcal{T}^{op}\otimes \mathcal{U}.$

\item [(b)] $\overline{M}(\alpha^{op}\otimes \beta):=M(\beta\otimes\alpha^{op})$ for all $\alpha^{op}\otimes \beta:(T', U) \longrightarrow  (T, U')$ where $\alpha:T\longrightarrow T'$ in $\mathcal{T}$ and $\beta:U\longrightarrow U'$ in $\mathcal{U}$.
\end{enumerate}

\begin{proposition}\label{otrosdosfun}
Let $ \overline{M}\in \mathsf{Mod}(\mathcal{T}^{op}\otimes \mathcal{U})$ be. Then, there exists two covariant functors
$$\overline{E}:\mathcal{U}^{op}\longrightarrow \mathsf{Mod}(\mathcal{T}^{op})^{op},\quad \overline{E'}:\mathcal{T}^{op}\longrightarrow  \mathsf{Mod}(\mathcal{U}).$$
\end{proposition}
\begin{proof}
For $U\in \mathcal{U}^{op}$, we define  a covariant functor $\overline{E}(U):=\overline{M}_{U}:\mathcal{T}^{op} \longrightarrow \mathbf{Ab}$  (that is, $\overline{M}_{U}:\mathcal{T} \longrightarrow \mathbf{Ab}$ is a contravariant functor)  as follows.\\
$(i)$ $ \overline{M}_{U}(T):=\overline{M}(T,U) $, for all $ T\in \mathcal{T}^{op}$.\\
$(ii)$  $ \overline{M}_{U}(t^{op}):=\overline{M}(t^{op}\otimes 1_{U}) $, for all $ t^{op}\in  \mathsf{Hom}_{\mathcal{T}^{op}}(T',T).$\\
Now, given a morphism $u^{op}:U'\longrightarrow U$ in $\mathcal{U}^{op}$ we set $\overline{E}(u^{op}):=\overline{u^{op}}:\overline{M}_{U}\longrightarrow \overline{M}_{U'}$  where $ \overline{u^{op}}=\lbrace [\overline{u^{op}}\,]_{T}:\overline{M}_{U}(T)\longrightarrow \overline{M}_{U'}(T)\rbrace_{ T\in \mathcal{T}^{op}}$ with $ [\overline{u^{op}}\,]_{T}=\overline{M}(1_{T}\otimes u).$\\
Similarly for $ T\in \mathcal{T}^{op}$ we define a contravariant functor  $ \overline{E'}(T):=\overline{M}_{T}:\mathcal{U}^{op}\longrightarrow \mathbf{Ab}$ (that is, is a covariant functor  $\overline{M}_{T}:\mathcal{U}\longrightarrow \mathbf{Ab}$).
\end{proof}

We note that $\overline{M}_{U}(t^{op})=\overline{M}(t^{op}\otimes 1_{U})=M(1_{U}\otimes t^{op})$ and $\overline{M}_{T}(u^{op})=\overline{M}(1_{T}\otimes u^{})=M(1_{T}\otimes u).$

\begin{note}
We have a covariant functor $\overline{\mathbb{G}}:\mathsf{Mod}(\mathcal{T}^{op})\longrightarrow \mathsf{Mod}(\mathcal{U}^{op})$.\\
In detail, we have that the following holds.\\
$(i)$ For $B\in \mathsf{Mod}(\mathcal{T}^{op}) $ , $\overline{\mathbb{G}}(B)(U):= \mathsf{Hom}_{\mathsf{Mod}(\mathcal{T}^{op})}(\overline{M}_{U},B)$ for all $ U\in \mathcal{U}^{op} $.\\
Moreover, for all
$u^{op}\in \mathsf{Hom}_{\mathcal{U}^{op}}(U',U) $ we have that $$\overline{\mathbb{G}}(B)(u^{op}):=\mathrm{Hom}_{\mathsf{Mod}(\mathcal{T}^{op})}(\overline{u^{op}},B) .$$\\
$(ii)$ If $\eta:B\rightarrow B' $ is a morphism of $\mathcal{T}^{op} $-modules we have that
$\overline{\mathbb{G}}(\eta):\overline{\mathbb{G}}(B)\longrightarrow \overline{\mathbb{G}}(B')$  is such that 
$$\overline{\mathbb{G}}(\eta)= \Big\{  [\overline{ \mathbb{G}}(\eta)]_{{U}}:=\mathsf{Hom}_{\mathsf{Mod}(\mathcal{T}^{op})}(\overline{M}_{U},\eta):\overline{\mathbb{G}}(B)(U)\longrightarrow
\overline{\mathbb{G}}(B')(U) \Big \} _{U\in \mathcal{U}^{op}} .$$\\
$(iii)$ We note that $\overline{M}_{U}=M_{U}$.
\end{note}

Since we have that $\overline{M}\in \mathsf{Mod}(\mathcal{T}^{op}\otimes \mathcal{U})$. Following the definition \ref{defitrinagularmat}, we have the following.

\begin{definition}
We define the \textbf{triangular matrix category}
$\overline{\mathbf{\Lambda}}=\left[ \begin{smallmatrix}
\mathcal{U}^{op} & 0 \\ \overline{M} & \mathcal{T}^{op}
\end{smallmatrix}\right]$ as follows.
\begin{enumerate}
\item [(a)] The class of objects of this category are matrices $ \left[
\begin{smallmatrix}
U & 0 \\ 
\overline{M} & T
\end{smallmatrix}\right]  $ with $ U\in \mathrm{obj} (\mathcal{U}^{op}) $ and $ T\in \mathrm{obj} (\mathcal{T}^{op})$.

\item [(b)] Given a pair of objects in
$\left[ \begin{smallmatrix}
U' & 0 \\
\overline{M} & T'
\end{smallmatrix} \right] ,  \left[ \begin{smallmatrix}
U & 0 \\
\overline{M} & T
\end{smallmatrix} \right]$ in
$\overline{\mathbf{\Lambda}}$ we define

$$\mathsf{ Hom}_{\overline{\mathbf{\Lambda}}}\left (\left[ \begin{smallmatrix}
U' & 0 \\
\overline{M} & T'
\end{smallmatrix} \right] ,  \left[ \begin{smallmatrix}
U & 0 \\
\overline{M} & T
\end{smallmatrix} \right]  \right)  := \left[ \begin{smallmatrix}
\mathsf{Hom}_{\mathcal{U}^{op}}(U',U) & 0 \\
\overline{M}(T,U') & \mathsf{Hom}_{\mathcal{T}^{op}}(T',T)
\end{smallmatrix} \right]$$.
\end{enumerate}
The composition is given by
\begin{eqnarray*}
\circ&:&\left[
\begin{smallmatrix}
{\mathcal{U}^{op}}(U',U) & 0 \\
\overline{M}(T,U') & {\mathcal{T}^{op}}(T',T)
\end{smallmatrix} \right]\times 
\left[  \begin{smallmatrix}
{\mathcal{U}^{op}}(U'',U') & 0 \\
\overline{M}(T',U'') & {\mathcal{T}^{op}}(T'',T')
\end{smallmatrix}  \right]
\longrightarrow\left[
\begin{smallmatrix}
{\mathcal{U}^{op}}(U'',U) & 0 \\
\overline{M}(T,U'') & {\mathcal{T}^{op}}(T'',T)\end{smallmatrix} \right] \\
&& \left(\left[
\begin{smallmatrix}
u_{1}^{op} & 0 \\
m_{1} & t_{1}^{op}
\end{smallmatrix} \right],
 \left[ \begin{smallmatrix}
u_{2}^{op} & 0 \\
m_{2} & t_{2}^{op}
\end{smallmatrix} \right]
\right)\longmapsto\left[
\begin{smallmatrix}
u_{1}^{op}\circ u_{2}^{op} & 0 \\
m_{1}\bullet u_{2}^{op}+t_{1}^{op}\bullet m_{2} 
& t_{1}^{op}\circ t_{2}^{op}
\end{smallmatrix} \right].
\end{eqnarray*}
\end{definition}

We recall that $ m_{1}\bullet u_{2}^{op}:=\overline{M}(1_{T}\otimes u_{2}^{})(m_{1})$ and
$t_{1}^{op}\bullet m_{2}=\overline{M}(t_{1}^{op}\otimes 1_{U''})(m_{2})$.\\
Since $m_{1}\in \overline{M}(T,U')=M(U',T)$ and $m_{2}\in \overline{M}(T',U'')=M(U'',T')$.\\
Since 
$$\overline{M}(1_{T}\otimes u_{2})=M(u_{2}\otimes 1_{T}):M(U',T)\longrightarrow M(U'',T)$$

$$\overline{M}(t_{1}^{op}\otimes 1_{U''})=M(1_{U''}\otimes t_{1}^{op}):M(U'',T')\longrightarrow M(U'',T).$$
Then we have that
$$m_{1}\bullet u_{2}^{op}+t_{1}^{op}\bullet m_{2}=M(u_{2}\otimes 1_{T})(m_{1})+M(1_{U''}\otimes t_{1}^{op})(m_{2})=u_{2}\bullet m_{1}+m_{2}\bullet t_{1}.$$

\begin{proposition}\label{isotrivial}
There is an isomorphism
$$\mathbb{T}:\mathbf{\Lambda}^{op}\longrightarrow \overline{\mathbf{\Lambda}}, $$
defined as $\mathbb{T}\Big(\left[
\begin{smallmatrix}
T & 0 \\ M & U
\end{smallmatrix}\right]\Big)=\left[
\begin{smallmatrix}
U & 0 \\ 
\overline{M} & T
\end{smallmatrix}\right]$
and for  $\left[
\begin{smallmatrix}
t_{1} & 0 \\
m_{1} & u_{1}
\end{smallmatrix} \right]^{op}:\left[
\begin{smallmatrix}
T' & 0 \\ 
M & U'
\end{smallmatrix}\right]\longrightarrow \left[
\begin{smallmatrix}
T & 0 \\ M & U
\end{smallmatrix}\right]$
a morphism in $\mathbf{\Lambda}^{op}$ we set
$\mathbb{T}\Big(\left[
\begin{smallmatrix}
t_{1} & 0 \\
m_{1} & u_{1}
\end{smallmatrix} \right]^{op}\Big)= \left[
\begin{smallmatrix}
u_{1}^{op} & 0 \\
m_{1} & t_{1}^{op}
\end{smallmatrix} \right]:\left[
\begin{smallmatrix}
U' & 0 \\ 
\overline{M} & T'
\end{smallmatrix}\right]\longrightarrow \left[
\begin{smallmatrix}
U & 0 \\ 
\overline{M} & T
\end{smallmatrix}\right].$
\end{proposition}

\begin{proof}
Using the fact that $m_{1}\bullet u_{2}^{op}+t_{1}^{op}\bullet m_{2}=M(u_{2}\otimes 1_{T})(m_{1})+M(1_{U''}\otimes t_{1}^{op})(m_{2})=u_{2}\bullet m_{1}+m_{2}\bullet t_{1}$, it is straightforward to see that $\mathbb{T}$ is a functor and an isomorphism.
\end{proof}

Now, let $\mathcal{C}$ be a $K$-category, we know (see section 2.2 in this paper) that $\mathrm{\mathrm{Mod}}(\mathcal{C})$ is an $K$-variety, which we identify with the category of additive
covariant functors $(\mathcal{C},\mathrm{Mod}(K))$.\\
Now, we are going to consider the full subcategory of $K$-linear functors $M:\mathcal{C}\longrightarrow \mathsf{Mod}(K)$ with images in $\mathsf{mod}(K)$. That, is we consider 
$$(\mathcal{C},\mathsf{mod}(K)):=\{M\in (\mathcal{C},\mathsf{Mod}(K))\mid M(C)\in \mathsf{mod}(K)\,\forall C\in \mathcal{C}\}.$$
Let $\mathcal{C}$ be a Hom-finite $K$-category, then we have a duality
$$\mathbb{D}_{\mathcal{C}}:(\mathcal{C},\mathsf{mod}(K))\longrightarrow (\mathcal{C}^{op},\mathsf{mod}(K)),$$
where $\Big(\mathbb{D}_{\mathcal{C}}(M)\Big)(A):=\mathrm{Hom}_{K}(M(A),K)$ for $A\in \mathcal{C}$. If $\eta:M\longrightarrow N$ is a morphism in $(\mathcal{C},\mathsf{mod}(K))$, we have the morphism 
$$\mathbb{D}(\eta):=\Big \{ [\mathbb{D}(\eta)]_{A}:\Big(\mathbb{D}_{\mathcal{C}}(N)\Big)(A)\longrightarrow \Big(\mathbb{D}_{\mathcal{C}}(M)\Big)(A)\Big\}_{A\in \mathcal{C}}.$$
Where $[\mathbb{D}(\eta)]_{A}:=\mathrm{Hom}(\eta_{A},K):\mathrm{Hom}_{K}(N(A),K)\longrightarrow \mathrm{Hom}_{K}(M(A),K)$.\\

In all what follows we suppose that $\mathcal{U}$ and $\mathcal{T}$ are Hom-finite $K$-categories and therefore (by the previous lemma) $\mathcal{U}\otimes \mathcal{T}^{op}$ is also a Hom-finite $K$-category. We are going to consider the dualities
$$\mathbb{D}_{\mathcal{U}}:(\mathcal{U},\mathrm{mod}(K))\longrightarrow (\mathcal{U}^{op},\mathrm{mod}(K)),$$
$$\mathbb{D}_{\mathcal{T}}:(\mathcal{T},\mathrm{mod}(K))\longrightarrow (\mathcal{T}^{op},\mathrm{mod}(K)).$$
We have the following easy results.

\begin{lemma}
Let $\mathcal{U}$ and $\mathcal{T}$ be Hom-finite $K$-categories.
Then $\mathcal{U}\otimes_{K}\mathcal{T}^{op}$ is a Hom-finite $K$-category.
\end{lemma}
\begin{proof}
It is straighforward.
\end{proof}

\begin{proposition}\label{ElmorfiGam}
Let $\mathcal{U}$ and $\mathcal{T}$ be Hom-finite $K$-categories, $B\in \mathsf{Mod}(\mathcal{U})$,  and $U\in \mathcal{U}^{op}$. Then, we have a morphism of $\mathcal{T}$-modules
$\Gamma_{U}:\mathbb{G}(B) \longrightarrow  \mathrm{Hom}_{K}(-,B(U))\circ M_{U},$
where for $T\in \mathcal{T}$ we have that
$$[\Gamma_{U}]_{T}:\mathrm{Hom}_{\mathsf{Mod}(\mathcal{U})}(M_{T},B)\longrightarrow  \mathrm{Hom}_{K}(M(U,T),B(U))$$ is defined as
$$[\Gamma_{U}]_{T}(\beta):=\beta_{U} \,\, \, \forall \beta \in \mathrm{Hom}_{\mathsf{Mod}(\mathcal{U})}(M_{T},B).$$
\end{proposition}
\begin{proof}
It is straightforward.
\end{proof}

\begin{proposition}
Let $\mathcal{U}$ and $\mathcal{T}$ be Hom-finite $K$-categories, $B\in \mathsf{Mod}(\mathcal{U})$, $U\in \mathcal{U}^{op}$ and  $s\in \mathrm{Hom}_{K}(B(U), K)$. Then we have a morphism of $\mathcal{T}^{op}$-modules
$$\mathbb{S}_{B,U}:\overline{M}_{U}\longrightarrow \mathbb{D}_{\mathcal{T}}\mathbb{G}(B).$$
That is, $\mathbb{S}_{B,U}\in \overline{\mathbb{G}}(\mathbb{D}_{\mathcal{T}}(\mathbb{G}(B)))(U)=\mathrm{Hom}_{\mathsf{Mod}(\mathcal{T}^{op})}\Big(\overline{M}_{U},\,\,\,\mathbb{D}_{\mathcal{T}}\mathbb{G}(B) \Big),$
where for each $T\in \mathcal{T}^{op}$ we have that
$[\mathbb{S}_{B,U}]_{T}:M(U,T)\longrightarrow  
\mathrm{Hom}_{K}(\mathrm{Hom}_{\mathsf{Mod}(\mathcal{U})}(M_{T},B),K) $ is such that, for $m\in M(U,T)$,
$[\mathbb{S}_{B,U}]_{T}(m):\mathrm{Hom}_{\mathsf{Mod}(\mathcal{U})}(M_{T},B) \longrightarrow K$
is defined as
$$\Big([\mathbb{S}_{B,U}]_{T}(m)\Big)(\eta):=\Big(s\circ [\eta]_{U}\Big)(m)=s\Big([\eta]_{U}(m)\Big),$$
for every $\eta\in \mathrm{Hom}_{\mathsf{Mod}(\mathcal{U})}(M_{T},B)$. 
\end{proposition}
\begin{proof}
For each $U\in \mathcal{U}^{op}$, we have a natural isomorphism of $\mathcal{T}^{op}$-modules
$$\Delta_{U}: M_{U}\longrightarrow \mathrm{Hom}_{K}(-,K)\circ\mathrm{Hom}_{K}(-,K)\circ M_{U}$$
where for  $T\in \mathcal{T}^{op}$  and $m\in M(U,T)$ the function  $[\Delta_{U}]_{T}(m):\mathrm{Hom}_{K}(M(U,T),K)\rightarrow K$ is defined as $[\Delta_{U}]_{T}(m)(f):=f(m)$ $\forall f\in\mathrm{Hom}_{K}(M(U,T),K)$.\\
Now, we consider the canonical function (recall that $\overline{M}_{U}=M_{U}$)
$$\xymatrix{
\mathrm{Hom}_{\mathsf{Mod}(\mathcal{T}^{op})}\Big(M_{U},\,\,\, \mathrm{Hom}_{K}(-,K)\circ \mathrm{Hom}_{\mathsf{Mod}(\mathcal{U})}(-,B)\circ E\Big)\\
\mathrm{Hom}_{\mathsf{Mod}(\mathcal{T}^{})}\Big(\mathrm{Hom}_{\mathsf{Mod}(\mathcal{U})}(-,B)\circ E,\,\,\,\mathrm{Hom}_{K}(-,K)\circ M_{U}\Big)\ar[u]^{\Phi_{U}},}$$
defined as follows: If $\lambda\in \mathrm{Hom}_{\mathsf{Mod}(\mathcal{T}^{})}\Big(\mathrm{Hom}_{\mathsf{Mod}(\mathcal{U})}(-,B)\circ E,\,\,\,\mathrm{Hom}_{K}(-,K)\circ M_{U}\Big)$, we have the natural transformation 
$$\mathrm{Hom}_{K}(-,K)\lambda\!: \!\mathrm{Hom}_{}(-,K)\circ\mathrm{Hom}_{}(-,K)\circ M_{U}\longrightarrow 
\mathrm{Hom}_{}(-,K)\circ \mathrm{Hom}_{\mathsf{Mod}(\mathcal{U})}(-,B)\circ E$$
Then we define $\Phi_{U}(\lambda):=(\mathrm{Hom}_{K}(-,K)\lambda)\circ \Delta_{U}.$\\
Next, we consider the Yoneda functor
$$\mathbb{Y}:\mathsf{Mod}(K) \longrightarrow \mathsf{Mod}(\mathsf{Mod}(K)),$$
defined as $\mathbb{Y}(V):=\mathrm{Hom}_{K}(-,V)$ for $V\in \mathsf{Mod}(K)$. Since $B(U)\in \mathsf{Mod}(K)$, we have a function 
$$\mathbb{Y}:\mathrm{Hom}_{K}(B(U),K)\longrightarrow \mathrm{Hom}_{\mathsf{Mod}(\mathsf{Mod}(K))}\Big(\mathrm{Hom}_{K}(-,B(U)),\mathrm{Hom}_{K}(-,K)\Big).$$ That is, for $s \in \mathrm{Hom}_{K}(B(U),K)$ we obtain $\mathbb{Y}(s): \mathrm{Hom}_{K}(-,B(U))\rightarrow \mathrm{Hom}_{K}(-,K),$
in $\mathsf{Mod}(\mathsf{Mod}(K))$.Then we have the following natural transformation in $\mathsf{Mod}(\mathcal{T})$
$$\mathbb{Y}(s)M_{U}: \mathrm{Hom}_{K}(-,B(U))\circ M_{U}\longrightarrow \mathrm{Hom}_{K}(-,K)\circ M_{U}.$$
By composing with the morphism obtained in \ref{ElmorfiGam}, we have constructed the following morphism in $\mathsf{Mod}(\mathcal{T})$
$$H_{U}:=(\mathbb{Y}(s)M_{U})\circ \Gamma_{U}:\mathrm{Hom}_{\mathsf{Mod}(\mathcal{U})}(-,B)\circ E\longrightarrow \mathrm{Hom}_{K}(-,K)\circ M_{U}.$$
By applying $\Phi_{U}$ to $H_{U}$ we obtain the following
$$\Phi_{U}(H_{U})\in \mathrm{Hom}_{\mathsf{Mod}(\mathcal{T}^{op})}\Big(M_{U},\,\,\, \mathrm{Hom}_{K}(-,K)\circ \mathrm{Hom}_{\mathsf{Mod}(\mathcal{U})}(-,B)\circ E\Big).$$
We define $\mathbb{S}_{B,U}:=\Phi_{U}(H_{U}).$
Now, for each $T\in \mathcal{T}^{op}$ we have that 

\begin{align*}
\Big[\mathbb{S}_{B,U}\Big]_{T} & =
\Big[\mathrm{Hom}(-,K)\Big((\mathbb{Y}(s)M_{U})\circ (\Gamma_{U})\Big) \Big]_{T}\!\!\!\!\circ \!\![\Delta_{U}]_{T}\\
 & = \mathrm{Hom}\Big(\mathbb{Y}(s)_{M(U,T)}\circ [\Gamma_{U}]_{T}, K \Big)\circ [\Delta_{U}]_{T}.
\end{align*}
Therefore, we obtain the morphism
$$\mathrm{Hom}_{K}\Big(\mathbb{Y}(s)_{M(U,T)}\circ [\Gamma_{U}]_{T},\,\,\,K \Big)\circ [\Delta_{U}]_{T}:M(U,T)\longrightarrow \mathrm{Hom}_{K}(\mathrm{Hom}_{\mathsf{Mod}(\mathcal{U})}(M_{T},B),K)$$
Hence, for $m\in M(U,T)$ and $\eta\in \mathrm{Hom}_{\mathsf{Mod}(\mathcal{U})}(M_{T},B)$ we have that
\begin{align*}
\Big([\Delta_{U}]_{T}(m)\circ \mathbb{Y}(s)_{M(U,T)}\circ [\Gamma_{U}]_{T}\Big)(\eta) 
& =[\Delta_{U}]_{T}(m)\Big(\mathbb{Y}(s)_{M(U,T)}\Big([\eta]_{U}\Big)\Big)\\
& =[\Delta_{U}]_{T}(m)\Big(s\circ [\eta]_{U}\Big)\\
& =\Big(s\circ [\eta]_{U}\Big)(m)\\
& =s\Big([\eta]_{U}(m)\Big).
\end{align*}
\end{proof}

\begin{proposition}\label{conmuotroc}
Let $\mathcal{U}$ and $\mathcal{T}$ be Hom-finite $K$-categories and let $B\in \mathsf{Mod}(\mathcal{U})$ be. Then, there exists a morphism of $\mathcal{U}^{op}$-modules
$\Psi_{B}:\mathbb{D}_{\mathcal{U}}(B)\longrightarrow \overline{\mathbb{G}}(\mathbb{D}_{\mathcal{T}}\mathbb{G}(B)),$
where for each $U\in\mathcal{U}^{op}$ we have that 
$$\xymatrix{\mathrm{Hom}_{K}(B(U),K)\ar[d]^{[\Psi_{B}]_{U}}\\
\mathrm{Hom}_{\mathsf{Mod}(\mathcal{T}^{op})}\Big(\overline{M}_{U},\,\,\, \mathrm{Hom}_{K}(-,K)\circ \mathrm{Hom}_{\mathsf{Mod}(\mathcal{U})}(-,B)\circ E\Big)}$$
is defined as: $[\Psi_{B}]_{U}(s):=\mathbb{S}_{B,U}$
for every $s\in \mathrm{Hom}_{K}(B(U),K)$.\\ 
Moreover, if $\beta:B\longrightarrow B'$ is a morphism of $\mathcal{U}$-modules, the following diagram commutes
$$\xymatrix{\mathbb{D}_{\mathcal{U}}(B')\ar[rr]^{\Psi_{B'}}\ar[d]^{\mathbb{D}_{\mathcal{U}}(\beta)}& & \overline{\mathbb{G}}(\mathbb{D}_{\mathcal{T}}\mathbb{G}(B'))\ar[d]^{\overline{\mathbb{G}}(\mathbb{D}_{\mathcal{T}}\mathbb{G}(\beta))}\\
\mathbb{D}_{\mathcal{U}}(B)\ar[rr]^{\Psi_{B}} & &\overline{\mathbb{G}}(\mathbb{D}_{\mathcal{T}}\mathbb{G}(B)).}$$
\end{proposition}
\begin{proof}
Let $u\in \mathrm{Hom}_{\mathcal{U}}(U,U')$ be. Then $\overline{u}:M_{U}\longrightarrow M_{U'}$ is a morphism in $\mathsf{Mod}(\mathcal{T}^{op})$ and $B(u):B(U)\longrightarrow B(U')$ is a morphism of abelian groups.
We have to show that $L \circ [\Psi_{B}]_{U'}=[\Psi_{B}]_{U}\circ \mathrm{Hom}_{K}(B(u),K)$, where $L=\mathrm{Hom}_{\mathsf{Mod}(\mathcal{T}^{op})}\Big(\overline{u},\,\,\, \mathrm{Hom}_{K}(-,K)\circ \mathrm{Hom}_{\mathsf{Mod}(\mathcal{U})}(-,B)\circ E\Big)$.\\
Indeed, one one side, for $s\in \mathrm{Hom}_{K}(B(U'),K)$ we have that $(L\circ [\Psi_{B}]_{U'})(s)  =[\Psi_{B}]_{U'}(s)\circ \overline{u} = \mathbb{S}_{B,U'} \circ \overline{u}$. 
Then, for $T\in\mathcal{T}^{op}$  and $m\in M(U,T)$  the function
$$[\mathbb{S}_{B,U'}]_{T}\Big(M(u\otimes 1_{T})(m)\Big):\mathrm{Hom}_{\mathsf{Mod}(\mathcal{U})}(M_{T},B) \longrightarrow K$$
is defined as
$$\Big([\mathbb{S}_{B,U'}]_{T}\Big(M(u\otimes 1_{T})(m)\Big)\Big)(\eta):= \Big(s\circ \eta_{U'}\Big) (M(u\otimes 1_{T})(m))=s\Big([\eta]_{U'}\Big(M(u\otimes 1_{T})(m)\Big)\Big),$$
for every $\eta\in \mathrm{Hom}_{\mathsf{Mod}(\mathcal{U})}(M_{T},B)$ (recall that $[\overline{u}]_{T}=M(u\otimes 1_{T})$).\\
On the other side, $
([\Psi_{B}]_{U}\circ J)(s)= [\Psi_{B}]_{U}(J(s))=[\Psi_{B}]_{U}(s\circ B(u)).$ Let us denote $s':=s\circ B(u)$, hence $[\Psi_{B}]_{U}(s\circ B(u))=[\Psi_{B}]_{U}(s'):=\mathbb{S}'_{B,U}.$\\
Then, for $T\in\mathcal{T}^{op}$ and  $m\in M(U,T)$, the function
$$[\mathbb{S}'_{B,U}]_{T}(m):\mathrm{Hom}_{\mathsf{Mod}(\mathcal{U})}(M_{T},B) \longrightarrow K$$
is defined as
$$\Big([\mathbb{S}_{B,U}]_{T}(m)\Big)(\eta):= \Big(s'\circ \eta_{U}\Big) (m)=s'\Big([\eta]_{U}(m)\Big)=(s\circ B(u))\Big([\eta]_{U}(m)\Big),$$
for every $\eta\in \mathrm{Hom}_{\mathsf{Mod}(\mathcal{U})}(M_{T},B)$.\\
Since  $\eta\in \mathrm{Hom}_{\mathsf{Mod}(\mathcal{U})}(M_{T},B)$, we have that $B(u)\circ \eta_{U}=\eta_{U'}\circ M(u\otimes 1_{T})$.
Therefore, for $m\in M(U,T)$ we have that $B(u)(\eta_{U}(m))=\eta_{U'}(M(u\otimes 1_{T})(m)).$ Proving that the required diagram commutes.\\
Let $\beta:B\longrightarrow B'$ be a a morphism of $\mathcal{U}$-modules, we have to show that for $U\in\mathcal{U}^{op}$ the following  equality holds: $[\overline{\mathbb{G}}(\mathbb{D}_{\mathcal{T}}\mathbb{G}(\beta))]_{U} \circ [\Psi_{B'}]_{U}=[\Psi_{B}]_{U}\circ [\mathbb{D}_{\mathcal{U}}(\beta)]_{U}.$
We note that for $U\in \mathcal{U}^{op}$
$$\Big[\overline{\mathbb{G}}(\mathbb{D}_{\mathcal{T}}\mathbb{G}(\beta))\Big]_{U}:\overline{\mathbb{G}}(\mathbb{D}_{\mathcal{T}}\mathbb{G}(B'))(U)\longrightarrow \overline{\mathbb{G}}(\mathbb{D}_{\mathcal{T}}\mathbb{G}(B))(U)$$ is defined as $\Big[\overline{\mathbb{G}}(\mathbb{D}_{\mathcal{T}}\mathbb{G}(\beta))\Big]_{U}(\delta):=\mathrm{Hom}_{\mathsf{Mod}(\mathcal{T}^{op})}\Big(\overline{M}_{U},\mathbb{D}_{\mathcal{T}}\mathbb{G}(\beta)\Big)(\delta)=\mathbb{D}_{\mathcal{T}}\mathbb{G}(\beta)\circ \delta$ for $\delta\in \overline{\mathbb{G}}(\mathbb{D}_{\mathcal{T}}(\mathbb{G}(B')))(U)$. Now, for each $T\in \mathcal{T}^{op}$ we have 
$$\Big[\mathbb{D}_{\mathcal{T}}\mathbb{G}(\beta)\Big]_{T}:\mathrm{Hom}_{K}(\mathrm{Hom}_{\mathsf{Mod}(\mathcal{U})}(M_{T},B'),K)\longrightarrow \mathrm{Hom}_{K}(\mathrm{Hom}_{\mathsf{Mod}(\mathcal{U})}(M_{T},B),K),$$
where for $\lambda\in \mathrm{Hom}_{K}(\mathrm{Hom}_{\mathsf{Mod}(\mathcal{U})}(M_{T},B'), K)$ and  $\eta \in \mathrm{Hom}_{\mathsf{Mod}(\mathcal{U})}(M_{T},B)$ we obtain that $\Big(\Big[\mathbb{D}_{\mathcal{T}}\mathbb{G}(\beta)\Big]_{T}(\lambda)\Big)(\eta):=\lambda(\beta\circ \eta).$\\
For $s:B'(U)\longrightarrow K$ we have that $[\Psi_{B'}]_{U}(s):=\mathbb{S}_{B',U}.$
Then 
\begin{align*}
\Big(\Big[\overline{\mathbb{G}}(\mathbb{D}_{\mathcal{T}}\mathbb{G}(\beta))\Big]_{U}\circ \Big[\Psi_{B'}\Big]_{U}\Big)(s)= (\mathbb{D}_{\mathcal{T}}\mathbb{G}(\beta))\circ \mathbb{S}_{B',U} 
\end{align*}
Then for $T\in\mathcal{T}^{op}$ and $m\in M(U,T)$ the function 
$$\Big[\mathbb{D}_{\mathcal{T}}\mathbb{G}(\beta)\Big]_{T}\Big([\mathbb{S}_{B',U}]_{T}(m)\Big):\mathrm{Hom}_{\mathsf{Mod}(\mathcal{U})}(M_{T},B)\longrightarrow K$$
is defined as
\begin{align*}
\Big(\Big[\mathbb{D}_{\mathcal{T}}\mathbb{G}(\beta)\Big]_{T}\Big([\mathbb{S}_{B',U}]_{T}(m)\Big)\Big)(\eta) & =\Big([\mathbb{S}_{B',U}]_{T}(m)\Big)(\beta\circ \eta)\\
& = s\Big([\beta\circ \eta]_{U}(m)\Big)\\
& =s\Big(([\beta]_{U}\circ [\eta]_{U})(m)\Big) \quad \forall \eta \in \mathrm{Hom}_{\mathsf{Mod}(\mathcal{U})}(M_{T},B)
\end{align*}
On the other hand,  we consider $s':=[\mathbb{D}_{\mathcal{U}}(\beta)]_{U}(s)=s\circ [\beta]_{U}$. In this case, we obtain thta $\Psi_{B}(s'):=\mathbb{S'}_{B,U}.$ Therefore,  for each $T\in \mathcal{T}^{op}$ and $m\in M(U,T)$ the function 
$$[\mathbb{S}'_{B,U}]_{T}(m):\mathrm{Hom}_{\mathsf{Mod}(\mathcal{U})}(M_{T},B) \longrightarrow K$$
is defined as
$$\Big([\mathbb{S}'_{B,U}]_{T}(m)\Big)(\eta):=s'\Big([\eta]_{U}(m)\Big)=s\circ [\beta]_{U}\Big([\eta]_{U}(m)\Big)=s\Big(([\beta]_{U}\circ [\eta]_{U})(m)\Big)$$
for every $\eta\in \mathrm{Hom}_{\mathsf{Mod}(\mathcal{U})}(M_{T},B)$. Proving that the required diagram commutes.
\end{proof}

\begin{proposition}\label{funPsipreli}
Let $\mathcal{U}$ and $\mathcal{T}$ be Hom-finite $K$-categories. Then, there exists a contravariant functor $\Psi: \mathsf{Mod}(\mathcal{U})\longrightarrow \Big(\mathsf{Mod}(\mathcal{U}^{op}), \overline{\mathbb{G}}(\mathsf{Mod}(\mathcal{T}^{op}))\Big).$
\end{proposition}
\begin{proof} For $B\in \mathsf{Mod}(\mathcal{U})$ we define
$\Psi(B):=\Big(\mathbb{D}_{\mathcal{U}}(B),\Psi_{B},\mathbb{D}_{\mathcal{T}}\mathbb{G}(B)\Big).$ For $\beta:B\longrightarrow B'$ we have that
$$\Psi(\beta):=(\mathbb{D}_{\mathcal{U}}(\beta), \mathbb{D}_{\mathcal{T}}\mathbb{G}(\beta)):
\Big(\mathbb{D}_{\mathcal{U}}(B'),\Psi_{B'}, \mathbb{D}_{\mathcal{T}}\mathbb{G}(B')\Big)\longrightarrow 
\Big(\mathbb{D}_{\mathcal{U}}(B),\Psi_{B}, \mathbb{D}_{\mathcal{T}}\mathbb{G}(B)\Big).$$
By proposition \ref{conmuotroc},  it follows that $\Psi(\beta)$ is a morphism in $\Big(\mathsf{Mod}(\mathcal{U}^{op}), \overline{\mathbb{G}}(\mathsf{Mod}(\mathcal{T}^{op}))\Big)$.\end{proof}

\begin{proposition}\label{morpduality}
Let $\mathcal{U}$ and $\mathcal{T}$ be Hom-finite $K$-categories. Then, there exists a contravariant functor $\widehat{\Theta}:\Big(\mathsf{Mod}(\mathcal{T}), \mathbb{G}(\mathsf{Mod}(\mathcal{U})\Big) \longrightarrow \Big(\mathsf{Mod}(\mathcal{U}^{op}), \overline{\mathbb{G}}(\mathsf{Mod}(\mathcal{T}^{op}))\Big)$ which
defined as follows:
\begin{enumerate}
\item [(a)] For $f:A\longrightarrow \mathbb{G}(B)$ a morphism of $\mathcal{T}$-modules. We define
$$\overline{f}=\widehat{\Theta}(A,f,B):=\Theta(A,f,B)\circ \Psi_{B}:\mathbb{D}_{\mathcal{U}}(B)\longrightarrow \overline{\mathbb{G}}(\mathbb{D}_{\mathcal{T}}A).$$
That is, $\overline{f}$ is the following composition
$$\xymatrix{\mathbb{D}_{\mathcal{U}}(B)\ar[r]^(.4){\Psi_{B}} &  \overline{\mathbb{G}}(\mathbb{D}_{\mathcal{T}}\mathbb{G}(B))\ar[rr]^{\overline{\mathbb{G}}(\mathbb{D}_{\mathcal{T}}(f))} &  &\overline{\mathbb{G}}(\mathbb{D}_{\mathcal{T}}A).}$$

\item [(b)] If $(\alpha,\beta):(A,f,B)\longrightarrow (A',f',B')$ then
$\widehat{\Theta}(\alpha,\beta)=\Big(\mathbb{D}_{\mathcal{U}}(\beta),\mathbb{D}_{\mathcal{T}}(\alpha)\Big)
.$
\end{enumerate}
\end{proposition}

\begin{proof}
Consider $\overline{\mathbb{G}}\circ \mathbb{D}_{\mathcal{T}}:\mathsf{Mod}(\mathcal{T})\longrightarrow \mathsf{Mod}(\mathcal{U}^{op})$. This induces a contravariant functor
$\Theta:\Big(\mathsf{Mod}(\mathcal{T}), \mathbb{G}(\mathsf{Mod}(\mathcal{U})\Big) \longrightarrow \Big(\mathsf{Mod}(\mathcal{U}^{op}), \overline{\mathbb{G}}(\mathsf{Mod}(\mathcal{T}^{op}))\Big)$ as follows. Let $f:A\longrightarrow \mathbb{G}(B)$ a morphism of $\mathcal{T}$-modules. Then we have a morphism of $\mathcal{U}^{op}$-modules $\overline{\mathbb{G}}(\mathbb{D}_{\mathcal{T}}(f)):\overline{\mathbb{G}}(\mathbb{D}_{\mathcal{T}}\mathbb{G}(B))\longrightarrow \overline{\mathbb{G}}(\mathbb{D}_{\mathcal{T}}A).$
Thus  we define $$\Theta(A,f,B):=\Big(\overline{\mathbb{G}}(\mathbb{D}_{\mathcal{T}}\mathbb{G}(B)), \overline{\mathbb{G}}(\mathbb{D}_{\mathcal{T}}(f)),\mathbb{D}_{\mathcal{T}}A\Big).$$
If $(\alpha,\beta):(A,f,B)\longrightarrow (A',f',B')$ is a morphism in 
$\Big(\mathsf{Mod}(\mathcal{T}), \mathbb{G}(\mathsf{Mod}(\mathcal{U})\Big)$ we set
$$\Theta(\alpha,\beta):=\Big(\overline{\mathbb{G}}\mathbb{D}_{\mathcal{T}}\mathbb{G}(\beta),\mathbb{D}_{\mathcal{T}}(\alpha)\Big).$$ It is easy to show that $\Theta$ is a contravariant functor. Now, by using the functor $\Psi$ in \ref{funPsipreli}, 
It is straightforward to see that $\widehat{\Theta}$ is a contravariant functor.
\end{proof}

\begin{note}
Consider the morphism  $\overline{f}:\mathbb{D}_{\mathcal{U}}(B)\longrightarrow \overline{\mathbb{G}}(\mathbb{D}_{\mathcal{T}}A)$ given in \ref{morpduality}.  For $U\in \mathcal{U}^{op}$, we have the morphism $\Big[\overline{f}\Big]_{U}:\mathbb{D}_{\mathcal{U}}(B)(U)=\mathrm{Hom}_{K}(B(U),K)\longrightarrow \overline{\mathbb{G}}(\mathbb{D}_{\mathcal{T}}A)(U).$ Therefore,  for $s\in \mathrm{Hom}_{K}(B(U),K)$ and  $T\in \mathcal{T}^{op}$, we have the function
$$\Big[\Big[\overline{f}\Big]_{U}(s)\Big]_{T}:\overline{M}_{U}(T)\longrightarrow  \mathbb{D}_{\mathcal{T}}A(T),$$ given as follows:
for $m\in M(U,T)$ the function $\Big[\Big[\overline{f}\Big]_{U}(s)\Big]_{T}(m):A(T)\longrightarrow K$ is defined as $\Big(\Big[\Big[\overline{f}\Big]_{U}(s)\Big]_{T}(m)\Big)(x):=
s\Big([f_{T}(x)]_{U}(m)\Big)  \quad \forall x\in A(T).$
\end{note}
\begin{proof}
Indeed, by definition we have that $\Big[\overline{f}\Big]_{U}=[\overline{\mathbb{G}}\mathbb{D}_{\mathcal{T}}(f)]_{U}\circ [\Psi_{B}]_{U}.$ Then, for $s:B(U)\longrightarrow K$ we have that $
[\overline{\mathbb{G}}\mathbb{D}_{\mathcal{T}}(f)]_{U}\circ [\Psi_{B}]_{U}(s)=[\overline{\mathbb{G}}\mathbb{D}_{\mathcal{T}}(f)]_{U}(\mathbb{S}_{B,U})$.\\
We recall that for $\alpha\in \mathrm{Hom}_{\mathsf{Mod}(\mathcal{T}^{op})}\Big(M_{U}, \mathbb{D}_{\mathcal{T}}(\mathbb{G}(B))\Big)$ if follows that 
$$[\overline{\mathbb{G}}(\eta)]_{U}(\alpha)= \mathbb{D}_{\mathcal{T}}(f) \circ \alpha$$
where $\mathbb{D}_{\mathcal{T}}(f)$ is the following morphism of $\mathcal{T}^{op}$-modules $\mathbb{D}_{\mathcal{T}}(f):\mathbb{D}_{\mathcal{T}}\mathbb{G}(B)\longrightarrow \mathbb{D}_{\mathcal{T}}(A).$
Therefore, 
\begin{align*}
[\overline{\mathbb{G}}\mathbb{D}_{\mathcal{T}}(f)]_{U}\circ [\Psi_{B}]_{U}(s)=[\overline{\mathbb{G}}\mathbb{D}_{\mathcal{T}}(f)]_{U}(\mathbb{S}_{B,U})=\mathbb{D}_{\mathcal{T}}(f)\circ \mathbb{S}_{B,U}:\overline{M}_{U}\longrightarrow \mathbb{D}_{\mathcal{T}}(A)
\end{align*}
For $T\in \mathcal{T}^{op}$, the function $[\mathbb{D}_{\mathcal{T}}(f)\circ \mathbb{S}_{B,U}]_{T}:M(U,T)\longrightarrow \mathrm{Hom}_{K}(A(T),K)$, is given as follows: $[\mathbb{D}_{\mathcal{T}}(f)\circ \mathbb{S}_{B,U}]_{T}(m)=
[\mathbb{D}_{\mathcal{T}}(f)]_{T}\Big([\mathbb{S}_{B,U}]_{T}(m)\Big)$  $\forall m\in \overline{M}_{U}(T):=M(U,T)$, where
$[\mathbb{S}_{B,U}]_{T}(m):\mathrm{Hom}_{\mathsf{Mod} (\mathcal{U})}(M_{T},B)\longrightarrow K.$\\
Then $[\mathbb{D}_{\mathcal{T}}(f)\circ \mathbb{S}_{B,U}]_{T}(m)=
[\mathbb{D}_{\mathcal{T}}(f)]_{T}\Big([\mathbb{S}_{B,U}]_{T}(m)\Big)=\Big([\mathbb{S}_{B,U}]_{T}(m)\Big)\circ f_{T}$. Hence, for $x\in A(T)$ we obtain that $[\mathbb{D}_{\mathcal{T}}(f)\circ \mathbb{S}_{B,U}]_{T}(m)(x):=
\Big([\mathbb{S}_{B,U}]_{T}(m)\Big)( f_{T}(x))=s\Big([f_{T}(x)]_{U}(m)\Big)$.
\end{proof}
Let $\mathcal{U}$ and $\mathcal{T}$ be Hom-finite $K$-categories.
Now, consider the isomorphism
$\mathbb{T}:\mathbf{\Lambda}^{op}\longrightarrow \overline{\mathbf{\Lambda}},$ given in \ref{isotrivial}. Then, there exists a functor $\mathbb{T}^{\ast}: \mathsf{Mod}(\overline{\mathbf{\Lambda}})\longrightarrow \mathsf{Mod}(\mathbf{\Lambda}^{op}).$\\
Also, we have an equivalence $\overline{\textswab{F}}:\Big( \mathsf{Mod}(\mathcal{U}^{op}), \overline{\mathbb{G}}(\mathsf{Mod}(\mathcal{T}^{op}))\Big) \longrightarrow 
\mathrm{Mod}(\overline{\mathbf{\Lambda}})$ constructed in \ref{equivalence1} (for the case of $\overline{\mathbf{\Lambda}}$).\\
We will see that the following diagram is commutative up to a natural isomorphism $\nu$
$$\xymatrix{
\Big(\mathsf{Mod}(\mathcal{T}),\mathbb{G}\mathsf{Mod}(\mathcal{U})\Big)\ar[rr]^{\textswab{F}}\ar[d]_{\widehat{\Theta}} & & \mathsf{Mod}(\mathbf{\Lambda})\ar[d]^{\mathbb{D}_{\mathbf{\Lambda}}}\\
\Big(\mathsf{Mod}(\mathcal{U}^{op}),\overline{\mathbb{G}}\mathsf{Mod}(\mathcal{T}^{op})\Big)\ar@{=>}[urr]_{\nu}\ar[rr]_{\mathbb{T}^{\ast}\circ \overline{\textswab{F}}} & &   \mathsf{Mod}(\mathbf{\Lambda}^{op}).}$$
So, first we are going to describe the two contravariant  functors.\\
$$\textbf{FIRST FUNCTOR:}$$
The first functor we are going to describe is the following:
$$\mathbb{T}^{\ast}\circ \overline{\textswab{F}}\circ \widehat{\Theta}:\Big(\mathsf{Mod}(\mathcal{T}), \mathbb{G}(\mathsf{Mod}(\mathcal{U})\Big)\longrightarrow \mathsf{Mod}(\mathbf{\Lambda}^{op}).$$
$\textbf{Objects}$. Consider $(A,f,B)$ an object in $\Big(\mathsf{Mod}(\mathcal{T}), \mathbb{G}(\mathsf{Mod}(\mathcal{U})\Big)$.
Let $$\overline{f}:=\widehat{\Theta}\Big((A,f,B)\Big):\mathbb{D}_{\mathcal{U}}(B)\longrightarrow \overline{\mathbb{G}}(\mathbb{D}_{\mathcal{T}}A).$$
Then $\overline{\textswab{F}}(\overline{f}):=\mathbb{D}_{\mathcal{U}}(B)\amalg_{\overline{f}} \mathbb{D}_{\mathcal{T}}A:\overline{\mathbf{\Lambda}}\longrightarrow \mathbf{Ab}$ is defined as follows:\\
$(a)$ For
$\left[\begin{smallmatrix}
U& 0 \\
\overline{M} & T \\
\end{smallmatrix} \right]\in\overline{\mathbf{\Lambda}}$ we set  $\Big(\mathbb{D}_{\mathcal{U}}(B)\amalg_{\overline{f}} \mathbb{D}_{\mathcal{T}}A\Big)(\left[\begin{smallmatrix}
U& 0 \\
\overline{M} & T \\
\end{smallmatrix} \right]) :=\mathbb{D}_{\mathcal{U}}(B)(U)\amalg  \mathbb{D}_{\mathcal{T}}(A)(T).$\\
$(b)$ If $\left[\begin{smallmatrix}
u^{op} & 0 \\
m& t^{op} \\
\end{smallmatrix} \right]\in \mathsf{Hom}_{\overline{\mathbf{\Lambda}}}(\left[\begin{smallmatrix}
U'& 0 \\
\overline{M} & T' \\
\end{smallmatrix} \right], \left[\begin{smallmatrix}
U & 0 \\
\overline{M} & T \\
\end{smallmatrix} \right])=\left[ \begin{smallmatrix}
\mathsf{Hom}_{\mathcal{U}^{op}}(U',U) & 0 \\
\overline{M}(T,U') & \mathsf{Hom}_{\mathcal{T}^{op}}(T',T)
\end{smallmatrix} \right]$ we define the map
$$\Big(\mathbb{D}_{\mathcal{U}}(B)\amalg_{\overline{f}}  \mathbb{D}_{\mathcal{T}}A\Big)(\left[\begin{smallmatrix}
u^{op}& 0 \\
m& t^{op} \\
\end{smallmatrix} \right]):=\left [\begin{smallmatrix}
\mathbb{D}_{\mathcal{U}}(B)(u^{op})& 0 \\
m & \mathbb{D}_{\mathcal{T}}(A)(t^{op}) \\
\end{smallmatrix} \right],$$  given by $\left [\begin{smallmatrix}
\mathbb{D}_{\mathcal{U}}(B)(u^{op})& 0 \\
m& \mathbb{D}_{\mathcal{T}}(A)(t^{op}) \\
\end{smallmatrix} \right]\left [\begin{smallmatrix}
s\\
\\
w \\
\end{smallmatrix} \right]=\left [\begin{smallmatrix}
\mathbb{D}_{\mathcal{U}}(B)(u^{op})(s) \\
m\cdot s+ \mathbb{D}_{\mathcal{T}}(A)(t^{op})(w) \\
\end{smallmatrix} \right]$
for $(s,w)\in \mathbb{D}_{\mathcal{U}}(B)(U')\amalg \mathbb{D}_{\mathcal{T}}(A)(T')$, with $m\cdot s:=\Big[\Big[\overline{f}\Big]_{U'}(s)\Big]_{T}(m)\in \mathbb{D}_{\mathcal{T}}(A)(T)$ (see  \ref{Tlemma1}).\\
For $u^{op}:U'\longrightarrow U$ and  $y\in B(U)$ we get that
$\Big(\mathbb{D}_{\mathcal{U}}B(u^{op})(s)\Big)(y):=s(B(u)(y)).$\\
For  $t^{op}:T'\longrightarrow T$ and $x\in A(T)$ we get that
\begin{align*}
\Big(m\cdot s+ \mathbb{D}_{\mathcal{T}}(A)(t^{op})(w)\Big)(x)= &
\Big(\Big[\Big[\overline{f}\Big]_{U'}(s)\Big]_{T}(m)\Big)(x)+w(A(t)(x))\\
& =s\Big([f_{T}(x)]_{U'}(m)\Big)+ w(A(t)(x))\\
& = s(m\cdot x)+w(A(t)(x))
\end{align*}
Therefore,
\begin{align*}
\Big(\mathbb{T}^{\ast}\circ \overline{\textswab{F}}\circ \widehat{\Theta}\Big)(A,f,B)(\left[\begin{smallmatrix}
T & 0 \\
M & U \\
\end{smallmatrix} \right])=\mathbb{T}^{\ast}\Big(\overline{\textswab{F}}(\overline{f})\Big)(\left[\begin{smallmatrix}
T & 0 \\
M & U \\
\end{smallmatrix} \right])=
\Big(\overline{\textswab{F}}(\overline{f})\Big)(\left[\begin{smallmatrix}
U& 0 \\
\overline{M} & T \\
\end{smallmatrix} \right]) \quad \forall \left[\begin{smallmatrix}
T & 0 \\
M & U \\
\end{smallmatrix} \right]\in \mathbf{\Lambda}^{op}.
\end{align*}
and
\begin{align*}
\Big(\mathbb{T}^{\ast}\circ \overline{\textswab{F}}\circ \widehat{\Theta}\Big)(A,f,B)(\left[\begin{smallmatrix}
t & 0 \\
m & u\\
\end{smallmatrix} \right]^{op})=\Big(\overline{\textswab{F}}(\overline{f})\Big)(\left[\begin{smallmatrix}
u^{op} & 0 \\
m & t^{op} \\
\end{smallmatrix} \right]).
\end{align*}
if $\left[\begin{smallmatrix}
t& 0 \\
m & u \\
\end{smallmatrix} \right]^{op}:\left[\begin{smallmatrix}
T'& 0 \\
M & U' \\
\end{smallmatrix} \right]\longrightarrow \left[\begin{smallmatrix}
T & 0 \\
M & U \\
\end{smallmatrix} \right]$ is a morphism in $\mathbf{\Lambda}^{op}$.

$\textbf{Morphisms}$. Let us consider $(\alpha,\beta):(A,f,B)\longrightarrow (A',f',B')$ a morphism in $\Big(\mathsf{Mod}(\mathcal{T}), \mathbb{G}(\mathsf{Mod}(\mathcal{U})\Big).$
Then $
\Big(\mathbb{T}^{\ast}\circ \overline{\textswab{F}}\circ \widehat{\Theta}\Big)(\alpha,\beta)=\mathbb{T}^{\ast}\Big(\overline{\textswab{F}}\Big(\mathbb{D}_{\mathcal{U}}(\beta),\mathbb{D}_{\mathcal{T}}(\alpha)\Big)\Big).$
Hence,
$$\Big(\mathbb{T}^{\ast}\circ \overline{\textswab{F}}\circ \widehat{\Theta}\Big)(\alpha,\beta)
: \Big(\mathbb{T}\circ \overline{\textswab{F}}\circ \widehat{\Theta}\Big)(A',f',B')\longrightarrow \Big(\mathbb{T}\circ \overline{\textswab{F}}\circ \widehat{\Theta}\Big)(A,f,B)$$
is the natural transformation, where for  $\left[
\begin{smallmatrix}
T & 0 \\
M & U\\
\end{smallmatrix}
\right]\in \mathbf{\Lambda}^{op}$ 

$$\Big[\Big(\mathbb{T}^{\ast}\circ \overline{\textswab{F}}\circ \widehat{\Theta}\Big)(\alpha,\beta)\Big]_{{}_{\left[
\begin{smallmatrix}
T & 0 \\
M & U\\
\end{smallmatrix}
\right] }}:=\Big[\mathbb{D}_{\mathcal{U}}(\beta)\!\amalg\!\mathbb{D}_{\mathcal{T}}(\alpha)\Big]_{\mathbb{T}\Big(\left[
\begin{smallmatrix}
T & 0 \\
M & U\\
\end{smallmatrix}
\right]\Big)}=\Big[\mathbb{D}_{\mathcal{U}}(\beta)\!\amalg\!\mathbb{D}_{\mathcal{T}}(\alpha)\Big]_{{}_{\left[
\begin{smallmatrix}
U & 0 \\
\overline{M} & T\\
\end{smallmatrix}
\right] }}$$
with  

$$\Big[\Big(\mathbb{T}^{\ast}\circ \overline{\textswab{F}}\circ \widehat{\Theta}\Big)(\alpha,\beta)\Big]_{{}_{\left[
\begin{smallmatrix}
T & 0 \\
M & U\\
\end{smallmatrix}
\right] }}
\!\!:\!\! \Big(\mathbb{D}_{\mathcal{U}}(B')\amalg_{\overline{f'}} \mathbb{D}_{\mathcal{T}}(A')\Big) (\left [\begin{smallmatrix}
U& 0 \\
\overline{M}& T\\
\end{smallmatrix} \right])\rightarrow \Big(\mathbb{D}_{\mathcal{U}}(B)\amalg_{\overline{f}}\mathbb{D}_{\mathcal{T}}(A)\Big)(\left [\begin{smallmatrix}
U& 0 \\
\overline{M} & T\\
\end{smallmatrix} \right]).$$
For simplicity, for  $\left[\begin{smallmatrix}
T & 0 \\
 M & U\\
\end{smallmatrix} \right ]\in \mathbf{\Lambda}^{op}$ we set 
$\Big[\zeta \Big]_{{}_{\left[
\begin{smallmatrix}
T & 0 \\
M & U\\
\end{smallmatrix}
\right] }}:=\Big[\Big(\mathbb{T}^{\ast}\circ \overline{\textswab{F}}\circ \widehat{\Theta}\Big)(\alpha,\beta)\Big]_{{}_{\left[
\begin{smallmatrix}
T & 0 \\
M & U\\
\end{smallmatrix}
\right] }}$
where $\Big[\zeta \Big]_{{}_{\left[
\begin{smallmatrix}
T & 0 \\
M & U\\
\end{smallmatrix}
\right] }}(s,w):=(s\circ \beta_{U}, w\circ \alpha_{T})$ $\,\,\,$ $\forall$ $(s,w)\in \mathrm{Hom}_{K}(B'(U),K)\amalg \mathrm{Hom}_{K}(A'(T),K)$.\\

$$\textbf{SECOND FUNCTOR:}$$\\
Now, we are going to describe the contravariant functor
$$\mathbb{D}_{\mathbf{\Lambda}}\circ \textswab{F}:\Big(\mathsf{Mod}(\mathcal{T}), \mathbb{G}(\mathsf{Mod}(\mathcal{U})\Big)\longrightarrow \mathsf{Mod}(\mathbf{\Lambda}^{op}).$$
$\textbf{Objects}$. Consider $(A,f,B)$ an object in $\Big(\mathsf{Mod}(\mathcal{T}), \mathbb{G}(\mathsf{Mod}(\mathcal{U})\Big)$. Then we have that
$$\mathbb{D}_{\mathbf{\Lambda}}(A\amalg_{f}B):\mathbf{\Lambda}^{op} \longrightarrow \mathbf{Ab}$$ is defined as follows:\\
$(a)$ For $\left[\begin{smallmatrix}
T& 0 \\
M & U \\
\end{smallmatrix} \right]\in \mathbf{\Lambda}^{op}$, we get
$\mathbb{D}_{\mathbf{\Lambda}}(A\amalg_{f}B)\Big(\left[\begin{smallmatrix}
T& 0 \\
M & U \\
\end{smallmatrix} \right]\Big):=\mathrm{Hom}_{K}(A(T)\amalg B(U),K).$\\
$(b)$ If $\left[\begin{smallmatrix}
t& 0 \\
m & u \\
\end{smallmatrix} \right]:\left[\begin{smallmatrix}
T& 0 \\
M & U \\
\end{smallmatrix} \right]\longrightarrow \left[\begin{smallmatrix}
T'& 0 \\
M & U' \\
\end{smallmatrix} \right]$
is a morphism in $\mathbf{\Lambda}$, we have the morphism
$$(A\amalg_{f}B)\Big(\left[\begin{smallmatrix}
 t & 0 \\
m & u \\
\end{smallmatrix} \right]\Big)=\left[\begin{smallmatrix}
 A(t )& 0 \\
m & B(u) \\
\end{smallmatrix} \right]:A(T)\amalg B(U)\longrightarrow  A(T')\amalg B(U').$$ 
Therefore, by definition $\mathbb{D}_{\mathbf{\Lambda}}(A\amalg_{f} B)\Big(\left[\begin{smallmatrix}
 t & 0 \\
m & u \\
\end{smallmatrix} \right]^{op}\Big):=\mathrm{Hom}_{K}\Big(\left[\begin{smallmatrix}
 A(t )& 0 \\
m & B(u) \\
\end{smallmatrix} \right],K\Big),$ where
$$\mathrm{Hom}_{K}\Big(\left[\begin{smallmatrix}
 A(t )& 0 \\
m & B(u) \\
\end{smallmatrix} \right],K\Big):\mathrm{Hom}_{K}(A(T')\amalg B(U'),K)\longrightarrow \mathrm{Hom}_{K}(A(T)\amalg B(U),K)$$ 
is defined as follows:
$\Big(\mathrm{Hom}_{K}\Big(\left[\begin{smallmatrix}
 A(t )& 0 \\
m & B(u) \\
\end{smallmatrix} \right],K\Big)\Big)((w,s))=(w,s)\circ \left[\begin{smallmatrix}
 A(t )& 0 \\
m & B(u) \\
\end{smallmatrix} \right]$ $\forall $ $(w, s):A(T')\amalg B(U')\longrightarrow K$. Then for $(x,y)\in A(T)\amalg B(U)$, we have that

\begin{align*}
\Big((w,s)\circ \left[\begin{smallmatrix}
 A(t )& 0 \\
m & B(u) \\
\end{smallmatrix} \right]\Big)\Big(\left[\begin{smallmatrix}
 x \\
 \\
y \\
\end{smallmatrix} \right]\Big) & =(w,s)\Big(\left[\begin{smallmatrix}
A(t)(x) \\
 \\
m\cdot x+ B(u)(y )\\
\end{smallmatrix} \right]\Big)\\
& =w(A(t)(x))+s(m\cdot x+B(u)(y))\\
& =w(A(t)(x))+s(m\cdot x)+s(B(u)(y)).
\end{align*}
$\textbf{Morphisms}$. Now, consider $(\alpha,\beta):(A,f,B)\longrightarrow (A',f',B')$ 
in $\Big(\mathsf{Mod}(\mathcal{T}), \mathbb{G}(\mathsf{Mod}(\mathcal{U})\Big).$
Then, we have that $\textswab{F}(\alpha,\beta)=\alpha\amalg \beta$ is the following natural transformation
$$\alpha\amalg \beta\! =\!\left\{\!\!
(\alpha\!\amalg\!\beta)_{{}_{\left[
\begin{smallmatrix}
T& 0 \\
M& U\\
\end{smallmatrix}
\right] }}\!\!:=\!\alpha_{T}\!\amalg \!\beta_{U}\!\!:\!\! (A\amalg_{f} B) (\left [\begin{smallmatrix}
T& 0 \\
M& U\\
\end{smallmatrix} \right])\rightarrow (A'\amalg_{f'}\!\!B')(\left [\begin{smallmatrix}
T& 0 \\
M& U\\
\end{smallmatrix} \right])\!\!\right\}_{\!\!\left[\begin{smallmatrix}
T& 0 \\
M& U\\
\end{smallmatrix} \right ]\in \mathbf{\Lambda}}$$
Then $\mathbb{D}_{\mathbf{\Lambda}}(\alpha\amalg \beta):\mathbb{D}_\mathbf{\Lambda}(A'\amalg_{f'} B')\longrightarrow \mathbb{D}_\mathbf{\Lambda}(A \amalg_{f} B)$ is the natural transformation 
such that for  $\left[\begin{smallmatrix}
T& 0 \\
M& U\\
\end{smallmatrix} \right ]\in \mathbf{\Lambda}^{op}$ the correspoding morphism is
$$\Big[\mathbb{D}_{\mathbf{\Lambda}}(\alpha\amalg \beta)\Big]
_{{}_{\left[
\begin{smallmatrix}
T & 0 \\
M & U\\
\end{smallmatrix}
\right] }}:\mathrm{Hom}_{K}(A'(T)\amalg B'(U),K)\longrightarrow \mathrm{Hom}_{K}(A(T)\amalg B(U),K).$$
For simplicity, we write $\Big[\delta\Big]_{{}_{\left[
\begin{smallmatrix}
T & 0 \\
M & U\\
\end{smallmatrix}
\right] }}:=\Big[\Big(\mathbb{D}_{\mathbf{\Lambda}}\circ \textswab{F}\Big)(\alpha,\beta)\Big]_{{}_{\left[
\begin{smallmatrix}
T & 0 \\
M & U\\
\end{smallmatrix}
\right] }}=\Big[\mathbb{D}_{\mathbf{\Lambda}}(\alpha\amalg \beta)\Big]
_{{}_{\left[
\begin{smallmatrix}
T & 0 \\
M & U\\
\end{smallmatrix}
\right] }}.$
We note that, for $h\in \mathrm{Hom}_{K}(A'(T)\amalg B'(U),K)$, we have that 
$$\Big[\delta\Big]_{{}_{\left[
\begin{smallmatrix}
T & 0 \\
M & U\\
\end{smallmatrix}
\right] }}(h):=h \circ (\alpha_{T}\amalg \beta_{U}).$$

\begin{note}
For $\left[\begin{smallmatrix}
T& 0 \\
M& U\\
\end{smallmatrix} \right]\in \mathbf{\Lambda}^{op}$  we have an isomorphism in $\mathsf{Mod}(K)$
$$\Big[\nu_{(A,f,B)}\Big]_{\left[\begin{smallmatrix}
T& 0 \\
M& U\\
\end{smallmatrix} \right]}:(\mathbb{D}_{\mathcal{U}}B)(U)\amalg (\mathbb{D}_{\mathcal{T}}A)(T)\longrightarrow \mathrm{Hom}_{K}(A(T)\amalg B(U),K)$$
given as follows: let $(s',w')\in (\mathbb{D}_{\mathcal{U}}B)(U)\amalg (\mathbb{D}_{\mathcal{T}}A)(T)$ be, then $\Big[\nu_{(A,f,B)}\Big]\!_{\left[\begin{smallmatrix}
T& 0 \\
M& U\\
\end{smallmatrix} \right]}\!\Big(\!s',w'\!\Big)\!=(w',s'):A(T)\amalg B(U)\longrightarrow K$ is defined by $\Big((w',s')\Big)(x,y):=w'(x)+s'(y),$ $\forall $ $(x,y)\in A(T)\amalg B(U)$.
\end{note}
The following proposition give us a relation between the two functors described above. It tell us that they are the same up to a natural equivalence.

\begin{proposition}\label{cuadrocasicon}
Let $\mathcal{U}$ and $\mathcal{T}$ be Hom-finite $K$-categories.
Consider the contravariant functors 
$$\mathbb{T}^{\ast}\circ \overline{\textswab{F}}\circ \widehat{\Theta},\,\,\, \mathbb{D}_{\mathbf{\Lambda}}\circ \textswab{F}:\Big(\mathsf{Mod}(\mathcal{T}), \mathbb{G}(\mathsf{Mod}(\mathcal{U})\Big)\longrightarrow \mathsf{Mod}(\mathbf{\Lambda}^{op}).$$
Then, there exists an isomorphism $\nu:\mathbb{T}^{\ast}\circ \overline{\textswab{F}}\circ \widehat{\Theta}\longrightarrow \mathbb{D}_{\mathbf{\Lambda}}\circ \textswab{F}.$ That is, the following diagram is commutative up to  the isomorphism $\nu$
$$\xymatrix{
\Big(\mathsf{Mod}(\mathcal{T}),\mathbb{G}\mathsf{Mod}(\mathcal{U})\Big)\ar[rr]^{\textswab{F}}\ar[d]_{\widehat{\Theta}} & & \mathsf{Mod}(\mathbf{\Lambda})\ar[d]^{\mathbb{D}_{\mathbf{\Lambda}}}\\
\Big(\mathsf{Mod}(\mathcal{U}^{op}),\overline{\mathbb{G}}\mathsf{Mod}(\mathcal{T}^{op})\Big)\ar@{=>}[urr]_{\nu}\ar[rr]_{\mathbb{T}^{\ast}\circ \overline{\textswab{F}}} & &   \mathsf{Mod}(\mathbf{\Lambda}^{op}).}$$
\end{proposition}
\begin{proof}
Let $(A,f,B)$ be an object in $\Big(\mathsf{Mod}(\mathcal{T}), \mathbb{G}(\mathsf{Mod}(\mathcal{U})\Big)$. First, let us see  that $\nu_{(A,f,B)}:\Big(\mathbb{T}^{\ast}\circ \overline{\textswab{F}}\circ \widehat{\Theta}\Big)(A,f,B)\longrightarrow \Big(\mathbb{D}_{\mathbf{\Lambda}}\circ \textswab{F}\Big)(A,f,B),$ is a morphism in
$\mathsf{Mod}(\mathbf{\Lambda}^{op}).$
Indeed, for $\left[\begin{smallmatrix}
T& 0 \\
M& U\\
\end{smallmatrix} \right ]\in \mathbf{\Lambda}^{op}$ we have that 

\begin{align*}
\Big(\mathbb{T}^{\ast}\circ \overline{\textswab{F}}\circ \widehat{\Theta}\Big)(A,f,B) \Big(
\left[\begin{smallmatrix}
T & 0 \\
M & U \\
\end{smallmatrix} \right ]\Big) & =
\mathbb{T}^{\ast}\Big(\mathbb{D}_{\mathcal{U}}(B)\amalg_{\overline{f}}\mathbb{D}_{\mathcal{T}}(A)\Big)\Big(\left[\begin{smallmatrix}
T& 0 \\
M& U\\
\end{smallmatrix} \right ]\Big) \\
& =\Big(\mathbb{D}_{\mathcal{U}}(B)\amalg_{\overline{f}}\mathbb{D}_{\mathcal{T}}(A)\Big)\Big(\left[\begin{smallmatrix}
U& 0 \\
\overline{M} & T \\
\end{smallmatrix} \right ]\Big)\\
& = \mathrm{Hom}_{K}(B(U),K)\amalg \mathrm{Hom}_{K}(A(T) ,K).
\end{align*}
and
\begin{align*}
\Big(\Big(\mathbb{D}_{\mathbf{\Lambda}}\circ \textswab{F}\Big)(A,f,B)\Big)(\left[\begin{smallmatrix}
T & 0 \\
M & U \\
\end{smallmatrix} \right ]) & =\Big(\mathbb{D}_{\mathbf{\Lambda}}(A\amalg_{f}B)\Big)\Big(
\left[\begin{smallmatrix}
T & 0 \\
M & U \\
\end{smallmatrix} \right ]\Big)\\
& =\mathrm{Hom}_{K}(A(T)\amalg B(U),K).
\end{align*}
Let $\left[\begin{smallmatrix}
t& 0 \\
m & u \\
\end{smallmatrix} \right ]^{op}:\left[\begin{smallmatrix}
T'& 0 \\
M & U' \\
\end{smallmatrix} \right ]\longrightarrow \left[\begin{smallmatrix}
T& 0 \\
M & U \\
\end{smallmatrix} \right ]$ be a morphism in $\mathbf{\Lambda}^{op}$ with $m\in M(U',T)$.\\
We have to show that the following diagram commutes in $\mathbf{Ab}$
$$\xymatrix{\mathrm{Hom}_{K}(B(U'),K)\amalg \mathrm{Hom}_{K}(A(T') ,K)\ar[rrr]^{\Big[\nu_{(A,f,B)}\Big]_{\left[\begin{smallmatrix}
T'& 0 \\
M& U'\\
\end{smallmatrix} \right]}}\ar[d]^{\Big(\mathbb{D}_{\mathcal{U}}(B)\amalg_{\overline{f}}\mathbb{D}_{\mathcal{T}}(A)\Big)\Big({{\left[\begin{smallmatrix}
u^{op}& 0 \\
m & t^{op}\\
\end{smallmatrix} \right]}} \Big)} && &  \mathrm{Hom}_{K}(A(T')\amalg B(U'),K)\ar[d]_{\mathbb{D}_{\mathbf{\Lambda}}(A\amalg_{f}B)\Big({{\left[\begin{smallmatrix}
t& 0 \\
m& u\\
\end{smallmatrix} \right]^{op}}} \Big)} \\
\mathrm{Hom}_{K}(B(U),K)\amalg \mathrm{Hom}_{K}(A(T) ,K)\ar[rrr]_{\Big[\nu_{(A,f,B)}\Big]_{\left[\begin{smallmatrix}
T& 0 \\
M& U\\
\end{smallmatrix} \right]}} & & &  \mathrm{Hom}_{K}(A(T)\amalg B(U),K)}$$
Indeed, let $(s,w)\in \mathrm{Hom}_{K}(B(U'),K)\amalg \mathrm{Hom}_{K}(A(T') ,K)$. Then 
$$\Big(\mathbb{D}_{\mathcal{U}}(B)\amalg_{\overline{f}}\mathbb{D}_{\mathcal{T}}(A)\Big)\Big({{\left[\begin{smallmatrix}
u^{op}& 0 \\
m & t^{op}\\
\end{smallmatrix} \right]}} \Big)\left[\begin{smallmatrix}
s\\
\\
w\\
\end{smallmatrix} \right]=\left [\begin{smallmatrix}
\mathbb{D}_{\mathcal{U}}(B)(u^{op})(s) \\
m\cdot s+ \mathbb{D}_{\mathcal{T}}(A)(t^{op})(w) \\
\end{smallmatrix} \right].$$
Then for $(x,y)\in A(T)\amalg B(U)$ we have that
\begin{align*}
& \Big(\Big[\nu_{(A,f,B)}\Big]_{\left[\begin{smallmatrix}
T& 0 \\
M& U\\
\end{smallmatrix} \right]}\Big(\left [\begin{smallmatrix}
\mathbb{D}_{\mathcal{U}}(B)(u^{op})(s) \\
m\cdot s+ \mathbb{D}_{\mathcal{T}}(A)(t^{op})(w) \\
\end{smallmatrix} \right]\Big)\Big)(x,y)\\
& = \Big(m\cdot s+ \mathbb{D}_{\mathcal{T}}(A)(t^{op})(w)\Big)(x)+\Big(\mathbb{D}_{\mathcal{U}}(B)(u^{op})(s)\Big)(y)
\end{align*}
On the other hand, $\mathbb{D}_{\mathbf{\Lambda}}(A\amalg_{f} B)\Big(\left[\begin{smallmatrix}
 t & 0 \\
m & u \\
\end{smallmatrix} \right]^{op}\Big):=\mathrm{Hom}_{K}\Big(\left[\begin{smallmatrix}
 A(t )& 0 \\
m & B(u) \\
\end{smallmatrix} \right],K\Big).$
Hence, for $(s,w)\in \mathrm{Hom}_{K}(B(U'),K)\amalg \mathrm{Hom}_{K}(A(T') ,K)$ we get that 
\begin{align*}
\Big(\!\mathbb{D}_{\mathbf{\Lambda}}(A\amalg_{f}B)\!\big({{\left[\begin{smallmatrix}
t& 0 \\
m& u\\
\end{smallmatrix} \right]^{op}}} \big)\!\!\circ\!\!
\Big[\nu_{(A,f,B)}\Big]_{\left[\begin{smallmatrix}
T'& 0 \\
M& U'\\
\end{smallmatrix} \right]}\Big)(s,w) & =\Big(\mathbb{D}_{\mathbf{\Lambda}}(A\amalg_{f}B)\Big)\Big({{\left[\begin{smallmatrix}
t& 0 \\
m& u\\
\end{smallmatrix} \right]^{op}}} \Big)(w,s)\\
& \! = \mathrm{Hom}_{K}\Big(\left[\begin{smallmatrix}
 A(t )& 0 \\
m & B(u) \\
\end{smallmatrix} \right],K\Big)\Big((w,s)\Big)\\
& =(w,s)\circ \left[\begin{smallmatrix}
 A(t )& 0 \\
m & B(u) \\
\end{smallmatrix} \right]
\end{align*}
Then for $(x,y)\in A(T)\amalg B(U)$, we have that
\begin{align*}
\Big((w,s)\circ \left[\begin{smallmatrix}
 A(t )& 0 \\
m & B(u) \\
\end{smallmatrix} \right]\Big)\Big(\left[\begin{smallmatrix}
 x \\
 \\
y \\
\end{smallmatrix} \right]\Big) & =(w,s)\Big(\left[\begin{smallmatrix}
A(t)(x) \\
 \\
m\cdot x+ B(u)(y )\\
\end{smallmatrix} \right]\Big)\\
& =w(A(t)(x))+s(m\cdot x+B(u)(y))\\
& =w(A(t)(x))+s(m\cdot x)+s(B(u)(y))
\end{align*}
Then
\begin{align*}
\Big[\nu_{(A,f,B)}\Big]_{\left[\begin{smallmatrix}
T& 0 \\
M& U\\
\end{smallmatrix} \right]}\Big(\left [\begin{smallmatrix}
\mathbb{D}_{\mathcal{U}}(B)(u^{op})(s) \\
m\cdot s+ \mathbb{D}_{\mathcal{T}}(A)(t^{op})(w) \\
\end{smallmatrix} \right]\Big) & =\Big(\mathrm{Hom}_{K}\Big(\left[\begin{smallmatrix}
 A(t )& 0 \\
m & B(u) \\
\end{smallmatrix} \right],K\Big)\Big)\Big((w,s)\Big).
\end{align*}
Proving that the required diagram commutes and thus $\nu_{(A,f,B)}:\Big(\mathbb{T}^{\ast}\circ \overline{\textswab{F}}\circ \widehat{\Theta}\Big)(A,f,B)\longrightarrow \Big(\mathbb{D}_{\mathbf{\Lambda}}\circ \textswab{F}\Big)(A,f,B),$ is a morphism in
$\mathsf{Mod}(\mathbf{\Lambda}^{op}).$ Now, let $(\alpha,\beta):(A,f,B)\longrightarrow (A',f',B')$ be  a morphism in $\Big(\mathsf{Mod}(\mathcal{T}), \mathbb{G}(\mathsf{Mod}(\mathcal{U})\Big)$. Let us see that the following diagram commutes in $\mathsf{Mod}(\mathbf{\Lambda}^{op})$
$$\xymatrix{
\Big(\mathbb{T}^{\ast}\circ \overline{\textswab{F}}\circ \widehat{\Theta}\Big)(A',f',B')\ar[rrr]^{\nu_{(A',f',B')}}\ar[d]_{\Big(\mathbb{T}^{\ast}\circ \overline{\textswab{F}}\circ \widehat{\Theta}\Big)(\alpha,\beta)} && & 
\Big(\mathbb{D}_{\mathbf{\Lambda}}\circ \textswab{F}\Big)(A',f',B')\ar[d]^{\Big(\mathbb{D}_{\mathbf{\Lambda}}\circ \textswab{F}\Big)(\alpha,\beta)}\\
\Big(\mathbb{T}^{\ast}\circ \overline{\textswab{F}}\circ \widehat{\Theta}\Big)(A,f,B)\ar[rrr]^{\nu_{(A,f,B)}} & & & \Big(\mathbb{D}_{\mathbf{\Lambda}}\circ \textswab{F}\Big)(A,f,B).}$$
We have to show that for $\left[\begin{smallmatrix}
T& 0 \\
M& U\\
\end{smallmatrix} \right ]\in \mathbf{\Lambda}^{op}$ the following diagram commutes
$$\xymatrix{\mathrm{Hom}_{K}(B'(U),K)\amalg \mathrm{Hom}_{K}(A'(T),K) \ar[rrr]^(.55){\Big[\nu_{(A',f',B')}\Big]_{\left[\begin{smallmatrix}
T& 0 \\
M& U\\
\end{smallmatrix} \right]}}\ar[d]^{\Big[
\zeta \Big]_{\left[\begin{smallmatrix}
T& 0 \\
M& U\\
\end{smallmatrix} \right]}} && & \mathrm{Hom}_{K}(A'(T)\amalg B'(U),K)\ar[d]^{\Big[ \delta \Big]_{\left[\begin{smallmatrix}
T& 0 \\
M& U\\
\end{smallmatrix} \right]}}\\
\mathrm{Hom}_{K}(B(U),K)\amalg \mathrm{Hom}_{K}(A(T),K)\ar[rrr]_(.6){\Big[\nu_{(A,f,B)}\Big]_{\left[\begin{smallmatrix}
T& 0 \\
M& U\\
\end{smallmatrix} \right]}} & & &  \mathrm{Hom}_{K}(A(T)\amalg B(U),K).}$$
Indeed, let $(s,w)\in \mathrm{Hom}_{K}(B'(U),K)\amalg \mathrm{Hom}_{K}(A'(T),K)$. Thus, we get 
\begin{align*}
\Big(\Big[\nu_{(A,f,B)}\Big]_{\left[\begin{smallmatrix}
T& 0 \\
M& U\\
\end{smallmatrix} \right]}\circ 
\Big[ \zeta \Big]_{\left[\begin{smallmatrix}
T& 0 \\
M& U\\
\end{smallmatrix} \right]}\Big)(s,w)=\Big[\nu_{(A,f,B)}\Big]_{\left[\begin{smallmatrix}
T& 0 \\
M& U\\
\end{smallmatrix} \right]}\Big((s\circ \beta_{U},w\circ \alpha_{T})\Big).
\end{align*}
where
$\Big[\nu_{(A,f,B)}\Big]_{\left[\begin{smallmatrix}
T& 0 \\
M& U\\
\end{smallmatrix} \right]}\Big((s\circ \beta_{U},w\circ \alpha_{T})\Big):A(T)\amalg B(U)\longrightarrow K$ is defined as follows:
$\Big(\Big[\nu_{(A,f,B)}\Big]_{\left[\begin{smallmatrix}
T& 0 \\
M& U\\
\end{smallmatrix} \right]}\Big((s\circ \beta_{U},w\circ \alpha_{T})\Big)\Big)(x,y)=w(\alpha_{T}(x))+s(\beta_{U}(y))$ for every $(x,y)\in A(T)\amalg B(U)$.\\
On the other side, $
\Big(\Big[\delta \Big]_{\left[\begin{smallmatrix}
T& 0 \\
M& U\\
\end{smallmatrix} \right]}\circ 
\Big[\nu_{(A',f',B')}\Big]_{\left[\begin{smallmatrix}
T& 0 \\
M& U\\
\end{smallmatrix} \right]}\Big)(s,w)=\Big(\Big[\nu_{(A',f',B')}\Big]_{\left[\begin{smallmatrix}
T& 0 \\
M& U\\
\end{smallmatrix} \right]}(s,w)\Big)\circ (\alpha_{T}\amalg \beta_{U})
$ and $\Big(\Big[\nu_{(A',f',B')}\Big]_{\left[\begin{smallmatrix}
T& 0 \\
M& U\\
\end{smallmatrix} \right]}(s,w)\Big)\circ (\alpha_{T}\amalg \beta_{U}):A(T)\amalg B(U)\longrightarrow K$
is defined as follows: for $(x,y)\in A(T)\amalg B(U)$ we have  that 
\begin{align*}
 \Big(\Big(\Big[\nu_{(A',f',B')}\Big]_{\left[\begin{smallmatrix}
T& 0 \\
M& U\\
\end{smallmatrix} \right]}(s,w)\Big)\circ  (\alpha_{T}\amalg  & \beta_{U})\Big)  (x,y)= \\
& =\Big(\Big(\Big[\nu_{(A',f',B')}\Big]_{\left[\begin{smallmatrix}
T& 0 \\
M& U\\
\end{smallmatrix} \right]}(s,w)\Big)(\alpha_{T}(x),\beta_{U}(y))\\
& =(w,s)(\alpha_{T}(x),\beta_{U}(y))\\
& =w(\alpha_{T}(x))+s(\beta_{U}(y)).
\end{align*}
Proving that the required diagrama commutes. Therefore
$$\nu:=\!\!\Big \{\nu_{(A,f,B)}\!\!:\!\!\Big(\mathbb{T}^{\ast}\circ \overline{\textswab{F}}\circ \widehat{\Theta}\Big)(A,f,B)\longrightarrow \Big(\mathbb{D}_{\mathbf{\Lambda}}\circ \textswab{F}\Big)(A,f,B)\Big \}_{(A,f,B)\in \Big(\mathsf{Mod}(\mathcal{T}),\mathbb{G}\mathsf{Mod}(\mathcal{U})\Big)}$$
defines an isomorphism $\nu:\mathbb{T}^{\ast}\circ \overline{\textswab{F}}\circ \widehat{\Theta}\longrightarrow \mathbb{D}_{\mathbf{\Lambda}}\circ \textswab{F}$.
\end{proof}




\section{An adjoint pair}
In this section we will see that the functor $\mathbb{G}$ has a left adjoint $\mathbb{F}$ and we will prove that there is an isomorphism of comma categories 
$\Big( \mathbb{F}(\mathsf{Mod}(\mathcal{T})),\mathsf{Mod}(\mathcal{U})\Big) \simeq \Big( \mathsf{Mod}(\mathcal{T}), \mathbb{G}(\mathsf{Mod}(\mathcal{U}))\Big)$ (see \ref{equicoma2}).\\
Firts, let us consider the contravariant functor $E:\mathcal{T}\longrightarrow \mathsf{Mod}(\mathcal{U})$ defined in  \ref{dosfuntores}. By \cite[Theorem 6.3]{Popescu} in page 101, there exists a unique functor (up to isomorphism of functors) $\mathbb{F}:\mathsf{Mod}(\mathcal{T})\longrightarrow \mathsf{Mod}(\mathcal{U})$  which commutes with direct limits and such that
\begin{enumerate}
\item [(a)] $\mathbb{F}\circ Y\simeq E$, where $Y:\mathcal{T}\longrightarrow \mathsf{Mod}(\mathcal{T})$ is the Yoneda functor.

\item [(b)] $\mathbb{F}$ has a right adjoint.
\end{enumerate}
By checking the proof in \cite[Theorem 6.3]{Popescu}, it can be see that the adjoint of $\mathbb{F}$ is our functor $\mathbb{G}:\mathsf{Mod}(\mathcal{U})\longrightarrow \mathsf{Mod}(\mathcal{T})$. That is, we have that the pair $(\mathbb{F},\mathbb{G})$ is an adjoint pair. That is, for every $A\in \mathsf{Mod}(\mathcal{T})$ and $B\in \mathsf{Mod}(\mathcal{U})$ there exists a natural isomorphism
$$\varphi_{A,B}:\mathrm{Hom}_{\mathsf{Mod}(\mathcal{U})}(\mathbb{F}(A),B)\longrightarrow \mathrm{Hom}_{\mathsf{Mod}(\mathcal{T})}(A,\mathbb{G}(B)).$$
In the same way, considering the contravariant functor $\overline{E}:\mathcal{U}^{op}\longrightarrow \mathsf{Mod}(\mathcal{T}^{op})$ defined in  \ref{otrosdosfun}, there exists a unique functor (up to isomorphism of functors) $\overline{\mathbb{F}}:\mathsf{Mod}(\mathcal{U}^{op})\longrightarrow \mathsf{Mod}(\mathcal{T}^{op})$  which commutes with direct limits and such that
\begin{enumerate}
\item [(a)] $\overline{\mathbb{F}}\circ Y\simeq E$, where $\overline{Y}:\mathcal{U}^{op}\longrightarrow \mathsf{Mod}(\mathcal{U}^{op})$ is the Yoneda functor.

\item [(b)] $\overline{\mathbb{F}}$ has a right adjoint.
\end{enumerate}
Moreover, the adjoint of $\overline{\mathbb{F}}$ is our functor $\overline{\mathbb{G}}:\mathsf{Mod}(\mathcal{T}^{op})\longrightarrow \mathsf{Mod}(\mathcal{U}^{op})$. That is, we have that the pair $(\overline{\mathbb{F}},\overline{\mathbb{G}})$ is an adjoint pair. That is, for every $A\in \mathsf{Mod}(\mathcal{U}^{op})$ and $B\in \mathsf{Mod}(\mathcal{T}^{op})$ there exists a natural isomorphism
$$\overline{\varphi}_{A,B}:\mathrm{Hom}_{\mathsf{Mod}(\mathcal{T}^{op})}(\overline{\mathbb{F}}(A),B)\longrightarrow \mathrm{Hom}_{\mathsf{Mod}(\mathcal{U}^{op})}(A,\overline{\mathbb{G}}(B)).$$
Now, consider $\mathbb{D}_{\mathcal{U}^{op}}:\mathsf{Mod}(\mathcal{U}^{op})\longrightarrow \mathsf{Mod}(\mathcal{U})$ and  $\mathbb{D}_{\mathcal{T}}:\mathsf{Mod}(\mathcal{T})\longrightarrow \mathsf{Mod}(\mathcal{T}^{op})$. Therefore we get
$$\mathbb{D}_{\mathcal{T}}\circ \mathbb{G}\circ \mathbb{D}_{\mathcal{U}^{op}}:\mathsf{Mod}(\mathcal{U}^{op})\longrightarrow \mathsf{Mod}(\mathcal{T}^{op}).$$

\begin{proposition}
Let $\mathcal{U}$ and $\mathcal{T}$ be Hom-finite $K$-categories.
There exist  natural isomorphisms
$$\Psi:\overline{\mathbb{F}} \longrightarrow \mathbb{D}_{\mathcal{T}}\circ \mathbb{G}\circ \mathbb{D}_{\mathcal{U}^{op}},\quad \Xi:\mathbb{F} \longrightarrow \mathbb{D}_{\mathcal{U}^{op}}\circ \overline{\mathbb{G}}\circ \mathbb{D}_{\mathcal{T}}.$$
\end{proposition}
\begin{proof}
We will prove just the first isomorphism, since the second is analogous.\\
Let us see that there exists an isomorphism of $\mathcal{T}^{op}$-modules
$$\Psi_{\mathrm{Hom}_{\mathcal{U}}(-,U)}: M_{U}=\overline{\mathbb{F}}(\mathrm{Hom}_{\mathcal{U}}(-,U))\longrightarrow  \Big(\mathbb{D}_{\mathcal{T}}\circ \mathbb{G}\circ \mathbb{D}_{\mathcal{U}^{op}}\Big)(\mathrm{Hom}_{\mathcal{U}}(-,U)).$$
Let $T\in \mathcal{T}^{op}$ be, we want 
$$\Big[\Psi_{\mathrm{Hom}_{\mathcal{U}}(-,U)}\Big]_{T}: M(U,T)\longrightarrow  
\mathrm{Hom}_{K}\Big(\mathrm{Hom}_{\mathsf{Mod}(\mathcal{U})}(M_{T},\mathbb{D}_{\mathcal{U}^{op}}(\mathrm{Hom}_{\mathcal{U}}(-,U))),K\Big).$$

Therefore, for $m\in M(U,T)$ we define
$$\Big[\Psi_{\mathrm{Hom}_{\mathcal{U}}(-,U)}\Big]_{T}(m):\mathrm{Hom}_{\mathsf{Mod}(\mathcal{U})}\Big(M_{T},\mathbb{D}_{\mathcal{U}^{op}}(\mathrm{Hom}_{\mathcal{U}}(-,U))\Big)\longrightarrow K,$$
as follows: 
$$\Big(\Big[\Psi_{\mathrm{Hom}_{\mathcal{U}}(-,U)}\Big]_{T}(m)\Big)(\eta):=\Big(\Big[\mathbb{D}_{\mathcal{U}}(\eta)\circ \Gamma_{\mathrm{Hom}_{\mathcal{U}}(-,U)}\Big]_{U}(1_{U})\Big)(m)=\Big(\eta_{U}(m)\Big)(1_{U})$$ for every $\eta: M_{T}\longrightarrow \mathbb{D}_{\mathcal{U}^{op}}(\mathrm{Hom}_{\mathcal{U}}(-,U))$.\\
Let $t^{op}:T'\longrightarrow T$ be a morphism in $\mathcal{T}^{op}$.
We assert that the following diagram commutes
$$\xymatrix{M(U,T')\ar[rrr]^(.4){[\Psi_{\mathrm{Hom}_{\mathcal{U}}(-,U)}]_{T'}}\ar[d] & & & \Big((\mathbb{D}_{\mathcal{T}}\circ \mathbb{G}\circ \mathbb{D}_{\mathcal{U}^{op}})(\mathrm{Hom}_{\mathcal{U}}(-,U))\Big)(T')\ar[d]\\
M(U,T)\ar[rrr]_(.4){[\Psi_{\mathrm{Hom}_{\mathcal{U}}(-,U)}]_{T}} & &  &\Big((\mathbb{D}_{\mathcal{T}}\circ \mathbb{G}\circ \mathbb{D}_{\mathcal{U}^{op}})(\mathrm{Hom}_{\mathcal{U}}(-,U))\Big)(T)}$$
Indeed, let $B:=\mathrm{Hom}_{\mathcal{U}}(-,U)$ and $m\in M(U,T')$ we have that

\begin{align*}
& \Big(\Big((\mathbb{D}_{\mathcal{T}}\circ \mathbb{G}\circ \mathbb{D}_{\mathcal{U}^{op}})(B)\Big)(t^{op})\Big)([\Psi_{\mathrm{Hom}_{\mathcal{U}}(-,U)}]_{T'}(m)):=\\
& = [\Psi_{\mathrm{Hom}_{\mathcal{U}}(-,U)}]_{T'}(m)\circ \mathrm{Hom}_{\mathsf{Mod}(\mathcal{U}^{op})}(\overline{t},\mathbb{D}_{\mathcal{U}^{op}}(B)).
\end{align*}
For $\delta:M_{T}\longrightarrow \mathbb{D}_{\mathcal{U}^{op}}(B)$,  we get
$$\mathrm{Hom}_{\mathsf{Mod}(\mathcal{U}^{op})}(\overline{t},\mathbb{D}_{\mathcal{U}^{op}}(B))(\delta)=\delta\circ \overline{t}:M_{T'}\longrightarrow \mathbb{D}_{\mathcal{U}^{op}}(B)$$
Then 
\begin{align*}
[\Psi_{\mathrm{Hom}_{\mathcal{U}}(-,U)}]_{T'}(m)\Big(\mathrm{Hom}_{\mathsf{Mod}(\mathcal{U}^{op})}(\overline{t},\mathbb{D}_{\mathcal{U}^{op}}(B))(\delta)\Big) & =\Big([\Psi_{\mathrm{Hom}_{\mathcal{U}}(-,U)}]_{T'}(m)\Big)(\delta\circ \overline{t})\\
& =\Big([\delta\circ \overline{t}]_{U}(m)\Big)(1_{U})\\
& = \Big( [\delta]_{U}([\overline{t}]_{U}(m))\Big)(1_{U})\\
& =\Big([\delta]_{U}(M(1_{U}\otimes t)(m))\Big)(1_{U})\\
& =\Big([\delta]_{U}(M_{U}(t)(m))\Big)(1_{U})\\
& =\Big([\Psi_{\mathrm{Hom}_{\mathcal{U}}(-,U)}]_{T}
(M_{U}(t)(m))\Big)(\delta)
\end{align*}
Hence, the required diagram is commutative.\\
Then 
$$\Psi_{\mathrm{Hom}_{\mathcal{U}}(-,U)}: M_{U}=\overline{\mathbb{F}}(\mathrm{Hom}_{\mathcal{U}}(-,U))\longrightarrow  \Big(\mathbb{D}_{\mathcal{T}}\circ \mathbb{G}\circ \mathbb{D}_{\mathcal{U}^{op}}\Big)(\mathrm{Hom}_{\mathcal{U}}(-,U)),$$
is an isomorphism of $\mathcal{T}^{op}$-modules.

Now, we have to show that if
 $\mathrm{Hom}_{\mathcal{U}}(-,u):\mathrm{Hom}_{\mathcal{U}}(-,U)\longrightarrow \mathrm{Hom}_{\mathcal{U}}(-,U')$  is a morphism in $\mathsf{Mod}(\mathcal{U}^{op})$, then for each $T\in\mathcal{T}^{op}$ the following diagram commutes
$$\xymatrix{M(U,T)=\overline{\mathbb{F}}(\mathrm{Hom}_{\mathcal{U}}(-,U))(T)\ar[d]_{M(u\otimes 1_{T})}\ar[rr]^(.45){[\Psi_{\mathrm{Hom}_{\mathcal{U}}(-,U)}]_{T}} & & 
 \Big((\mathbb{D}_{\mathcal{T}}\circ \mathbb{G}\circ \mathbb{D}_{\mathcal{U}^{op}})(\mathrm{Hom}_{\mathcal{U}}(-,U))\Big)(T)\ar[d]^{L}\\
 M(U',T)=\overline{\mathbb{F}}(\mathrm{Hom}_{\mathcal{U}}(-,U'))(T)\ar[rr]^(.45){[\Psi_{\mathrm{Hom}_{\mathcal{U}}(-,U')}]_{T}} & & \Big((\mathbb{D}_{\mathcal{T}}\circ \mathbb{G}\circ \mathbb{D}_{\mathcal{U}^{op}})(\mathrm{Hom}_{\mathcal{U}}(-,U'))\Big)(T)}$$
 where $L=\Big[(\mathbb{D}_{\mathcal{T}}\circ \mathbb{G}\circ \mathbb{D}_{\mathcal{U}^{op}})(\mathrm{Hom}_{\mathcal{U}}(-,u))\Big]_{T}$.\\
For $m\in M(U,T)$ we  obtain that
$$L\Big(\Big[\Psi_{\mathrm{Hom}_{\mathcal{U}}(-,U)}\Big]_{T}(m)\Big)=\Big(\Big[\Psi_{\mathrm{Hom}_{\mathcal{U}}(-,U)}\Big]_{T}(m)\Big)\circ \Big( \mathrm{Hom}_{\mathsf{Mod}(\mathcal{U})}\Big(M_{T},\mathbb{D}_{\mathcal{U}^{op}}(\mathrm{Hom}_{\mathcal{U}}(-,u))\Big)\Big).$$
For  $\eta: M_{T}\longrightarrow \mathbb{D}_{\mathcal{U}^{op}}(\mathrm{Hom}_{\mathcal{U}}(-,U'))$, we have that
$\mathbb{D}_{\mathcal{U}^{op}}(\mathrm{Hom}_{\mathcal{U}}(-,u))\circ \eta: M_{T}\longrightarrow \mathbb{D}_{\mathcal{U}^{op}}(\mathrm{Hom}_{\mathcal{U}}(-,U)).$
Therefore,
\begin{align*}
& \Big(\Big(\Big[\Psi_{\mathrm{Hom}_{\mathcal{U}}(-,U)}\Big]_{T}(m)\Big)\circ \Big( \mathrm{Hom}_{\mathsf{Mod}(\mathcal{U})}\Big(M_{T},\mathbb{D}_{\mathcal{U}^{op}}(\mathrm{Hom}_{\mathcal{U}}(-,u))\Big)\Big)\Big)(\eta)=\\
&=\Big(\Big[\Psi_{\mathrm{Hom}_{\mathcal{U}}(-,U)}\Big]_{T}(m)\Big)\Big( \mathrm{Hom}_{\mathsf{Mod}(\mathcal{U})}\Big(M_{T},\mathbb{D}_{\mathcal{U}^{op}}(\mathrm{Hom}_{\mathcal{U}}(-,u))\Big)(\eta)\Big)=\\
& =\Big(\Big[\Psi_{\mathrm{Hom}_{\mathcal{U}}(-,U)}\Big]_{T}(m)\Big)(\mathbb{D}_{\mathcal{U}^{op}}(\mathrm{Hom}_{\mathcal{U}}(-,u))\circ \eta)\\
 &=\Big(\Big[\mathbb{D}_{\mathcal{U}^{op}}(\mathrm{Hom}_{\mathcal{U}}(-,u))\circ \eta\Big]_{U}(m)\Big)(1_{U})\\
& = \Big(\Big[\mathbb{D}_{\mathcal{U}^{op}}(\mathrm{Hom}_{\mathcal{U}}(-,u))\Big]_{U}([\eta ]_{U}(m))\Big)(1_{U})\\
& = \Big([\eta]_{U}(m)\Big)(u\circ 1_{U})=\Big([\eta]_{U}(m)\Big)(u).
\end{align*}
On the other hand, let $m':=M(u\otimes 1_{T})(m)\in M(U',T)$.  Then for $\eta: M_{T}\longrightarrow \mathbb{D}_{\mathcal{U}^{op}}(\mathrm{Hom}_{\mathcal{U}}(-,U'))$
we have that
\begin{align*}
\Big[\Psi_{\mathrm{Hom}_{\mathcal{U}}(-,U')}\Big]_{T}(m')(\eta)=\eta_{U'}(m')(1_{U'}).
\end{align*}
Since $\eta_{T}: M_{T}\longrightarrow \mathbb{D}_{\mathcal{U}^{op}}(\mathrm{Hom}_{\mathcal{U}}(-,U'))$ is a morphism of $\mathcal{U}$-modules we have that
$$\Big(\mathrm{Hom}_{K}(\mathrm{Hom}_{\mathcal{U}}(u,U'),K)\Big)(\eta_{U}(m))=\eta_{U'}(m').$$
Therefore we get that $\eta_{U'}(m')(1_{U'})=\eta_{U}(m)(u)$. Proving that the required diagram commutes. Therefore, we have an isomorphism
$$\Psi:\overline{\mathbb{F}}|_{\mathrm{Proj}(\mathsf{Mod}(\mathcal{U}^{op}))}\longrightarrow (\mathbb{D}_{\mathcal{T}}\circ \mathbb{G}\circ \mathbb{D}_{\mathcal{U}^{op}})|_{\mathrm{Proj}(\mathsf{Mod}(\mathcal{U}^{op}))}.$$
By \cite[Theorem 5.4]{MitBook}  and \cite[Corollary 5.5]{MitBook}, we extend $\Psi$  to an isomorphism
$$\Psi:\overline{\mathbb{F}} \longrightarrow \mathbb{D}_{\mathcal{T}}\circ \mathbb{G}\circ \mathbb{D}_{\mathcal{U}^{op}}.$$
\end{proof}

\begin{corollary}\label{corcambia}
Let $\mathcal{U}$ and $\mathcal{T}$ be Hom-finite $K$-categories
\begin{enumerate}
\item [(i)] For every $B\in \mathsf{Mod}(\mathcal{U}^{op})$ there exists an isomorphism $\mathbb{D}_{\mathcal{T}^{op}}\overline{\mathbb{F}}(B)\simeq \mathbb{G}\mathbb{D}_{\mathcal{U}^{op}}(B).$

\item[(ii)] For every $A\in \mathsf{Mod}(\mathcal{T})$ there exists an isomorphism $\mathbb{D}_{\mathcal{U}}\mathbb{F}(A)\simeq \overline{\mathbb{G}}\mathbb{D}_{\mathcal{T}}(A).$

\end{enumerate}
\end{corollary}
\begin{proof}
\end{proof}
Since for every $A\in \mathsf{Mod}(\mathcal{T})$ and $B\in \mathsf{Mod}(\mathcal{U})$ there exists a natural isomorphism
$$\varphi_{A,B}:\mathrm{Hom}_{\mathsf{Mod}(\mathcal{U})}(\mathbb{F}(A),B)\longrightarrow \mathrm{Hom}_{\mathsf{Mod}(\mathcal{T})}(A,\mathbb{G}(B)),$$
we have the comma category $ \Big(\mathbb{F}(\mathsf{Mod}(\mathcal{T})),\mathsf{Mod}(\mathcal{U})\Big)$
whose objects are the triples $ (A,g,B)$ with $A\in \mathsf{Mod}(\mathcal{T}), B\in \mathsf{Mod}(\mathcal{U})$
and $ g:\mathbb F(A)\longrightarrow B $ a morphism of $ \mathcal{U} $-modules.
A morphism between two objects $ (A,g,B) $ and $ (A',g',B') $ is a pairs of morphism $(\alpha,\beta) $ where $\alpha:A\longrightarrow A'$ is a morphism of $\mathcal{T}$-modules and
$\beta:B\longrightarrow B'$ is a morphism of $\mathcal{U}$-modules and such that the diagram

\[
\begin{diagram}
\node{\mathbb F(A)} \arrow{e,t}{\mathbb F(\alpha)}\arrow{s,l}{g}\node{ \mathbb F(A')}\arrow{s,r}{g'}\\
\node{B} \arrow{e,b}{\beta}\node{B'}
\end{diagram}
\]
commutes.\\
Now,  we define a functor  $H: \Big( \mathbb{F}(\mathsf{Mod}(\mathcal{T})),\mathsf{Mod}(\mathcal{U})\Big) \longrightarrow \Big( \mathsf{Mod}(\mathcal{T}), \mathbb{G}(\mathsf{Mod}(\mathcal{U}))\Big)$ by  $H(A,g,B)=(A,\varphi_{A,B}(g),B)$ on objects and $H((\alpha, \beta))=(\alpha, \beta)$  on morphisms. Clearly the functor
 $H': \Big( \mathsf{Mod}(\mathcal{T}), \mathbb{G}(\mathsf{Mod}(\mathcal{U}))\Big)\longrightarrow \Big(\mathbb{F}(\mathsf{Mod}(\mathcal{T})),\mathsf{Mod}(\mathcal{U})\Big) $,
 given by $H'(A,g,B)=(A,\varphi_{A,B}^{-1}(g),B)$ on objects and  $H'((\alpha, \beta))=(\alpha, \beta)$ on morphisms is an inverse of $H$. Then, we have the following

\begin{proposition}\label{equicoma2}
Let $\mathcal{U}$ and $\mathcal{T}$ be Hom-finite $K$-categories.
\begin{enumerate}
\item [(a)]There exists an isomorphism 
$$H: \Big( \mathbb{F}(\mathsf{Mod}(\mathcal{T})),\mathsf{Mod}(\mathcal{U})\Big) \longrightarrow \Big( \mathsf{Mod}(\mathcal{T}), \mathbb{G}(\mathsf{Mod}(\mathcal{U}))\Big),$$

\item [(b)] There exists an isomorphism  
$$\overline{H}: \Big(\overline{\mathbb{F}}(\mathsf{Mod}(\mathcal{U}^{op})),\mathsf{Mod}(\mathcal{T}^{op})\Big) \longrightarrow \Big( \mathsf{Mod}(\mathcal{U}^{op}), \overline{\mathbb{G}}(\mathsf{Mod}(\mathcal{T}^{op}))\Big).$$
\end{enumerate}
\end{proposition}
\begin{proof}
\end{proof}
Now, the following theorem characterizes the projective objects in the category  $\Big( \mathsf{Mod}(\mathcal{T}), \mathbb{G}\mathsf{Mod}(\mathcal{U})\Big)$.
\begin{proposition}\label{projecinmaps}
Consider the projective $\mathbf \Lambda$-module, $P:=\mathrm{Hom}_{\mathbf{\Lambda}}\Big(\left[
\begin{smallmatrix}
T&0 \\
M&U
\end{smallmatrix} \right],-\Big):\Lambda \rightarrow \mathbf{Ab}$ and the morphism of $\mathcal T$-modules 
$$g:\mathrm{Hom}_{\mathcal{T}}(T,-)\longrightarrow \mathbb{G}\Big( M_T\amalg \mathrm{Hom}_{\mathcal{U}}(U,-)\Big )$$
given by $g:=\Big \{[g]_{T'}: \mathrm{Hom}_{\mathcal{T}}(T,T')\longrightarrow \mathrm{Hom}_{\mathcal {U}}\big(M_{T'},M_{T}\amalg \mathrm{Hom}_{\mathcal{U}}(U,-)\big)\Big \}_{T'\in\mathcal {T}}$, with  $[g]_{T'}(t):=\left[
\begin{smallmatrix}
\overline{t} \\
0
\end{smallmatrix} \right]:M_{T'}\rightarrow M_{T}\amalg \mathrm{Hom}_{\mathcal{U}}(U,-)$ for all $t\in \mathrm{Hom}_{\mathcal{T}}(T,T')$, where for $U'\in \mathcal{U}$ we have that
$\Big[\left[
\begin{smallmatrix}
\overline{t} \\
0
\end{smallmatrix} \right]\Big]_{U'}:M(U',T')\rightarrow M(U',T)\amalg \mathrm{Hom}_{\mathcal{U}}(U,U')$ is such that  for $m\in M(U',T)$ we get 
$\Big[\left[
\begin{smallmatrix}
\overline{t} \\
0
\end{smallmatrix} \right]\Big]_{U'}(m):=(M(1_{U}\otimes t^{op})(m),0)=(m\bullet t,0).$ Then
$$P\cong \big(\mathrm{Hom}_{\mathcal{T}}(T,-)\big)\underset{g}\amalg \big(M_T\amalg \mathrm{Hom}_{\mathcal{U}}(U,-)\big).$$
\end{proposition}
\begin{proof}
We are going to use the notation of \ref{lemmorcom} and remark \ref{notaC1C2}. Firstly, we define a morphism of $\mathcal{T}$-modules $\alpha: \mathrm{Hom}_{\mathcal{T}}(T,-)\longrightarrow 
P_{1}$. Indeed, for $T'\in \mathcal{T}$ we  define
$$\alpha_{T'}:\mathrm{Hom}_{\mathcal{T}}(T,T')\longrightarrow  P_{1}(T')=
\mathrm{Hom}_{\mathbf{\Lambda}}\Big(\left[
\begin{smallmatrix}
T&0 \\
M&U
\end{smallmatrix} \right],  \left[
\begin{smallmatrix}
T'&0 \\
M&0
\end{smallmatrix} \right] \Big)$$
as $\alpha_{T'}(t)= \left[
\begin{smallmatrix}
t & 0 \\
0 & 0
\end{smallmatrix} \right].$
It is straighforward that $\alpha$ is a morphism of $\mathcal{T}$-modules.\\
Secondly,  we define a morphism of $\mathcal{U}$-modules
$\beta:M_{T}\amalg \mathrm{Hom}_{\mathcal{U}}(U,-)\longrightarrow P_{2}$ as follows. For $U'\in \mathcal{U}$
$$\beta_{U'}:M(U',T)\amalg \mathrm{Hom}_{\mathcal{U}}(U,U')\longrightarrow  P_{2}(U')=
\mathrm{Hom}_{\mathbf{\Lambda}}\Big(\left[
\begin{smallmatrix}
T&0 \\
M&U
\end{smallmatrix} \right],  \left[
\begin{smallmatrix}
0 &0 \\
M& U'
\end{smallmatrix} \right] \Big)$$
is defined as
$\beta_{U'}(m,u)= \left[
\begin{smallmatrix}
0 & 0 \\
m & u
\end{smallmatrix} \right].$\\
It is straighforward that $\beta$ is a morphism of $\mathcal{U}$-modules.\\
Finally, we have to show that for each $T'\in \mathcal{T}$ the following diagram commutes
$$\xymatrix{\mathrm{Hom}(T,T')\ar[rr]^{[\alpha]_{T'}}\ar[d]_{[g]_{T'}} & & {\mathrm{Hom}_{\mathbf{\Lambda}}\Big(\left[
\begin{smallmatrix}
T&0 \\
M&U
\end{smallmatrix} \right],  \left[
\begin{smallmatrix}
T'&0 \\
M&0
\end{smallmatrix} \right] \Big)}\ar[d]^{f_{T'}}\\
\mathrm{Hom}_{\mathcal {U}}\big(M_{T'},M_{T}\amalg \mathrm{Hom}_{\mathcal{U}}(U,-)\big)\ar[rr]^{[\mathbb{G}(\beta)]_{T'}} & & 
\mathrm{Hom}_{\mathcal {U}}\big(M_{T'},P_{2}\big)}
$$
Indeed, let $t:T\longrightarrow T'$, then
$[\alpha]_{T'}(t):\left[
\begin{smallmatrix}
T&0 \\
M&U
\end{smallmatrix} \right]\longrightarrow \left[
\begin{smallmatrix}
T'&0 \\
M&0
\end{smallmatrix} \right]$. Then, for $U'\in \mathcal{U}$ we have that 
$$\Big[[f]_{T'}\Big([\alpha]_{T'}(t)\Big)\Big]_{U'}:M(U',T')\longrightarrow P_{2}(U')=\mathrm{Hom}_{\mathbf{\Lambda}}\Big(\left[
\begin{smallmatrix}
T&0 \\
M&U
\end{smallmatrix} \right],  \left[
\begin{smallmatrix}
0 &0 \\
M& U'
\end{smallmatrix} \right] \Big)$$ 
is defined as 
$$\Big[[f]_{T'}\Big([\alpha]_{T'}(t)\Big)\Big]_{U'}(m)=P(\overline{m})([\alpha]_{T'}(t))=\overline{m}\circ ([\alpha]_{T'}(t))=
\left[
\begin{smallmatrix}
0&0 \\
m& 0
\end{smallmatrix} \right]
\left[
\begin{smallmatrix}
t &0 \\
0 & 0
\end{smallmatrix} \right]=\left[
\begin{smallmatrix}
0&0 \\
m\bullet t & 0
\end{smallmatrix} \right].$$ for $m\in M(U',T')$, where $\overline{m}=\left[
\begin{smallmatrix}
0&0 \\
m&0
\end{smallmatrix} \right]:\left[
\begin{smallmatrix}
T'&0 \\
M&0
\end{smallmatrix} \right]\longrightarrow \left[
\begin{smallmatrix}
0&0 \\
M&U'
\end{smallmatrix} \right].$
On the other hand, we get
$[g]_{T'}(t):=\left[
\begin{smallmatrix}
\overline{t} \\
0
\end{smallmatrix} \right]:M_{T'}\rightarrow M_{T}\amalg \mathrm{Hom}_{\mathcal{U}}(U,-)$
and for $U'\in \mathcal{U}$ we have that 
$$\Big[\beta \circ \Big(\left[
\begin{smallmatrix}
\overline{t} \\
0
\end{smallmatrix} \right]\Big)\Big]_{U'}:M(U',T')\longrightarrow P_{2}(U')=\mathrm{Hom}_{\mathbf{\Lambda}}\Big(\left[
\begin{smallmatrix}
T&0 \\
M&U
\end{smallmatrix} \right],  \left[
\begin{smallmatrix}
0 &0 \\
M& U'
\end{smallmatrix} \right] \Big)$$ 
is defined as 
$\Big[\beta \circ \Big(\left[
\begin{smallmatrix}
\overline{t} \\
0
\end{smallmatrix} \right]\Big)\Big]_{U'}(m)=\beta_{U'}(m\bullet t,0)=\left[
\begin{smallmatrix}
0&0 \\
m\bullet t &0
\end{smallmatrix} \right]$
for $m\in M(U',T')$. Therefore the required diagram commutes. Since $\alpha$  and $\beta$ are isomorphism we get that
$$(\alpha,\beta):\Big(\mathrm{Hom}_{\mathcal{T}}(T,-),g, M_T\amalg \mathrm{Hom}_{\mathcal{U}}(U,-)\Big)\longrightarrow (P_{1},f,P_{2})$$ is an isomorphism. Therefore  we get 
$$P\cong P_{1}\amalg_{f} P_{2}\cong  \big(\mathrm{Hom}_{\mathcal{T}}(T,-)\big)\underset{g}\amalg \big(M_T\amalg \mathrm{Hom}_{\mathcal{U}}(U,-)\big).$$
\end{proof}

Recall that if $\mathcal{C}$ is a Hom-finite $K$-category. We denote by $\mathsf{mod}(\mathcal{C})$ the full subcategory of $\mathsf{Mod}(\mathcal{C})$ whose objects are the finitely presented functors.That is, $M\in \mathsf{mod}(\mathcal{C})$ if and only if, there exists an exact sequence in $\mathsf{Mod}(\mathcal{C})$
$$\xymatrix{P_{1}\ar[r] & P_{0}\ar[r] & M\ar[r] & 0,}$$
where $P_{1}$ and $P_{0}$ are finitely generated projective $\mathcal{C}$-modules. Since $\mathsf{mod}(\mathcal{C})\subseteq (\mathcal{C},\mathrm{mod}(K))$,  we have that $A\in \mathsf{mod}(\mathcal{T})$ if and only if there exists an exact sequence
$$\xymatrix{\mathrm{Hom}_{\mathcal{T}}(T_{1},-)\ar[r]^{\alpha'} & \mathrm{Hom}_{\mathcal{T}}(T_{0},-)\ar[r]^(.7){\alpha} &  A\ar[r] & 0,}$$
with $T_{1},T_{0}\in \mathcal{T}$.

\begin{proposition}\label{exacsedos}
A sequence of maps 
$$\xymatrix{0 \ar[r] & (A,f,B)\ar[r]^{(\alpha,\beta)} & (A',f',B')\ar[r]^{(\alpha',\beta')} & (A'',f'',B'')\ar[r] & 0}$$ is exact  in
 $\Big( \mathsf{Mod}(\mathcal{T}), \mathbb{G}\mathsf{Mod}(\mathcal{U})\Big)$ if and only if the following are exact sequences
 $\xymatrix{0\ar[r] & A\ar[r]^{\alpha} & A'\ar[r]^{\alpha} & A\ar[r] & 0},\quad$
$\xymatrix{0\ar[r] & B\ar[r]^{\beta} & B'\ar[r]^{\beta} & B''\ar[r] & 0}$ 
in $\mathsf{Mod}(\mathcal{T})$ and $\mathsf{Mod}(\mathcal{U})$ respectively.
 \end{proposition}
\begin{proof}
It is straightforward.
\end{proof}

\begin{proposition}\label{epiintocateg}
Let $(A,f,B)$ be an object in $\Big( \mathsf{Mod}(\mathcal{T}), \mathbb{G}\mathsf{Mod}(\mathcal{U})\Big)$ and
let $\alpha:\mathrm{Hom}_{\mathcal{T}}(T,-)\longrightarrow A$ an epimorphism in $\mathsf{Mod}(\mathcal{T})$  and 
$\beta:\mathrm{Hom}_{\mathcal{U}}(U,-)\longrightarrow B$ an epimorphism in $\mathsf{Mod}(\mathcal{U})$. Then there exists an epimorphism in $\mathsf{Mod}\Big(\left[
\begin{smallmatrix}
\mathcal{T}&0 \\
M& \mathcal{U}
\end{smallmatrix} \right]\Big)$
$$\Gamma: P=\mathrm{Hom}_{\mathbf{\Lambda}}\Big(\left[
\begin{smallmatrix}
T&0 \\
M&U
\end{smallmatrix} \right],-\Big)\longrightarrow \textswab{F}(A,f,B).$$ 
\end{proposition}
\begin{proof}
Let us construct an epimorphism
$(\theta,\psi): \Big(\mathrm{Hom}_{\mathcal{T}}(T,-),g, M_T\amalg \mathrm{Hom}_{\mathcal{U}}(U,-)\Big)\longrightarrow (A,f,B).$
That is, we want morphisms $\theta: \mathrm{Hom}_{\mathcal{T}}(T,-)\longrightarrow A$ and $\psi:  M_T\amalg \mathrm{Hom}_{\mathcal{U}}(U,-)\longrightarrow B$ such that the following diagram commutes
$$\xymatrix{\mathrm{Hom}_{\mathcal{T}}(T,-)\ar[rr]^{\theta}\ar[d]_{g} & & A\ar[d]^{f}\\
\mathbb{G}\Big( M_T\amalg \mathrm{Hom}_{\mathcal{U}}(U,-)\Big )\ar[rr]^{\mathbb{G}(\psi)} & & \mathbb{G}(B).}$$
First, we will construct $\psi: M_{T}\amalg \mathrm{Hom}_{\mathcal{U}}(U,-)\longrightarrow B.$
We  have $\rho:=f_{T}\Big(\alpha_{T}(1_{T})\Big):M_{T}\longrightarrow B$.
Therefore, we define
$$\psi:=(\rho,\beta):M_{T}\amalg \mathrm{Hom}_{\mathcal{U}}(U,-)\longrightarrow B$$
We note that for each $U'\in \mathcal{U}$ we have that
 $$[\psi]_{U'}(m,u)=\rho_{U'}(m)+\beta_{U'}(u)=\Big[f_{T}\Big(\alpha_{T}(1_{T})\Big)\Big ]_{U'}(m)+\beta_{U'}(u)$$ for 
every $(m,u)\in M(U',T)\amalg \mathrm{Hom}_{\mathcal{U}}(U,U')$.
Now, since $\beta$ is an epimorphism we get that $\psi$ is an 
epimorphism.\\
We define $\theta:=\alpha$. Let us check that the following diagramm commutes for each $T'\in \mathcal{T}$
$$\xymatrix{\mathrm{Hom}(T,T')\ar[rr]^{[\alpha]_{T'}}\ar[d]_{[g]_{T'}} & & A(T')\ar[d]^{f_{T'}}\\
\mathrm{Hom}_{\mathsf{Mod}(\mathcal {U})}\big(M_{T'},M_{T}\amalg \mathrm{Hom}_{\mathcal{U}}(U,-)\big)\ar[rr]^{[\mathbb{G}(\psi)]_{T'}} & & 
\mathrm{Hom}_{\mathsf{Mod}(\mathcal {U})}\big(M_{T'},B\big)}
$$
Indeed, for $t:T\longrightarrow T'$ we have 
$[g]_{T'}(t)=\left[
\begin{smallmatrix}
\overline{t} \\
0
\end{smallmatrix} \right]:M_{T'}\rightarrow M_{T}\amalg \mathrm{Hom}_{\mathcal{U}}(U,-)$ and 
for $U'\in \mathcal{U}$ we get $\Big[[\mathbb{G}(\psi)]_{T'}\Big([g]_{T'}(t)\Big)\Big]_{U}=[\psi]_{U'}\circ \Big[\left[
\begin{smallmatrix}
\overline{t} \\
0
\end{smallmatrix} \right]\Big]_{U'}:M(U',T')\longrightarrow B(U').$
Then, for $m\in M(U',T')$ we obtain that

\begin{align*}
\Big([\psi]_{U'}\circ \Big[\left[
\begin{smallmatrix}
\overline{t} \\
0
\end{smallmatrix} \right]\Big]_{U'}\Big)(m)=[\psi]_{U'}\Big(
\Big[\left[
\begin{smallmatrix}
\overline{t} \\
0
\end{smallmatrix} \right]\Big]_{U'}(m)\Big)
& =[\psi]_{U'}(m\bullet t,0)\\
& =\Big[f_{T}\Big(\alpha_{T}(1_{T})\Big)\Big ]_{U'}(m\bullet  t).
\end{align*}
On the other hand, since $\alpha:\mathrm{Hom}_{\mathcal{T}}(T,-)\longrightarrow A$ and $f:A\longrightarrow \mathbb{G}(B)$ are morphisms of $\mathcal{T}$-modules we have the following commutative diagram in $\mathbf{Ab}$
$$\xymatrix{\mathrm{Hom}_{\mathcal{T}}(T,T)\ar[r]^{\alpha_{T}}\ar[d]_{\mathrm{Hom}_{\mathcal{T}}(T,t)} & A(T)\ar[rr]^{f_{T}}\ar[d]^{A(t)} & & \mathrm{Hom}_{\mathsf{Mod}(\mathcal{U})}(M_{T},B)\ar[d]^{\mathrm{Hom}_{\mathsf{Mod}(\mathcal{U})}(\overline{t},B)}\\
\mathrm{Hom}_{\mathcal{T}}(T,T')\ar[r]^{\alpha_{T'}} & A(T')\ar[rr]^{f_{T'}} &  &\mathrm{Hom}_{\mathsf{Mod}(\mathcal{U})}(M_{T'},B).}$$
Therefore, we get that $\mathrm{Hom}_{\mathcal{T}}(T,t)(1_{T})=t$. 
In this way we obtain that
\begin{align*}
f_{T'}([\alpha]_{T'}(t))=f_{T'}\Big([\alpha]_{T'}\big(\mathrm{Hom}_{\mathcal{T}}(T,t)(1_{T})\big)\Big)& =\mathrm{Hom}_{\mathsf{Mod}(\mathcal{U})}(\overline{t},B)\Big(f_{T}\Big(\alpha_{T}(1_{T})\Big)\Big)\\
& =f_{T}\Big(\alpha_{T}(1_{T})\Big)\circ \overline{t}.
\end{align*}
Then,  for $U'\in \mathcal{U}$ and $m\in M(U',T')$  we have that
\begin{align*}
\Big(\Big[f_{T'}([\alpha]_{T'}(t))\Big]_{U'}\Big)(m) &=\Big(\Big[f_{T}\Big(\alpha_{T}(1_{T})\Big)\Big]_{U'}\circ \Big[\overline{t}\Big]_{U'}\Big)(m)\\
&=\Big[f_{T}\Big(\alpha_{T}(1_{T})\Big)\Big]_{U'}\Big(\Big[\overline{t}\Big]_{U'}(m)\Big)\\
& = \Big[f_{T}\Big(\alpha_{T}(1_{T})\Big)\Big]_{U'}(m\bullet t)
\end{align*}
Proving that the required diagramm commutes.\\
Now, since $\alpha$ and $\beta$ are epimorphisms, it follows by \ref{exacsedos} that 
$$\textswab{F}(\alpha,\psi):\textswab{F}\Big(\mathrm{Hom}_{\mathcal{T}}(T,-),g, M_T\amalg \mathrm{Hom}_{\mathcal{U}}(U,-)\Big)\longrightarrow \textswab{F}(A,f,B)$$ 
is an epimorphism.
\end{proof}

\begin{lemma}\label{fpresented1}
Let $\mathcal{U}$ and $\mathcal{T}$ be Hom-finite $K$-categories $M\in \mathsf{Mod}(\mathcal{U}\otimes \mathcal{T}^{op})$ and $\mathbb{F}$ the left adjoint to $\mathbb{G}$. Then,
\begin{itemize}
\item[(i)]$\mathbb{F}(\mathrm{Hom}_{\mathcal{T}}(T,-))\cong M_{T}$.

\item[(ii)] Assume $M_{T}\in\mathsf{mod}(\mathcal U)$ for all $T\in\mathcal{T}$. Then  $A\in\mathrm{mod}(\mathcal{T})$ implies
 $\mathbb{F}(A)\in\mathsf{mod}(\mathcal{U})$.
\end{itemize} 
\end{lemma}
\begin{proof}
\begin{enumerate}
\item [(i)] Since $\mathbb{F}\circ Y\simeq E$ where $Y:\mathcal{T}\longrightarrow 
\mathsf{Mod}(\mathcal{T})$ is the Yoneda functor, we get that
$M_{T}=E(T)\simeq (\mathbb{F}\circ Y)(E)=\mathbb{F}(\mathrm{Hom}_{\mathcal{T}}(T,-))$.

\item [(ii)]
Suppose that $A\in\mathrm{mod}(\mathcal{T})$. Then, there exists an exact sequence
$$\xymatrix{\mathrm{Hom}_{\mathcal{T}}(T_{1},-)\ar[r] & 
\mathrm{Hom}_{\mathcal{T}}(T_{0},-)\ar[r] & A\ar[r] & 0}$$
Since $(\mathbb{F},\mathbb{G})$ is an adjoint pair we have that
$\mathbb{F}$ is right exact. Then we get
$$\xymatrix{M_{T_{1}}\ar[r] & 
M_{T_{0}}\ar[r] & \mathbb{F}(A)\ar[r] & 0}$$
Since $M_{T_{0}},M_{T_{1}}\in \mathsf{mod}(\mathcal{U})$, by \cite[4.2 (b)]{AusM1}, we get that $\mathbb{F}(A)\in \mathsf{mod}(\mathcal{U})$.
\end{enumerate}
\end{proof}

\begin{lemma}\label{fpresented2}
Let $\mathcal{U}$ and $\mathcal{T}$ be Hom-finite $K$-categories $\overline{M}\in \mathsf{Mod}(\mathcal{T}^{op}\otimes \mathcal{U})$ 
and $\overline{\mathbb{F}}$ the left adjoint to $\overline{\mathbb{G}}$. Then
\begin{itemize}
\item[(i)]$\overline{\mathbb{F}}(\mathrm{Hom}_{\mathcal{U}}(-,U))\cong M_{U}$

\item[(ii)] Assume $M_{U}\in\mathsf{mod}(\mathcal{T}^{op})$ for all $U \in\mathcal{U}$. Then  $B\in\mathrm{mod}(\mathcal{U}^{op})$ implies
 $\overline{\mathbb{F}}(B)\in\mathsf{mod}(\mathcal{T}^{op})$.
\end{itemize} 
\end{lemma}
\begin{proof}
The same proof as in \ref{fpresented1}.
\end{proof}

\section{Dualizing Varieties}
The pourpose of this section is to show that under certain conditions on $\mathcal{U}$ and $\mathcal{T}$, we can restric the diagram in the theorem \ref{cuadrocasicon} to the category of finitely presented modules and as a consecuence we will have that $\mathbf{\Lambda}:=\left[\begin{smallmatrix}
\mathcal{T} &0 \\
M&\mathcal{U}
\end{smallmatrix} \right]$ is dualizing.\\
In order to give a description for the category $\mathrm{mod}(\mathbf{\Lambda})$ inside $\Big( \mathsf{Mod}(\mathcal{T}), \mathbb{G}\mathsf{Mod}(\mathcal{U})\Big)$,  we will give conditions
on the functor $M$  and on the categories  $\mathcal U$ and $ \mathcal T$  to obtain a subcategory $\mathcal X\subset \Big( \mathsf{Mod}(\mathcal{T}), \mathbb{G}\mathsf{Mod}(\mathcal{U})\Big)$ for which the restriction functor $\textswab{F}|_{\mathcal X}:\mathcal X\rightarrow \mathrm{Mod}(\mathbf{\Lambda})$  induces a equivalence of subcategories between $\mathcal X$ and $\mathrm{mod}(\mathbf{\Lambda})$.\\

\begin{note}
From now on, we will assume the $\mathcal{U}$ and $\mathcal{T}$ are dualizing $K$-varieties and in this case we are going to take $M\in \mathrm{Mod}(\mathcal{U}\otimes_{K}\mathcal{T}^{op})$ (see \ref{productensorR}).
\end{note}

We recall that a Hom-finite $K$-category $\mathcal{U}$ is a dualizing $K$-variety if the duality
$\mathbb{D}_{\mathcal{U}}:(\mathcal{C}, \mathsf{mod}(K))\longrightarrow (\mathcal{C}^{op}, \mathsf{mod}(K))$ restricts to a duality
$$\mathbb{D}_{\mathcal{U}}:\mathsf{mod}(\mathcal{U})\longrightarrow \mathsf{mod}(\mathcal{U}^{op}).$$
Assume that $\mathcal U$  and $\mathcal T$ are dualizing varietes  and that $M_T\in \mathrm{mod}(\mathcal U)$, for all $T\in\mathcal T$ and that
$M_U\in \mathrm{mod}(\mathcal T ^{op})$. Then $\mathrm{mod}(\mathcal U)$ and  $\mathrm{mod}(\mathcal T)$  are abelian categories, and  by Lemma
\ref{fpresented1}  the restriction $\mathbb{F}^{*}:=\mathbb{F}|_{\mathrm{mod}(\mathcal{T})}:\mathrm{mod}(\mathcal{T})\rightarrow  \mathrm{Mod}(\mathcal{U})$ has image in $\mathrm{mod}(\mathcal{U})$. Similarly by \ref{fpresented2} we have a functor 
 $\overline{\mathbb{F}}^{*}:=\overline{\mathbb{F}}|_{\mathrm{mod}(\mathcal{U}^{op})}:\mathrm{mod}(\mathcal{U}^{op})\rightarrow  \mathrm{mod}(\mathcal{T}^{op})$.

\begin{lemma}\label{resringebien}
Let $\mathcal{U}$ and $\mathcal{T}$ be dualizing $K$-varieties and
$M\in \mathrm{Mod}(\mathcal{U}\otimes_{K}\mathcal{T}^{op})$. Assume that $M_{T}\in\mathsf{mod}(\mathcal{U})$ for all $T\in\mathcal{T}$ and $M_{U}\in\mathsf{mod}(\mathcal{T}^{op})$ for all $U\in\mathcal{U}$. Then for all $B\in \mathsf{mod}(\mathcal{U})$ we get that $\mathbb{G}(B)\in \mathsf{mod}(\mathcal{T})$. That is if
$\mathbb{G}^{\ast}:=\mathbb{G}|_{\mathsf{mod}(\mathcal{U})}$ we have that
$$\mathbb{G}^{\ast}:\mathsf{mod}(\mathcal{U})\longrightarrow \mathsf{mod}(\mathcal{T}).$$
In the same way we have if
$\overline{\mathbb{G}}^{\ast}:=\overline{\mathbb{G}}|_{\mathsf{mod}(\mathcal{T}^{op})}$ we have that
$$\overline{\mathbb{G}}^{\ast}:\mathsf{mod}(\mathcal{T}^{op})\longrightarrow \mathsf{mod}(\mathcal{U}^{op}).$$
\end{lemma}
\begin{proof}
Since $\mathcal{U}$ is dualizing we have that
$B\simeq \mathbb{D}_{\mathcal{U}^{op}}\mathbb{D}_{\mathcal{U}}(B)$ with $\mathbb{D}_{\mathcal{U}}(B)\in \mathsf{mod}(\mathcal{U}^{op})$. Then by \ref{corcambia}(i), we have that
$\mathbb{G}(B)\simeq (\mathbb{G}\mathbb{D}_{\mathcal{U}^{op}})\mathbb{D}_{\mathcal{U}}(B)\simeq (\mathbb{D}_{\mathcal{T}^{op}}\overline{\mathbb{F}})\mathbb{D}_{\mathcal{U}}(B).$\\
Since $\mathbb{D}_{\mathcal{U}}(B)\in \mathsf{mod}(\mathcal{U}^{op})$ we have by lemma \ref{fpresented2}, that $\overline{\mathbb{F}}(\mathbb{D}_{\mathcal{U}}(B))\in \mathsf{mod}(\mathcal{T}^{op})$.
Since $\mathcal{T}$ is dualizing we get that $\mathbb{G}(B)\simeq
\mathbb{D}_{\mathcal{T}^{op}}(\overline{\mathbb{F}}(\mathbb{D}_{\mathcal{U}}(B))\in \mathsf{mod}(\mathcal{T})$.
\end{proof}
We denote by $\Big(\!\mathrm{mod}(\mathcal T), \mathbb{G}^{*}\mathrm{mod}(\mathcal U)\!\Big)$ the full subcategory of $ \Big(\!\mathrm{Mod}(\mathcal{T}), \mathbb{G}\mathrm{Mod}(\mathcal U)\!\Big)$ whose objects are the morphisms of $\mathcal{T}$-modules $A\xrightarrow{f} \mathbb{G}(B)$ for which $A\in \mathrm{mod}(\mathcal{T})$ and $B\in \mathrm{mod}(\mathcal{U})$. Similarly, we denote by $\Big(\mathbb{F}^{*}(\mathrm{mod}(\mathcal{T})), \mathrm{mod}(\mathcal U)\Big)$ the full subcategory of $ \Big(\mathbb{F}(\mathrm{Mod}(\mathcal{T})), \mathrm{Mod}(\mathcal{U})\Big)$
 whose objects are the morphisms of $\mathcal{U}$-modules $\mathbb{F}(A)\xrightarrow{g} B$ for which $A\in \mathrm{mod}(\mathcal{T})$ and $B\in \mathrm{mod}(\mathcal{U})$. Now,  we can restrict the morphism given in \ref{equicoma2} to subcategories.
 
\begin{proposition}
Let $\mathcal{U}$ and $\mathcal{T}$ be dualizing $K$-varieties
and assume that $M\in \mathrm{Mod}(\mathcal{U}\otimes_{K}\mathcal{T}^{op})$ satisfies that $M_{T}\in \mathsf{mod}(\mathcal{U})$ and  $M_{U}\in \mathsf{mod}(\mathcal{T}^{op})$ for all $T\in \mathcal{T}$ and $U\in \mathcal{U}^{op}$. Then, there exists equivalences
$$\Big(\mathbb{F}^{*}(\mathrm{mod}(\mathcal{T})), \mathrm{mod}(\mathcal{U})\Big)\longrightarrow \Big(\mathrm{mod}(\mathcal{T}), \mathbb{G}^{*}(\mathrm{mod}(\mathcal{U}))\Big),$$
$$\Big(\overline{\mathbb{F}}^{*}(\mathrm{mod}(\mathcal{U}^{op})), \mathrm{mod}(\mathcal{T}^{op})\Big)\longrightarrow \Big(\mathrm{mod}(\mathcal{U}^{op}), \overline{\mathbb{G}}^{*}(\mathrm{mod}(\mathcal{U}^{op}))\Big) .$$
 \end{proposition}
 \begin{proof}
 It follows from \ref{equicoma2}.
 \end{proof}

\begin{proposition}
Let $\mathcal{U}$ and $\mathcal{T}$ dualizing varieties and $M\in \mathsf{Mod}(\mathcal{U}\otimes_{K} \mathcal{T}^{op})$. Assume that $M_{T}\in \mathsf{mod}(\mathcal{U})$ and  $M_{U}\in \mathsf{mod}(\mathcal{T}^{op})$ for all $T\in \mathcal{T}$ and $U\in \mathcal{U}^{op}$. Then we get the functor $\textswab{F}|_{( \mathsf{mod}(\mathcal{T}), \mathbb{G}\mathsf{mod}(\mathcal{U}))}
:\Big( \mathsf{mod}(\mathcal{T}), \mathbb{G}\mathsf{mod}(\mathcal{U})\Big) \longrightarrow 
\mathrm{mod}(\mathbf{\Lambda})$ which is an equivalence.
\end{proposition}
\begin{proof}
By \ref{resringebien} we have that $ \mathbb{G}\mathsf{mod}(\mathcal{U})\subseteq \mathsf{mod}(\mathcal{T})$.\\
Let $(A,f,B)\in \Big( \mathsf{mod}(\mathcal{T}), \mathbb{G}\mathsf{mod}(\mathcal{U})\Big)$. Then $A\in \mathrm{mod}(\mathcal{T})$ and $B\in \mathrm{mod}(\mathcal{U})$. Hence, there exists exact sequences
$$(\ast):\xymatrix{\mathrm{Hom}_{\mathcal{T}}(T_{1},-)\ar[r]^{\alpha'} & \mathrm{Hom}_{\mathcal{T}}(T_{0},-)\ar[r]^(.7){\alpha} &  A\ar[r] & 0}$$
$$(\ast\ast):\xymatrix{\mathrm{Hom}_{\mathcal{U}}(U_{1},-)\ar[r]^{\beta'} & \mathrm{Hom}_{\mathcal{U}}(U_{0},-)\ar[r]^(.7){\beta} &  B\ar[r] & 0}$$
By \ref{epiintocateg}, we can construct an epimorphism
$$(\alpha,\psi): \Big(\mathrm{Hom}_{\mathcal{T}}(T_{0},-),g, M_{T_{0}}\amalg \mathrm{Hom}_{\mathcal{U}}(U_{0},-)\Big)\longrightarrow (A,f,B)$$
in $\Big( \mathsf{Mod}(\mathcal{T}), \mathbb{G}\mathsf{Mod}(\mathcal{U})\Big)$.

Now, let $i:K\longrightarrow \mathrm{Hom}_{\mathcal{U}}(U_{0},-)$ the kernel of $\beta:\mathrm{Hom}_{\mathcal{U}}(U_{0},-)\longrightarrow B$. Then we have the following exact sequence
$$\xymatrix{0\ar[r] & M_{T_{0}}\amalg K\ar[r]^(.4){\delta} & M_{T_{0}}\amalg \mathrm{Hom}_{\mathcal{U}}(U_{0},-)\ar[rr]^(.6){\psi=(\rho,\beta)} & & B\ar[r] & 0}$$
where $\delta=\left[
\begin{smallmatrix}
1_{M_{T_{0}}} & 0 \\
0 & i
\end{smallmatrix} \right]$. 
Let $j:L\longrightarrow \mathrm{Hom}_{\mathcal{U}}(T_{0},-)$ the kernel of $\alpha:\mathrm{Hom}_{\mathcal{U}}(T_{0},-)\longrightarrow A$.
We can define $\gamma:L\longrightarrow \mathbb{G}(M_{T_{0}}\amalg K)$ as follows: for each $T\in\mathcal{T}$ we  define
$$\gamma_{T}:L(T)\longrightarrow \mathrm{Hom}_{\mathsf{Mod}(\mathcal{U})}(M_{T}, M_{T_{0}}\amalg K)$$
For $x\in L(T)$ we have $t:=[j]_{T}(x):T_{0}\longrightarrow T$, then we have $\overline{t}:M_{T}\longrightarrow M_{T_{0}}$, then we define
$$\gamma_{T}(x):=\left[
\begin{smallmatrix}
\overline{t} \\
0
\end{smallmatrix} \right]:M_{T}\longrightarrow M_{T_{0}}\amalg K$$
We assert that the following diagram is commutative

$$\xymatrix{L\ar[rr]^{j}\ar[d]^{\gamma} & & \mathrm{Hom}_{\mathcal{T}}(T_{0},-)\ar[d]^{g}\\
\mathbb{G}(M_{T_{0}}\amalg K)\ar[rr]^{\mathbb{G}(\delta)} & &\mathbb{G}(M_{T_{0}}\amalg \mathrm{Hom}_{\mathcal{U}}(U_{0},-)}$$
For $T\in \mathcal{T}$

$$\xymatrix{L(T)\ar[r]^{[j]_{T}}\ar[d]^{[\gamma]_{T}} & \mathrm{Hom}_{\mathcal{T}}(T_{0},T)\ar[d]^{[g]_{T}}\\
\mathrm{Hom}_{\mathsf{Mod}(\mathcal{U})}(M_{T}, M_{T_{0}}\amalg K)\ar[r]^(.4){[\mathbb{G}(\delta)]_{T}} & \mathrm{Hom}_{\mathsf{Mod}(\mathcal{U})}(M_{T}, M_{T_{0}}\amalg \mathrm{Hom}_{\mathcal{U}}(U_{0},-))}$$

Indeed, we have that 
$[g]_{T}(t)=\left[
\begin{smallmatrix}
\overline{t} \\
0
\end{smallmatrix} \right]$. On the other side
$[\gamma]_{T}(x)=\left[
\begin{smallmatrix}
\overline{t} \\
0
\end{smallmatrix} \right]$ and then
$$[\mathbb{G}(\delta)]\left[
\begin{smallmatrix}
\overline{t} \\
0
\end{smallmatrix} \right]=\delta\circ \left[
\begin{smallmatrix}
\overline{t} \\
0
\end{smallmatrix} \right]=
\left[
\begin{smallmatrix}
1_{M_{T_{0}}} &0 \\
0  & i
\end{smallmatrix} \right]\left[
\begin{smallmatrix}
\overline{t} \\
0
\end{smallmatrix} \right]=\left[
\begin{smallmatrix}
\overline{t} \\
0
\end{smallmatrix} \right]$$
By \ref{exacsedos}, we have  an exact sequence
$$\xymatrix{0 \ar[r] & (L,\gamma, M_{T_{0}}\amalg K)\ar[r]^(.35){(j,\delta)} & (\mathrm{Hom}_{\mathcal{T}}(T_{0},-),g, M_{T_{0}}\amalg \mathrm{Hom}_{\mathcal{U}}(U_{0},-))\ar `[dl] `[l] `[lld] |{(\alpha,\psi)} `[l] [rlld]\\
& (A,f,B)\ar[r] & 0 &  }$$

Now, since $M_{T_{0}}\in \mathsf{mod}(\mathcal{U})$ and there exists an epimorphism  $\mathrm{Hom}_{\mathcal{U}}(U_{1},-)\longrightarrow K$ (see exact sequence $(\ast\ast)$, we conclude that $M_{T_{0}}\amalg K\in \mathsf{mod}(\mathcal{U})$. Hence, there exists an epimorphism $\nu:\mathrm{Hom}_{\mathcal{U}}(U,-)\longrightarrow M_{T_{0}}\amalg K$ for some $U\in\mathcal{U}$.  Considering the exact sequence $(\ast)$ we have an epimorphism $\mathrm{Hom}_{\mathcal{T}}(T_{1},-)\longrightarrow L$. Then we can apply  \ref{epiintocateg} to 
$(L,\gamma, M_{T_{0}}\amalg K)$. Therefore we can construct an epimorphism
$$Q=\mathrm{Hom}_{\mathbf{\Lambda}}\Big(\left[
\begin{smallmatrix}
T_{1}&0 \\
M&U
\end{smallmatrix} \right],-\Big)\longrightarrow \textswab{F}(L ,\gamma, M_{T_{0}}\amalg K).$$ 
Therefore we can construct an exact sequence
$$\mathrm{Hom}_{\mathbf{\Lambda}}\Big(\left[
\begin{smallmatrix}
T_{1} &0 \\
M&U
\end{smallmatrix} \right],-\Big)\longrightarrow \mathrm{Hom}_{\mathbf{\Lambda}}\Big(\left[
\begin{smallmatrix}
T_{0} &0 \\
M& U_{0}
\end{smallmatrix} \right],-\Big)\longrightarrow  \textswab{F}(A ,f,B)\longrightarrow 0$$
Proving that  $\textswab{F}(A ,f,B)\in \mathsf{mod}(\mathbf{\Lambda})$.\\
Now, since $\textswab{F}$ is full and faithfull, we hava that is restriction is also full and faithfull.\\
Now, let us see that  $\textswab{F}|_{( \mathsf{mod}(\mathcal{T}), \mathbb{G}\mathsf{mod}(\mathcal{U}))}
:\Big( \mathsf{mod}(\mathcal{T}), \mathbb{G}\mathsf{mod}(\mathcal{U})\Big) \longrightarrow 
\mathrm{mod}(\mathbf{\Lambda})$ is dense.\\
Let us consider an exact sequence $\xymatrix{Q\ar[r]^{F} & P\ar[r]^{G} & M\ar[r] & 0}$ where $M\simeq \textswab{F}(M_{1},f,M_{2})$ for some $A\in\mathrm{Mod}(\mathcal{T})$, $B\in \mathrm{Mod}(\mathcal{U})$ and $Q=\mathrm{Hom}_{\mathbf{\Lambda}}\Big(\left[
\begin{smallmatrix}
T_{1} &0 \\
M& U_{1}
\end{smallmatrix} \right],-\Big)$ and $P=\mathrm{Hom}_{\mathbf{\Lambda}}\Big(\left[
\begin{smallmatrix}
T_{0} &0 \\
M& U_{0}
\end{smallmatrix} \right],-\Big)$.
Then we have exact sequence

$$\xymatrix{ & \big(\mathrm{Hom}_{\mathcal{T}}(T_{1},-)\big)\underset{g_{1}}\amalg \big(M_{T_{1}}\amalg \mathrm{Hom}_{\mathcal{U}}(U_{1},-)\big) \ar `[ld] `[] `[ld] |{} `[] [ldr] & \\
& \big(\mathrm{Hom}_{\mathcal{T}}(T_{0},-)\big)\underset{g_{0}}\amalg \big(M_{T_{0}}\amalg \mathrm{Hom}_{\mathcal{U}}(U_{0},-)\big)\ar[r] & (M_{1},f,M_{2})\ar[r] &0.   }$$

By \ref{exacsedos}, we have the following exact sequences
$$\xymatrix{\mathrm{Hom}_{\mathcal{T}}(T_{1},-)\ar[r] & \mathrm{Hom}_{\mathcal{T}}(T_{0},-)\ar[r] & M_{1}\ar[r] & 0}$$
$$\xymatrix{M_{T_{1}}\amalg \mathrm{Hom}_{\mathcal{U}}(U_{1},-)\ar[r] & M_{T_{0}}\amalg \mathrm{Hom}_{\mathcal{U}}(U_{0},-)\ar[r] &  M_{2}\ar[r] & 0}$$
Now, since $M_{T_{1}}, M_{T_{2}}\in \mathsf{mod}(\mathcal{U})$, we conclude by \cite[4.2 (a), 4.2(d)]{AusM1}  that  $M_{T_{1}}\amalg \mathrm{Hom}_{\mathcal{U}}(U_{1},-),  M_{T_{0}}\amalg \mathrm{Hom}_{\mathcal{U}}(U_{0},-)\in 
\mathsf{mod}(\mathcal{U})$. By \cite[4.2 (b)]{AusM1}, we have that $M_{2}\in \mathsf{mod}(\mathcal{U})$. By the above exact sequence we have that $M_{1}\in \mathsf{mod}(\mathcal{T})$.
Proving that $(M_{1},f,M_{2})\in \Big( \mathsf{mod}(\mathcal{T}), \mathbb{G}\mathsf{mod}(\mathcal{U})\Big)$ and $M\simeq \textswab{F}(M_{1},f,M_{2})$. Proving that   $\textswab{F}|_{( \mathsf{mod}(\mathcal{T}), \mathbb{G}\mathsf{mod}(\mathcal{U}))}
:\Big( \mathsf{mod}(\mathcal{T}), \mathbb{G}\mathsf{mod}(\mathcal{U})\Big) \longrightarrow 
\mathrm{mod}(\mathbf{\Lambda})$ is dense. Therefore it is an equivalence.
\end{proof}

\begin{proposition}
Consider the functor given in \ref{morpduality}:
$$\widehat{\Theta}:\Big(\mathsf{Mod}(\mathcal{T}), \mathbb{G}(\mathsf{Mod}(\mathcal{U})\Big) \longrightarrow \Big(\mathsf{Mod}(\mathcal{U}^{op}), \overline{\mathbb{G}}(\mathsf{Mod}(\mathcal{T}^{op}))\Big).$$
Suppose that $\mathcal{U}$ and $\mathcal{T}$ are dualizing $K$-varieties and that $M\in \mathrm{Mod}(\mathcal{U}\otimes_{K}\mathcal{T}^{op})$ satisfies that $M_{T}\in \mathsf{mod}(\mathcal{U})$ and  $M_{U}\in \mathsf{mod}(\mathcal{T}^{op})$ for all $T\in \mathcal{T}$ and $U\in \mathcal{U}^{op}$. Then, we get a functor
$$\widehat{\Theta}^{\ast}:=\widehat{\Theta}|_{(\mathsf{mod}(\mathcal{T}), \mathbb{G}(\mathsf{mod}(\mathcal{U}))}:\Big(\mathsf{mod}(\mathcal{T}), \mathbb{G}(\mathsf{mod}(\mathcal{U})\Big) \longrightarrow \Big(\mathsf{mod}(\mathcal{U}^{op}), \overline{\mathbb{G}}(\mathsf{mod}(\mathcal{T}^{op}))\Big).$$
\end{proposition}
\begin{proof}
We recall that for $f:A\longrightarrow \mathbb{G}(B)$ a morphism of $\mathcal{T}$-modules, we have that
$$\overline{f}=\widehat{\Theta}(A,f,B):=\Theta(A,f,B)\circ \Psi_{B}:\mathbb{D}_{\mathcal{U}}(B)\longrightarrow \overline{\mathbb{G}}(\mathbb{D}_{\mathcal{T}}A).$$
Now, if $A\in \mathsf{mod}(\mathcal{T})$ we have that $\mathbb{D}_{\mathcal{T}}(A)\in \mathsf{mod}(\mathcal{T}^{op})$ and by \ref{resringebien} we get that $\overline{\mathbb{G}}(\mathbb{D}_{\mathcal{T}}A)\in \mathsf{mod}(\mathcal{U}^{op})$. Since $\mathcal{U}$ is dualizing we get that $\mathbb{D}_{\mathcal{U}}(B)\in \mathsf{mod}(\mathcal{U}^{op})$ if $B\in \mathsf{mod}(\mathcal{U})$. Therefore 
$\widehat{\Theta}(A,f,B)\in \Big(\mathsf{mod}(\mathcal{U}^{op}), \overline{\mathbb{G}}(\mathsf{mod}(\mathcal{T}^{op}))\Big)$
if $(A,f,B)\in \Big(\mathsf{mod}(\mathcal{T}), \mathbb{G}(\mathsf{mod}(\mathcal{U})\Big)$.
\end{proof}

\begin{Remark} \label{findimM}
Let $\mathcal{U}$ and $\mathcal{T}$ be $K$-categories and $M\in \mathrm{Mod}(\mathcal{U}\otimes_{K}\mathcal{T}^{op})$.
\begin{enumerate}
\item [(a)] For each $U\in \mathcal{U}$ and $T\in \mathcal{T}$ there exists ring morphisms 
$$\varphi:K\longrightarrow   \mathrm{End}_{\mathcal{U}}(U)\quad
k\mapsto k1_{U},\quad \psi:K\longrightarrow \mathrm{End}_{\mathcal{T}}(T), \quad
k\mapsto k1_{T},$$
such that the structure of $K$-vector spaces on $\mathrm{End}_{\mathcal{U}}(U)$ and $\mathrm{End}_{\mathcal{T}}(T)$ induced by $\varphi$ and $\psi$ respectively, is the same as the one given by the fact that $\mathcal{U}$ and $\mathcal{T}$ are $K$-categories. Since   $M(U,T)$ is an  $\mathrm{End}_{\mathcal{U}}(U)$-$\mathrm{End}_{\mathcal{T}}(T)$ bimodule, it is easy to show that
the structure of $K$-vector space on $M(U,T)$ induced by $\mathrm{End}_{\mathcal{U}}(U)$ via $\varphi$ is the same as the induced by $\mathrm{End}_{\mathcal{T}}(T)$ via $\psi$ (this is because $k1_{U}\otimes 1_{T}=1_{U}\otimes k1_{T}$).
\item [(b)] Suppose that $M_{T}\in \mathsf{mod}(\mathcal{U})$ and $\mathcal{U}$ is Hom-finite, then $M(U,T)$ is a finite dimensional $K$-vector space.
\end{enumerate}
\end{Remark}

\begin{lemma}\label{matrisartinal}Let $\mathcal{U}$ and $\mathcal{T}$ be Hom-finite $K$-categories and suppose  that $M\in \mathrm{Mod}(\mathcal{U}\otimes_{K}\mathcal{T}^{op})$ satisfies that $M_{T}\in \mathsf{mod}(\mathcal{U})$ $\forall$ $T\in \mathcal{T}$. Then
$\Gamma=\mathsf{End}_{\mathbf{\Lambda}}\left (\left[ \begin{smallmatrix}
T & 0 \\
M & U
\end{smallmatrix} \right]  \right)  := \left[ \begin{smallmatrix}
\mathsf{Hom}_{\mathcal{T}}(T,T) & 0 \\
M(U,T) & \mathsf{Hom}_{\mathcal{U}}(U,U)
\end{smallmatrix} \right]$ is an artin $K$-algebra. 
\end{lemma}
\begin{proof}
We have the ring morphism $\varphi:K\longrightarrow \mathrm{Hom}_{\mathcal{T}}(T,T)$ given by $\varphi(\lambda)=\lambda 1_{T}$  which made $\mathrm{End}_{\mathcal{T}}(T)$ into an artin $K$-algebra.\\
Similarly, $\psi:K\longrightarrow \mathrm{Hom}_{\mathcal{U}}(U,U)$ given by $\psi(\lambda)=\lambda 1_{U}$ is a ring morphism which made $\mathrm{End}_{\mathcal{U}}(U)$  into an artin $K$-algebra. \\
We note that the structure of vector spaces induced to the rings of endomorphisms is the same as the given by the definition that $\mathcal{U}$ and $\mathcal{T}$ are Hom-finite\\
By  \ref{findimM}(b), we have that $M(U,T)$ is finitely generated $K$-module. Then by \cite[Proposition 2.1]{AusBook} on page 72, we have that $\Gamma$ is an artin $K$-algebra via the morphism
$$\Phi:K \longrightarrow   \left[ \begin{smallmatrix}
\mathsf{Hom}_{\mathcal{T}}(T,T) & 0 \\
M(U,T) & \mathsf{Hom}_{\mathcal{U}}(U,U)
\end{smallmatrix} \right]$$
given by 
$$\Phi(\lambda)=\left[ \begin{smallmatrix}
\varphi(\lambda) & 0 \\
0 & \psi(\lambda)
\end{smallmatrix} \right]=\left[ \begin{smallmatrix}
\lambda 1_{T} & 0 \\
0 & \lambda 1_{U}
\end{smallmatrix} \right].$$
We note that
$$\lambda\cdot  \left[ \begin{smallmatrix}
 t & 0 \\
m & u
\end{smallmatrix} \right]:=\left[ \begin{smallmatrix}
\lambda 1_{T} & 0 \\
0 & \lambda 1_{U}
\end{smallmatrix} \right]\left[ \begin{smallmatrix}
t & 0 \\
m & u
\end{smallmatrix} \right]=\left[ \begin{smallmatrix}
(\lambda 1_{T})\circ  t & 0 \\
(\lambda 1_{U})\bullet m & (\lambda 1_{U})\circ u
\end{smallmatrix} \right]=\left[ \begin{smallmatrix}
1_{T}\circ (\lambda t) & 0 \\
(\lambda 1_{U})\bullet m & 1_{U}\circ (\lambda u)
\end{smallmatrix} \right]=\left[ \begin{smallmatrix}
\lambda t & 0 \\
\lambda m & \lambda  u
\end{smallmatrix} \right]$$
where  $t\in \mathrm{End}_{\mathcal{T}}(T)$ and $u\in \mathrm{End}_{\mathcal{U}}(U)$.\\
\end{proof}

\begin{lemma}
Let $\mathcal{U}$ and $\mathcal{T}$ be Hom-finite $K$-categories and suppose that $M\in \mathrm{Mod}(\mathcal{U}\otimes_{K}\mathcal{T}^{op})$ satisfies that $M_{T}\in \mathsf{mod}(\mathcal{U})$ for all $T\in \mathcal{T}$. Then
$\mathbf{\Lambda}=\left[ \begin{smallmatrix}
\mathcal{T} & 0 \\ M & \mathcal{U}
\end{smallmatrix}\right]$ is a Hom-finite $K$-category.
\end{lemma}
\begin{proof}
Let us consider
$$\mathsf{ Hom}_{\mathbf{\Lambda}}\left (\left[ \begin{smallmatrix}
T & 0 \\
M & U
\end{smallmatrix} \right] ,  \left[ \begin{smallmatrix}
T' & 0 \\
M & U'
\end{smallmatrix} \right]  \right)  := \left[ \begin{smallmatrix}
\mathsf{Hom}_{\mathcal{T}}(T,T') & 0 \\
M(U',T) & \mathsf{Hom}_{\mathcal{U}}(U,U')
\end{smallmatrix} \right],$$
and $A:=\mathsf{End}_{\mathbf{\Lambda}}\left (\left[ \begin{smallmatrix}
T' & 0 \\
M & U'
\end{smallmatrix} \right]  \right)$. 
We have that $\mathsf{ Hom}_{\mathbf{\Lambda}}\left (\left[ \begin{smallmatrix}
T & 0 \\
M & U
\end{smallmatrix} \right] ,  \left[ \begin{smallmatrix}
T' & 0 \\
M & U'
\end{smallmatrix} \right]  \right) $ is a lef $A$-module.
Therefore  $\mathsf{ Hom}_{\mathbf{\Lambda}}\left (\left[ \begin{smallmatrix}
T & 0 \\
M & U
\end{smallmatrix} \right] ,  \left[ \begin{smallmatrix}
T' & 0 \\
M & U'
\end{smallmatrix} \right]  \right) $ is a $K$-vector space as follows:\\
Let $\left[ \begin{smallmatrix}
 t & 0 \\
m & u
\end{smallmatrix} \right]\in \mathsf{ Hom}_{\mathbf{\Lambda}}\left (\left[ \begin{smallmatrix}
T & 0 \\
M & U
\end{smallmatrix} \right] ,  \left[ \begin{smallmatrix}
T' & 0 \\
M & U'
\end{smallmatrix} \right]  \right)$ and $\lambda\in K$, then

$$\lambda\cdot  \left[ \begin{smallmatrix}
 t & 0 \\
m & u
\end{smallmatrix} \right]:=\left[ \begin{smallmatrix}
 \lambda1_{T'} & 0 \\
0 & \lambda 1_{U'}
\end{smallmatrix} \right]\circ 
\left[ \begin{smallmatrix}
 t & 0 \\
m & u
\end{smallmatrix} \right]=\left[ \begin{smallmatrix}
 \lambda1_{T'}\circ t & 0 \\
(\lambda 1_{U'})\bullet m& (\lambda 1_{U'})\circ u
\end{smallmatrix} \right]=\left[ \begin{smallmatrix}
 1_{T'}\circ \lambda t & 0 \\
\lambda  m & 1_{U'}\circ \lambda u
\end{smallmatrix} \right]=\left[ \begin{smallmatrix}
 \lambda t & 0 \\
\lambda m & \lambda u
\end{smallmatrix} \right]$$
Since $\mathrm{Hom}_{\mathcal{U}}(U,U')$, $\mathrm{Hom}_{\mathcal{T}}(T,T')$,  and $M(U',T)$ are finite dimensional $K$-vector spaces we conclude that 
$$\mathsf{ Hom}_{\mathbf{\Lambda}}\left (\left[ \begin{smallmatrix}
T & 0 \\
M & U
\end{smallmatrix} \right] ,  \left[ \begin{smallmatrix}
T' & 0 \\
M & U'
\end{smallmatrix} \right]  \right)  := \left[ \begin{smallmatrix}
\mathsf{Hom}_{\mathcal{T}}(T,T') & 0 \\
M(U',T) & \mathsf{Hom}_{\mathcal{U}}(U,U')
\end{smallmatrix} \right]$$
is a finite dimensional $K$-vector space.
\end{proof}

\begin{proposition}\label{spliidempomat}
Assume that $\mathcal{T}$ and $\mathcal{U}$ are additive categories with splitting idempotents. Then $\mathbf{\Lambda}$ is an additive category with splitting idempotents.
\end{proposition}
\begin{proof}
Let $\left[ \begin{smallmatrix}
 t & 0 \\
 m & u
\end{smallmatrix} \right]:\left[ \begin{smallmatrix}
 T & 0 \\
M & U
\end{smallmatrix} \right]\longrightarrow \left[ \begin{smallmatrix}
 T & 0 \\
M & U
\end{smallmatrix} \right]$ be an idempoten morphism with $t\in \mathrm{Hom}_{\mathcal{T}}(T,T)$ and $u\in \mathrm{Hom}_{\mathcal{U}}(U,U)$ and $m\in M(U,T)$. Then
$$\left[ \begin{smallmatrix}
 t & 0 \\
 m & u
\end{smallmatrix} \right]=
\left[ \begin{smallmatrix}
 t & 0 \\
 m & u
\end{smallmatrix} \right]\left[ \begin{smallmatrix}
 t & 0 \\
 m & u
\end{smallmatrix} \right]=\left[ \begin{smallmatrix}
 t^{2} & 0 \\
m\bullet t +u\bullet m  & u^{2}
\end{smallmatrix} \right].$$
Then $m=m\bullet t+u\bullet m=M(1_{U}\otimes t^{op})(m)+M(u\otimes 1_{T})(m)$, $t^{2}=t$ and $u^{2}=u$.\\
Let $\mu:L\longrightarrow T$ be the kernel of $t:T\longrightarrow T$ and $\nu:K\longrightarrow U$ the kernel of $u:U\longrightarrow U$ (they exists because $\mathcal{U}$ and $\mathcal{T}$ are with split idempotents).
Then $0=t\mu$ and therefore
$0=M(1_{U}\otimes(t\mu)^{op})=M(1_{U}\otimes\mu^{op})\circ M(1_{U}\otimes t^{op})$. Hence, $0=\Big(M(1_{U}\otimes\mu^{op})\circ M(1_{U}\otimes t^{op})\Big)(m)=(m\bullet t)\bullet \mu$. Therefore
$$m\bullet \mu=(m\bullet t+u\bullet m)\bullet \mu=(m\bullet t)\bullet \mu+(u\bullet m)\bullet \mu=(u\bullet m)\bullet \mu\in M(U,L).$$
We claim that $\left[ \begin{smallmatrix}
 \mu & 0 \\
-m\bullet \mu & \nu
\end{smallmatrix} \right]:\left[ \begin{smallmatrix}
 L & 0 \\
M & K
\end{smallmatrix} \right]\longrightarrow \left[ \begin{smallmatrix}
 T & 0 \\
M & U
\end{smallmatrix} \right]$ is the kernel of $\left[ \begin{smallmatrix}
 t & 0 \\
m & u
\end{smallmatrix} \right]:\left[ \begin{smallmatrix}
 T & 0 \\
M & U
\end{smallmatrix} \right]\longrightarrow \left[ \begin{smallmatrix}
 T & 0 \\
M & U
\end{smallmatrix} \right]$.\\
Indeed,
\begin{enumerate}
\item [(a)] First we note that
\begin{align*}
\left[ \begin{smallmatrix}
 t & 0 \\
m & u
\end{smallmatrix} \right]
\left[ \begin{smallmatrix}
 \mu & 0 \\
-m\bullet \mu & \nu
\end{smallmatrix} \right]=\left[ \begin{smallmatrix}
 t\mu & 0 \\
m\bullet\mu +u\bullet (-m\bullet \mu)& u\nu
\end{smallmatrix} \right] & =\left[ \begin{smallmatrix}
 t\mu & 0 \\
m\bullet\mu - u\bullet (m\bullet \mu)& u\nu
\end{smallmatrix} \right]\\
& =  \left[ \begin{smallmatrix}
 t\mu & 0 \\
m\bullet\mu -(u\bullet m)\bullet \mu& u\nu
\end{smallmatrix} \right]\\
& = \left[ \begin{smallmatrix}
0 & 0 \\
0 & 0
\end{smallmatrix} \right]
\end{align*}

\item [(b)] Consider $\left[ \begin{smallmatrix}
 \alpha & 0 \\
n & \beta
\end{smallmatrix} \right]:\left[ \begin{smallmatrix}
T' & 0 \\
M & U'
\end{smallmatrix} \right]\longrightarrow \left[ \begin{smallmatrix}
 T & 0 \\
M & U
\end{smallmatrix} \right]$ with $n\in M(U,T')$ and $\alpha:T'\longrightarrow T$ and $\beta:U'\longrightarrow U$, such that $$
\left[ \begin{smallmatrix}
 0 & 0 \\
0 & 0
\end{smallmatrix} \right]=
\left[ \begin{smallmatrix}
 t & 0 \\
m & u
\end{smallmatrix} \right]\left[ \begin{smallmatrix}
 \alpha & 0 \\
n & \beta
\end{smallmatrix} \right]=\left[ \begin{smallmatrix}
 t \alpha & 0 \\
m\bullet \alpha+u\bullet n &  u\beta
\end{smallmatrix} \right]$$
Then $m\bullet \alpha=-u\bullet n\in M(U,T')$.\\
Consider $\nu:K\longrightarrow U$ then $\nu\bullet (m\bullet \alpha)=\nu\bullet (-u\bullet n)\in M(K,T')$.
We want a morphism
$\left[ \begin{smallmatrix}
\alpha' & 0 \\
m' & \beta'
\end{smallmatrix} \right]:\left[ \begin{smallmatrix}
T' & 0 \\
M & U'
\end{smallmatrix} \right]\longrightarrow \left[ \begin{smallmatrix}
 L & 0 \\
M & K
\end{smallmatrix} \right]$ with $m'\in M(K,T')$,  such that
$$\left[ \begin{smallmatrix}
\alpha & 0 \\
n & \beta
\end{smallmatrix} \right]=\left[ \begin{smallmatrix}
\mu & 0 \\
- (m\bullet \mu) & \nu
\end{smallmatrix} \right]\left[ \begin{smallmatrix}
\alpha' & 0 \\
m' & \beta'
\end{smallmatrix} \right]=\left[ \begin{smallmatrix}
\mu \alpha' & 0 \\
(- (m\bullet \mu))\bullet \alpha'+\nu\bullet m' & \nu\beta'
\end{smallmatrix} \right]$$
For this, consider $\nu':K'\longrightarrow U$ the kernel of $1_{U}-u$. 
Since idempotens split, we have that 
$U\simeq K'\oplus K$ with $\nu:K\longrightarrow U$ and $\nu':K'\longrightarrow U$ the natural inclusions. In particular, there exists $p:U\longrightarrow K$ and $p':U\longrightarrow K'$ such that
$$1_{U}=\nu p+\nu' p',\quad p\nu=1_{K},\quad p'\nu'=1_{K'}.$$
It can be seen that  $u=\nu' p'$.\\
From this we have that $n=\nu\bullet (p\bullet n)+\nu'\bullet (p'\bullet n).$
We set $m':=p\bullet n \in M(K,T')$.
Then we have that 
$$(- (m\bullet \mu))\bullet \alpha'=-(m\bullet (\mu \circ \alpha'))=-(m\bullet \alpha)= u\bullet n.$$
On the other hand
$$\nu\bullet m'=\nu\bullet(p\bullet n)=n-\nu'\bullet (p'\bullet n)=
n-(\nu'\circ p')\bullet=n-u\bullet n$$
Therefore 
$$\Big((- (m\bullet \mu))\bullet \alpha'\Big)+\nu\bullet m'=n$$
and $\alpha=\mu\alpha'$ and $\beta=\nu \beta'$. Proving that 
$\left[ \begin{smallmatrix}
\alpha' & 0 \\
m' & \beta'
\end{smallmatrix} \right]$ is the required morphism.

\end{enumerate}
Uniqueness.
Suppose that $\left[ \begin{smallmatrix}
\alpha'' & 0 \\
m'' & \beta''
\end{smallmatrix} \right]$ is such that 
$$\left[ \begin{smallmatrix}
\mu \alpha' & 0 \\
(- (m\bullet \mu))\bullet \alpha'+\nu\bullet m' & \nu\beta'
\end{smallmatrix} \right]=\left[ \begin{smallmatrix}
\mu \alpha'' & 0 \\
(- (m\bullet \mu))\bullet \alpha''+\nu\bullet m'' & \nu\beta''
\end{smallmatrix} \right]=\left[ \begin{smallmatrix}
\alpha & 0 \\
n & \beta
\end{smallmatrix} \right].$$
From this we have that $\alpha'=\alpha''$ and $\beta'=\beta''$  and then $\nu \bullet m'=\nu \bullet m''$. Composing with $p:U\longrightarrow K$ we have that
$$m'=1_{K}\bullet m'=(p\circ \nu)\bullet  m'=p\bullet (\nu\bullet m')=p\bullet (\nu \bullet m'')=(p\circ \nu )\bullet m''=1_{K}\bullet m''=m''.$$
Proving the uniqueness. Therefore $\mathbf{\Lambda}$ is a category with splitting idempotents.
\end{proof}

\begin{proposition}\label{homfinikrulsc}
Let $\mathcal{U}$ and $\mathcal{T}$ be Hom-finite $K$-categories which are Krull-Schmidt and $M\in \mathrm{Mod}(\mathcal{U}\otimes_{K}\mathcal{T}^{op})$ satisfies that $M_{T}\in \mathsf{mod}(\mathcal{U})$ for all $T\in \mathcal{T}$. 
\begin{enumerate}
\item [(a)] Then
$\mathbf{\Lambda}=\left[ \begin{smallmatrix}
\mathcal{T} & 0 \\ M & \mathcal{U}
\end{smallmatrix}\right]$ is a Hom-finite $K$-category and Krull-Schmidt.

\item [(b)] Then $\mathsf{mod}(\mathbf{\Lambda})$ is  Hom-finite $K$-category and Krull-Schmidt.

\end{enumerate}
\end{proposition}
\begin{proof}
\begin{enumerate}
\item [(a)]
By \ref{matrisartinal}, we have that 
$\Gamma=\mathsf{End}_{\mathbf{\Lambda}}\left (\left[ \begin{smallmatrix}
T & 0 \\
M & U
\end{smallmatrix} \right]  \right)  := \left[ \begin{smallmatrix}
\mathsf{Hom}_{\mathcal{T}}(T,T) & 0 \\
M(U,T) & \mathsf{Hom}_{\mathcal{U}}(U,U)
\end{smallmatrix} \right]$ is an artin $K$-algebra for every  $\left[ \begin{smallmatrix}
T & 0 \\
M & U
\end{smallmatrix} \right] \in \mathbf{\Lambda}$ and therefore semiperfect.\\
Since $\mathcal{U}$ and $\mathcal{T}$ are Krull-Schmidt, we have that  $\mathcal{U}$ and $\mathcal{T}$ are with splitting idempotents (see \cite[Corollary 4.4]{Krause}).
By \ref{spliidempomat} and \ref{fincopromat}, we have that $\mathbf{\Lambda}$ is an additive category with splitting idempotents.
Therefore $\mathbf{\Lambda}$ is Krull-Schmidt (see \cite[Corollary 4.4]{Krause}).

\item [(b)] It follows from item (a) and \cite[Proposition 2.4]{MOSS}.
\end{enumerate}
\end{proof}
Finally we have the following result that give us that $\mathbf{\Lambda}=\left[ \begin{smallmatrix}
\mathcal{T} & 0 \\
M & \mathcal{U}
\end{smallmatrix} \right]$ is a dualizing $K$-variety.

\begin{proposition}\label{matricesdualizan}
Suppose that $\mathcal{U}$ and $\mathcal{T}$ are dualizing $K$-varieties,  $M\in \mathrm{Mod}(\mathcal{U}\otimes_{K}\mathcal{T}^{op})$ satisfies that $M_{T}\in \mathsf{mod}(\mathcal{U})$ and  $M_{U}\in \mathsf{mod}(\mathcal{T}^{op})$ for all $T\in \mathcal{T}$ and $U\in \mathcal{U}^{op}$. Consider the contravariant functors 
$$\mathbb{T}^{\ast}\circ \overline{\textswab{F}}\circ \widehat{\Theta}^{\ast},\,\,\,\,\,\, \mathbb{D}_{\mathbf{\Lambda}}\circ \textswab{F}:\Big(\mathsf{mod}(\mathcal{T}), \mathbb{G}(\mathsf{mod}(\mathcal{U})\Big)\longrightarrow \mathsf{mod}(\mathbf{\Lambda}^{op}).$$
Then, there exists an isomorphism $\nu:\mathbb{T}^{\ast}\circ \overline{\textswab{F}}\circ \widehat{\Theta}^{\ast}\longrightarrow \mathbb{D}_{\mathbf{\Lambda}}\circ \textswab{F},$ such that the following diagram is commutative up to  the isomorphism $\nu$
$$\xymatrix{
\Big(\mathsf{mod}(\mathcal{T}),\mathbb{G}\mathsf{mod}(\mathcal{U})\Big)\ar[rrr]^{\textswab{F}|_{(\mathsf{mod}(\mathcal{T}),\mathbb{G}\mathsf{mod}(\mathcal{U}))}}\ar[d]_{\widehat{\Theta}|_{(\mathsf{mod}(\mathcal{T}),\mathbb{G}\mathsf{mod}(\mathcal{U}))}} & &  & \mathsf{mod}(\mathbf{\Lambda})\ar[d]^{(\mathbb{D}_{\mathbf{\Lambda}})|_{\mathsf{mod}(\mathbf{\Lambda})}}\\
\Big(\mathsf{mod}(\mathcal{U}^{op}),\overline{\mathbb{G}}\mathsf{mod}(\mathcal{T}^{op})\Big)\ar@{=>}[urrr]_{\nu}\ar[rrr]_(.6){(\mathbb{T}^{\ast}\circ \overline{\textswab{F}})_{(\mathsf{mod}(\mathcal{U}^{op}),\overline{\mathbb{G}}\mathsf{mod}(\mathcal{T}^{op}))}}& &  &   \mathsf{mod}(\mathbf{\Lambda}^{op}).}$$
In particular $\mathbf{\Lambda}$ is dualizing $K$-variety.
\end{proposition}
\begin{proof}
By \ref{cuadrocasicon},  we have the diagram and by the previous results, we can restric the diagram. Therefore the image of $(\mathbb{D}_{\mathbf{\Lambda}})|_{\mathsf{mod}(\mathbf{\Lambda})}$ is contained in $\mathsf{mod}(\mathbf{\Lambda}^{op})$, then we  have the functor
$$(\mathbb{D}_{\mathbf{\Lambda}})|_{\mathsf{mod}(\mathbf{\Lambda})}:\mathsf{mod}(\mathbf{\Lambda})\longrightarrow
\mathsf{mod}(\mathbf{\Lambda}^{op})$$
In the same way we have that
$$(\mathbb{D}_{\mathbf{\Lambda}^{op}})|_{\mathsf{mod}(\mathbf{\Lambda}^{op})}:\mathsf{mod}(\mathbf{\Lambda}^{op})\longrightarrow
\mathsf{mod}(\mathbf{\Lambda})$$
and since $\mathbb{D}_{\mathbf{\Lambda}^{op}}$ and $\mathbb{D}_{\mathbf{\Lambda}}$ are quasi inverse from each other, we have that the previos are quasi inverse. Therefore $\mathbf{\Lambda}$ is dualizing.

\end{proof}

\begin{proposition}\label{modtieneseq}
Suppose that $\mathcal{U}$ and $\mathcal{T}$ are dualizing $K$-varieties, $M\in \mathrm{Mod}(\mathcal{U}\otimes_{K}\mathcal{T}^{op})$ satisfies that $M_{T}\in \mathsf{mod}(\mathcal{U})$ and  $M_{U}\in \mathsf{mod}(\mathcal{T}^{op})$ for all $T\in \mathcal{T}$ and $U\in \mathcal{U}^{op}$. Then there are almost split sequences in $\mathrm{mod}(\mathbf{\Lambda})$.
\end{proposition}
\begin{proof}
If follows from \ref{matricesdualizan} and \cite[Theorem 7.1.3]{Reiten}.
\end{proof}

\section{Examples and Applications}
In this section we make some applications to splitting torsion pairs which are in relation with tilting theory (see \cite{Assem}) and path categories which are studied in \cite{RingelTame} and we give a generalization of the one-point extension algebra.\\
First, we recall the construction of the so called one-point extension algebra. Given a finite dimensional $K$-algebra $\Lambda:=KQ/I$. Let $i$ a source in $Q$ and $\overline{e}_{i}$ the corresponding idempotent in $\Lambda$. Since there are no nontrivial paths ending in $i$, we have $\overline{e}_{i}\Lambda \overline{e}_{i}\simeq K$ and  $\overline{e}_{i}\Lambda(1- \overline{e}_{i})=0$. If $Q'$ denote the quiver that we obtain by removing the vertix $i$ and $I'$ denote the relations in $I$ removing the ones which start in $i$, then
$(1-\overline{e}_{i})\Lambda (1-\overline{e}_{i})\simeq KQ'/I'$. So $\Lambda=KQ/I$ is obtained from $\Lambda':=KQ'/I'$ by adding one vertex $i$, together with arrows and relations starting in $i$. Then
$\Lambda:=\left[ \begin{smallmatrix}
K & 0 \\
(1-\overline{e}_{i})\Lambda \overline{e}_{i} & \Lambda'
\end{smallmatrix} \right].$ So $\Lambda$ is the one-point extension of $\Lambda'$.\\
Now, in order to give a generalization of the previous construction, we consider the following setting. Let $\mathcal{C}$ be a Krull-Schmidt category and $(\mathcal{U}, \mathcal{T})$ a pair of additive full subcategories of $\mathcal{C}$. It is said that $(\mathcal{U}, \mathcal{T})$ is a $\textbf{splitting torsion pair}$  if 
\begin{itemize}
\item[(i)] For all $X\in \mathrm{ind}(\mathcal{C})$, then either $X\in \mathcal{U}$ or $X\in \mathcal{T}$.
\item[(ii)] $\mathrm{Hom}_{\mathcal{C}}(X,-)|_{\mathcal{T}}=0$ for all $X\in\mathcal{U}$.
\end{itemize}
We get the following result that tell us that we can obtain a category as extension of two subcategories.

\begin{proposition}
Let $(\mathcal{U}, \mathcal{T})$ be  a splitting torsion pair. Then we have a equivalence of categories  
$$\mathcal{C} \cong \left[ \begin{matrix}
\mathcal{T} & 0 \\
\widehat{\mathbbm{Hom}} & \mathcal{U}
\end{matrix} \right].$$    
Here without danger to cause confusion $\widehat{\mathbbm{Hom}}$ denotes the restriction of 
$\widehat{\mathbbm{Hom}}:\mathcal{C}\otimes\mathcal{C}^{op}\rightarrow \mathbf{Ab}$  to the subcategory $\mathcal{U}\otimes\mathcal{T}^{op}$
of $\mathcal{C}\otimes\mathcal{C}^{op}$.
\end{proposition}
\begin{proof}
Let $X\in\mathcal{C}$. Then $X$ decomposes as $X=X_1\oplus X_2$ with $X_1\in\mathcal{T}$ and $X_2\in\mathcal{U}$. We define a functor $H:\mathcal{C} \rightarrow \left[ \begin{smallmatrix}
\mathcal{T} & 0 \\
\widehat{\mathbbm{Hom}} & \mathcal{U}
\end{smallmatrix} \right] $ in objects by $H(X)=\left[ \begin{smallmatrix}
 X_1 & 0 \\
\widehat{\mathbbm{Hom}} & X_2
\end{smallmatrix} \right] $.  Let $X,Y\in\mathcal{C}$ . Then $X=X_1\oplus X_2$ and $Y=Y_1\oplus Y_2$ with $X_1,Y_1\in\mathcal{T}$ and $X_2,Y_2\in\mathcal{U}$. Let
$f\in \mathrm{Hom}_{\mathcal{C}}(X,Y)$. Since $\mathrm{Hom}_{\mathcal{C}}(X_2,Y_1)=0$ and $\widehat{\mathbbm{Hom}}(X_1,Y_2) \cong \mathrm{Hom}_{\mathcal{C}}(X_2,Y_1)$, 
we have an isomorphism of abelian groups 
$$\mathrm{Hom}_{\mathcal{ C}}(X,Y)\cong \mathrm{Hom}_{\mathcal{C}}(X_1,Y_1)\oplus \widehat{\mathbbm{Hom}}(Y_2,X_1)\oplus \mathrm{Hom}_{\mathcal{C}}(X_2,Y_2)\oplus 0$$
 $$f\mapsto (f_{11},f_{21}, f_{22}, 0).$$ Thus, we get an isomorphism
\[
H: \mathrm{Hom}_{\mathcal{C}}(X,Y)\rightarrow \left[ \begin{smallmatrix}
 \mathrm{Hom}_{\mathcal{T}}(X_1,Y_1) & 0 \\
\widehat{\mathbbm{Hom}}(Y_2,X_1) & \mathrm{Hom}_{\mathcal{U}}(X_2,Y_2)
\end{smallmatrix} \right],  f\mapsto \left[ \begin{smallmatrix}
 f_{11} & 0 \\
f_{21} & f_{22}
\end{smallmatrix} \right].
\]
That is, we have an isomorphism $H:\mathrm{Hom}_{\mathcal{ C}}(X,Y)\rightarrow \left[ \begin{smallmatrix}
\mathcal{T} & 0 \\
\widehat{\mathbbm{Hom}} & \mathcal{U}
\end{smallmatrix} \right] \left (H(X),H(Y)\right)$.
\end{proof}
As an application of the last result, we consider $Q =(Q_1,Q_0)$ be a quiver. Recall that the \textbf{path category} $KQ$ is an additive category, with indecomposable objects the vertices, and given $ a,b\in Q_{0} $, the set of the maps
$\mathrm{Hom}_{KQ}(a,b)$  is given by the $K$-vector space with basis the set of all paths from $ a $ to $ b $. The composition of maps is induced from the usual composition of paths.
 Let $U=\lbrace x\in Q_0| x \text{ is a sink  }\rbrace $ and let $T=Q_0-U$, and consider $\mathcal{U}=\mathrm{add} \ U$ and $\mathcal{T}=\mathrm{add} \ T$. We consider the triangular matrix category $ \left[ \begin{smallmatrix}
\mathcal{T} & 0 \\ \mathsf{Hom}_{KQ}& \mathcal{U}
\end{smallmatrix}\right]$.
Then we have a equivalence of categories 
$$KQ\cong \left[ \begin{matrix}
\mathcal{T} & 0 \\
\widehat{\mathbbm{Hom}} & \mathcal{U}
\end{matrix} \right].$$    
\bigskip
As a concrete example, consider the following quiver  $Q=(Q_0,Q_1)$ with set of vertices $Q_0=\{u_i,t_i:i\in \mathbb Z\}$. As above, if   $U=\lbrace u_i: i\in\mathbb Z \rbrace $ and  $T=\lbrace t_i: i\in\mathbb Z \rbrace $, and we consider $\mathcal{U}=\mathrm{add} \ U$ and $\mathcal{T}=\mathrm{add} \ T$, then we have an equivalence of categories  $KQ\cong \left[ \begin{matrix}
\mathcal{T} & 0 \\
\widehat{\mathbbm{Hom}} & \mathcal{U}
\end{matrix} \right].$   

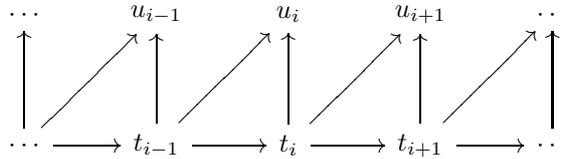
\[
\begin{diagram}
\node{\cdots}
 \node{u_{i-1}}
   \node{u_i}
    \node{u_{i+1}}
     \node{\cdots}\\
  \node{\cdots}\arrow{ne}{}\arrow{e}{}\arrow{n}{}
 \node{t_{i-1}}\arrow{ne}{}\arrow{e}{}\arrow{n}{}
   \node{t_i}\arrow{ne}{}\arrow{e}{}\arrow{n}{}
    \node{t_{i+1}}\arrow{ne}{}\arrow{e}{}\arrow{n}{}
     \node{\cdots}\arrow{n}{}
      \end{diagram}
\]

\begin{Remark}
It is worth to mention that the previous construction works when we consider quivers $Q$ with relations $I=\langle \rho_{i}\mid i\rangle$, since in this case the path category $KQ/I$ is and additive and Krull-Scmidt $K$-category.
\end{Remark}
As last application, we have the following result that give us a way to construct dualizing varieties from others.

\begin{proposition}\label{derivedfunc}
Let $\mathcal{C}$ be a dualizing $K$-variety with duality
$\mathbb{D}_{\mathcal{C}}:\mathrm{mod}(\mathcal{C})\longrightarrow \mathrm{mod}(\mathcal{C}^{op})$. Then the following statements hold.

\begin{enumerate}
\item [(a)] The triangular matrix category $\left[ \begin{matrix}
\mathcal{C} & 0 \\
\widehat{\mathbbm{{Hom}}} & \mathcal{C}
\end{matrix} \right] $ is dualizing.

\item [(b)] Let $\mathcal{C}$ be an abelian category with enough projectives and let $n\geq 1$. Suppose that $\mathrm{Ext}_{\mathcal{C}}^{n}(-,C)\in \mathrm{mod}(\mathcal{C}^{op})$ and $\mathrm{Ext}_{\mathcal{C}}^{n}(C,-)\in \mathrm{mod}(\mathcal{C})$ for every $C\in \mathcal{C}$. Then the triangular matrix category $\left[ \begin{matrix}
\mathcal{C} & 0 \\
\widehat{{\mathbbm{Ext}}^n} & \mathcal{C}
\end{matrix} \right] $ is dualizing.\\
Moreover, $\mathbb{D}_{\mathcal{C}^{op}}\overline{\mathbb{F}}(B)\simeq L^{n}(\mathbb{D}_{\mathcal{C}^{op}}(B))$  if $\mathbb{D}_{\mathcal{C}^{op}}(B)\in \mathrm{mod}(\mathcal{C}^{op})$ is right exact, where $L^{n}\mathbb{D}_{\mathcal{C}^{op}}(B)$ denotes the $n$-$th$ left derived functor of $\mathbb{D}_{\mathcal{C}^{op}}(B)$. Similarly 
$\mathbb{D}_{\mathcal{C}}\mathbb{F}(A)\simeq L^{n}(\mathbb{D}_{\mathcal{C}}(A))$ if $\mathbb{D}_{\mathcal{C}}(A)\in \mathrm{mod}(\mathcal{C})$ is right exact. 

\item [(c)]Suppose that $\mathcal{C}$ is an abelian category with enough projectives. Then the triangular matrix category $\left[ \begin{matrix}
\mathcal{C} & 0 \\
\widehat{{\mathbbm{Ext}}^1} & \mathcal{C}
\end{matrix} \right] $ is dualizing. Moreover $\mathbb{D}_{\mathcal{C}^{op}}\overline{\mathbb{F}}(B)\simeq L^{1}(\mathbb{D}_{\mathcal{C}^{op}}(B))$ if $\mathbb{D}_{\mathcal{C}^{op}}(B)\in \mathrm{mod}(\mathcal{C}^{op})$ is right exact and $\mathbb{D}_{\mathcal{C}}\mathbb{F}(A)\simeq L^{1}(\mathbb{D}_{\mathcal{C}}(A))$ if $\mathbb{D}_{\mathcal{C}}(A)\in \mathrm{mod}(\mathcal{C})$ is right exact.
\end{enumerate}
\end{proposition}
\begin{proof}
\begin{enumerate}
\item [(a)]  By \ref{funhomext} we have that $\widehat{\mathbbm{Hom}}_{C}=\widehat{\mathbbm{Hom}_{\mathcal{C}\otimes \mathcal{C}^{op}}}(-,C)\cong \mathrm{Hom}_{\mathcal C}(C,-)$ and $\widehat{\mathbbm{Hom}}_{C'}=\widehat{\mathbbm{Hom}_{\mathcal{C}\otimes \mathcal{C}^{op}}}(C',-)\cong \mathrm{Hom}_{\mathcal C}(-,C')$. The result follows from \ref{matricesdualizan}.

\item [(b)] By \ref{funhomext} we have that $\widehat{\mathbbm{Ext}_{C}^{n}}=\widehat{\mathbbm{Ext}_{\mathcal{C}\otimes \mathcal{C}^{op}}^{n}}(-,C)\cong \mathrm{Ext}_{\mathcal C}^{n}(C,-)$ and $\widehat{\mathbbm{Ext}_{C'}^{n}}=\widehat{\mathbbm{Ext}_{\mathcal{C}\otimes \mathcal{C}^{op}}^{n}}(C',-)\cong \mathrm{Ext}_{\mathcal C}^{n}(-,C')$. The result follows from \ref{matricesdualizan}.\\
On the other hand, by \ref{corcambia}(i), we have an isomorphism
$\mathbb{D}_{\mathcal{C}^{op}}\overline{\mathbb{F}}(B)(C)\cong
\mathrm{Hom}_{\mathrm{Mod}(\mathcal{C})}\Big(\mathrm{Ext}_{\mathcal{C}}^{n}(C,-),\mathbb{D}_{\mathcal{C}^{op}}(B)\Big)
\cong  L^{n}(\mathbb{D}_{\mathcal{C}^{op}}(B))$ (the last isomorphism is by \cite[Theorem 1.4]{Pesh}). The other isomorphism is analogous.

\item [(c)] Suppose that $\mathcal{C}$ is an abelian category with enough projectives. Since $\mathcal{C}$ is dualizing we have that  $\mathcal{C}$ has enough injectives. For every $C\in \mathcal{C}$ there exists an exact sequence $\xymatrix{0\ar[r] & K\ar[r]  & P\ar[r] & C\ar[r] & 0}$ with $P$ projective. Then we get the exact sequence in $\mathrm{Mod}(\mathcal{C})$ 
$$0\rightarrow \mathrm{Hom}_{\mathcal{C}}(C,-)\rightarrow
\mathrm{Hom}_{\mathcal{C}}(P,-)\rightarrow \mathrm{Hom}_{\mathcal{C}}(K,-)\rightarrow \mathrm{Ext}^{1}_{\mathcal{C}}(C,-)\rightarrow 0.$$
It follows that $\mathrm{Ext}^{1}_{\mathcal{C}}(C,-)\in \mathrm{mod}(\mathcal{C})$. Similarly, for $C\in \mathcal{C}$  get the an exact sequence in $\mathrm{Mod}(\mathcal{C}^{op})$ 
$$0\rightarrow \mathrm{Hom}_{\mathcal{C}}(-,C)\rightarrow
\mathrm{Hom}_{\mathcal{C}}(-,I)\rightarrow \mathrm{Hom}_{\mathcal{C}}(-,L)\rightarrow \mathrm{Ext}^{1}_{\mathcal{C}}(-,C)\rightarrow 0,$$ where $I$ is injective. Hence, the result follows from item (b).
\end{enumerate}
\end{proof}

\begin{corollary}
Let $A$ be an artin algebra and consider $\mathcal C=\mathrm{mod} (A)$. Then the triangular matrix categories $\left[ \begin{matrix}
\mathcal{C} & 0 \\
\widehat{\mathbbm{{Hom}}} & \mathcal{C}
\end{matrix} \right] $  and $\left[ \begin{matrix}
\mathcal{C} & 0 \\
\widehat{{\mathbbm{Ext}}^1} & \mathcal{C}
\end{matrix} \right] $ are dualizing.
\end{corollary}

\footnotesize

\vskip3mm \noindent Alicia Le\'on Galeana:\\ Facultad de Ciencias, Universidad  Aut\'onoma del Estado de M\'exico\\
T\'oluca, M\'exico.\\
{\tt alicialg@hotmail.com}

\vskip3mm \noindent Martin Ort\'iz Morales:\\ Facultad de Ciencias, Universidad  Aut\'onoma del Estado de M\'exico\\
T\'oluca, M\'exico.\\
{\tt mortizmo@uaemex.mx}

\vskip3mm \noindent Valente Santiago Vargas:\\ Departamento de Matem\'aticas, Facultad de Ciencias, Universidad Nacional Aut\'onoma de M\'exico\\
Circuito Exterior, Ciudad Universitaria,
C.P. 04510, M\'exico, D.F. MEXICO.\\ {\tt valente.santiago.v@gmail.com}

\end{document}